\documentclass[11pt]{article}
\usepackage[margin=2.5cm]{geometry}
\usepackage{pdflscape}
\usepackage{amsfonts}
\usepackage{amsmath}
\usepackage{amsthm}
\usepackage{amssymb}
\usepackage{amsmath,amscd}
\usepackage{graphicx}
\usepackage{enumerate}
\usepackage{subcaption}
\graphicspath{ {./Figure/} }
\usepackage[pdffitwindow=false,
colorlinks=true,linkcolor=black,urlcolor=black,citecolor=black]{hyperref}
\usepackage{authblk}

\newcommand*{\email}[1]{%
    \normalsize\href{mailto:#1}{#1}\par
    }

\newtheorem{theorem}{Theorem}[section]
\newtheorem{lemma}[theorem]{Lemma}
\theoremstyle{definition}

\newtheorem{proposition}[theorem]{Proposition}

\newtheorem{remark}[theorem]{Remark}
\numberwithin{equation}{section}

\newcommand{\spann}{\mathrm{span}}

\newcommand{\spin}{\mathrm{Spin}}

\newcommand{\trace}{\mathrm{tr}}

\newcommand{\diag}{\mathrm{diag}}

\begin{document}
\title {$\spin(7)$ metrics of cohomogeneity one with Aloff--Wallach spaces as principal orbits}
\author[]{Hanci Chi}
\affil[]{Department of Foundational Mathematics, Xi'an Jiaotong-Liverpool University\\
\email{hanci.chi@xjtlu.edu.cn}}

\date{\today}

\maketitle
\abstract
In this article, we construct two continuous 1-parameter family of non-compact $\spin(7)$ metrics with both chiralities, with the principal orbit an Aloff--Wallach space $N_{k,l}$ and the singular orbit $\mathbb{CP}^2$. For a generic $N_{k,l}$, metrics constructed are locally asymptotically conical (ALC). For $N_{1,1}$, we construct two continuous 1-parameter families with geometric transition from asymptotically conical (AC) metrics to ALC metrics.

\tableofcontents
\section{Introduction}
Metrics with $\spin(7)$ holonomy are interesting subjects in differential geometry and theoretical physics. The first example was constructed in \cite{bryant_metrics_1987} on the cone over Berger space $SO(5)/SO(3)$. The first complete example was constructed in \cite{bryant_construction_1989} and \cite{gibbons_einstein_1990} on the spinor bundle over $\mathbb{S}^4$. The first compact example was constructed in \cite{joyce_compact_1996}. In recent years, $\spin(7)$ metrics with an Aloff--Wallach space as principal orbit became an active field of research. It is expected that a cohomogeneity one space with an Aloff--Wallach space as its principal orbit admits a continuous 1-parameter family of $\spin(7)$ metrics with asymptotically locally conical (ALC) asymptotics: each metric's asymptotic limit is a product between a 7-dimensional Ricci-flat cone and a circle with fixed radius $l>0$. Moreover, with $l\to \infty$, the metric converges to an asymptotically conical (AC) one: its asymptotic limit is an 8-dimensional cone. In this way, one has a continuous 1-parameter family of metrics with non-maximal volume growth and its boundary is a metric with maximal volume growth.

An \emph{Aloff--Wallach space} is a homogeneous space $N_{k,l}:=SU(3)/U(1)_{k,l}$, where $U(1)_{k,l}$ is embedded in $SU(3)$ as $\diag\left(e^{k\sqrt{-1}t},e^{l\sqrt{-1}t},e^{-(k+l)\sqrt{-1}t}\right)$ with integers $k$ and $l$. Without loss of generality, we assume that $k$ and $l$ are coprime. Using outer automorphism and Weyl group of $SU(3)$ to normalize $(k,l)$, we assume $k\geq l\geq 0$ in this article. Due to the geometric differences, we call $N_{k,l}$ an \emph{exceptional Aloff--Wallach space} if $kl(k-l)=0$ and a \emph{generic Aloff--Wallach space} if otherwise. With an Aloff--Wallach space as the principal orbit, one can reduce the Einstein equations to a system of nonlinear second order ODEs. The $\spin(7)$ condition, being Ricci-flat, is reduced to a system of first order ODEs. The isolated explicit solution of cohomogeneity one $\spin(7)$ metric with $N_{1,0}$ as principal orbit and $\mathbb{CP}^2$ as singular orbit was given in \cite{gukov_m-theory_2002}. In \cite{reidegeld_exceptional_2011}, the author proved that singular orbit of a cohomogeneity one $\spin(7)$ metric with a fixed $N_{k,l}$ can only be either $\mathbb{S}^5$ or $\mathbb{CP}^2$. The singular orbit $\mathbb{S}^5$ appears exclusively for the case $N_{1,0}$.  Recently, some ALC $\spin(7)$ metrics with $N_{1,0}$ as principal orbit and $\mathbb{S}^5$ as singular orbit were constructed in \cite{foscolo_complete_2019}. The setting with $N_{1,0}$ as the principal orbit was further studied in \cite{lehmann_geometric_2020}, in which examples in \cite{gukov_m-theory_2002} and \cite{foscolo_complete_2019} were extended to two continuous families of ALC $\spin(7)$ metrics. The boundary of each family is an AC $\spin(7)$ metric. Explicit solutions with $N_{1,1}$ as principal orbit were given in \cite{cvetic_hyper-kahler_2001} and \cite{kanno_spin7_2002-1}. In \cite{bazaikin_new_2007} and \cite{bazaikin_noncompact_2008}, it is observed that formulas for $\spin(7)$ condition in \cite{cvetic_cohomogeneity_2002} and \cite{cvetic_new_2002} can be applied in a broader setting where the cohomogeneity one manifold is deformed from a HyperK\"ahler cone. Specifically, by using the 3-Sasakian structure on $N_{1,1}$, a 2-parameter family of $\spin(7)$ metrics was constructed on a complex line bundle over $SU(3)/T^2$ in \cite{bazaikin_new_2007} and a 2-parameter family of $\spin(7)$ metrics was constructed on a quaternionic line bundle over $\mathbb{CP}^2$ in \cite{bazaikin_noncompact_2008}. Explicit isolated solutions with a generic $N_{k,l}$ as principal orbit are included in \cite{kanno_spin7_2002-1} and\cite{cvetic_cohomogeneity_2002}. 

This article aims to prove the global existence of two continuous 1-parameter families of $\spin(7)$ metrics of cohomogeneity one with \emph{any} Aloff--Walach spaces $N_{k,l}$ as principal orbit. Past examples mentioned above are partially recovered in the construction. Specifically, a fixed $N_{k,l}$ can be viewed as fiber bundles over $\mathbb{CP}^2$ with lens spaces $L(1,|k+l|)$, $L(1,|k|)$ and $L(1,|l|)$ as fibers, respectively. Let $i\in\{k+l,k,l\}$ and $i\neq 0$, define $M_{k,l}^{(i)}$ to be the $\mathbb{R}^4/\mathbb{Z}_{|i|}$ bundle over $\mathbb{CP}^2$ with $N_{k,l}$ as principal orbit. We prove the following theorems.
\begin{theorem}
\label{thm: main1}
Let $k\geq l\geq 0$ be coprime. Let $i\in\{k+l,k\}$.
If $N_{k,l}$ is generic, there exists a continuous 1-parameter family of forward complete $\spin(7)$ metrics $\left\{\gamma^{(i)}_{(s_1,s_2)}\mid (s_1,s_2)\in\mathbb{S}^1, s_1,s_2>0\right\}$ on $M^{(i)}_{k,l}$. $\spin(7)$ metrics on $M^{(k+l)}_{k,l}$ and the ones on $M^{(k)}_{k,l}$ have opposite chirality. All $\gamma^{(k+l)}$'s and $\gamma^{(k)}$'s have ALC asymptotics with space at infinity as an $\mathbb{S}^1$ bundle over the $G_2$ cone over the nearly-K\"ahler $SU(3)/T^2$.
\end{theorem}
Note that for a generic $N_{k,l}$ as the principal orbit, metrics $\gamma_{(s_1,s_2)}^{(k+l)}$ on $M_{k,l}^{(k+l)}$ do not include the explicit ALC $\spin(7)$ metrics in \cite{cvetic_cohomogeneity_2002} and \cite{kanno_spin7_2002-1}. Hence it is likely to extend the continuous family of ALC metrics in Theorem \ref{thm: main1} to a larger family using other methods. Metrics $\gamma_{(s_1,s_2)}^{(k)}$ on $M_{k,l}^{(k)}$ are new to the best knowledge of the author. 

It is worth mentioning that our method for proving Theorem \ref{thm: main1} also applies for the case of exceptional $N_{1,0}$. Only ALC metrics are constructed in this way. For each $(s_1,s_2)$, the obtained metrics $\gamma_{(s_1,s_2)}^{(1)}$ with both chiralities on $M_{1,0}^{(1)}$ can be identified by a global $\mathbb{Z}_2$ symmetry. They are part of the results obtained in \cite{lehmann_geometric_2020}. They have smooth extension to the singular orbit $\mathbb{CP}^2$. For a more complete study on the $\spin(7)$ metrics with $N_{1,0}$ as principal orbit, please see \cite{foscolo_complete_2019} and \cite{lehmann_geometric_2020} for more details.

We also consider the case where the principal orbit is the exceptional $N_{1,1}$
\begin{theorem}
\label{thm: main2}
There exists a continuous 1-parameter family of forward complete $\spin(7)$ metrics $\left\{\gamma^{(i)}_{(s_1,s_2)}\mid (s_1,s_2)\in\mathbb{S}^1, s_1\geq s^{(i)}_*,s_2>0\right\}$ on $M^{(i)}_{1,1}$ with $s^{(2)}_*= -\frac{3}{\sqrt{10}}$ and $s^{(1)}_*= -\frac{1}{\sqrt{2}}$. $\spin(7)$ metrics on $M^{(2)}_{1,1}$ and the ones on $M^{(1)}_{1,1}$ have opposite chiralities. All $\gamma^{(2)}_{(s_1,s_2)}$'s and $\gamma^{(1)}_{(s_1,s_2)}$'s have an ALC asymptotics with space at infinity as an $\mathbb{S}^1$ bundle over the $G_2$ cone over the nearly-K\"ahler $SU(3)/T^2$ if $s_1>s_*$. If $s_1=s_*$, then the metrics have AC asymptotics with tangent cones at infinity with base homogeneous Einstein metrics on $N_{1,1}$. 
\end{theorem}

Metrics $\gamma_{(s_1,s_2)}^{(2)}$ on the orbifold $M_{1,1}^{(2)}$ are the ALC $\spin(7)$ metrics in \cite{cvetic_hyper-kahler_2001} and some of the solutions in \cite{bazaikin_new_2007}. It is worth mentioning that $M_{1,1}^{(1)}$ is in fact $T^*\mathbb{CP}^2$. Metrics $\gamma_{(s_1,s_2)}^{(1)}$ on $M_{1,1}^{(1)}$ are the ones that were conjectured in \cite{kanno_spin7_2002-2}. For most of the cases, we do not have metrics that shows geometric transition from AC $\spin(7)$ metrics to ALC $\spin(7)$ metrics like the ones that occur in \cite{cvetic_cohomogeneity_2002} and \cite{lehmann_geometric_2020}. However, if the principal orbit is $N_{1,1}$, we do have two continuous 1-parameter families of $\spin(7)$ metrics that have $AC$ metrics on their respective boundaries. In particular, the AC metric on $M_{1,1}^{(1)}$ is the Calabi HyperK\"ahler metric in \cite{calabi_metriques_1979}. In other words, the 1-parameter family of ALC $\spin(7)$ metrics on $M_{1,1}^{(1)}$ can be desingularized with Calabi's AC metric.

This article is structured as the following. In Section \ref{sec: Ricci-flat Equation}, we study the Ricci-flat system with an $N_{k,l}$ as the principal orbit and apply coordinate change that is similar to the one in \cite{chi_invariant_2019} and \cite{chi_einstein_2020}. The coordinate change transforms the singular orbit $\mathbb{CP}^2$, the AC asymptotics, and the ALC asymptotics to critical points of a polynomial system. The singular orbit $\mathbb{S}^5$, uniquely appears in the $N_{1,0}$ case, is blown up to the infinity in the new coordinate. Although being more complicated than the $\spin(7)$ subsystem, the Ricci-flat system has some important estimate that is not obvious in the $\spin(7)$ subsystem.

In Section \ref{sec: critical points}, we study critical points of the $\spin(7)$ subsystem and recover the locally existing result in \cite{reidegeld_exceptional_2011}. The cohomogeneity one $\spin(7)$ metrics are represented by integral curves that emanates from various critical points.

In Section \ref{sec: global}, we prove the global existence and the asymptotic limit by constructing compact invariant sets. The construction boils down to a classical algebraic geometry problem, which requires one to show non-negativity of a resultant polynomial. For readability, the very technical computation and formula are presented in the Appendix.

\textbf{Acknowledgements.} The author would like to thank NSFC for partial support under grants No. 11521101 and NO. 12071489. Thanks also go to Michael Baker, Jesse Madnick and McKenzie Wang for useful discussions.

\section{The $\spin(7)$ holonomy cohomogeneity one system}
\label{sec: Ricci-flat Equation}
In this section, we study the cohomogeneity one Ricci-flat system on $M^{(i)}_{k,l}$ derived in \cite{reidegeld_exceptional_2011}. We apply a coordinate change so that the system becomes a system of first order polynomial ODEs. The $\spin(7)$ condition, originally a first order subsystem, is then transformed to a set of algebraic equation in the new coordinate.

Consider an Aloff--Wallach space $N_{k,l}$, we fix the basis for $\mathfrak{su}(3)$ as the one in \cite{reidegeld_exceptional_2011}. Specifically, we work with the basis
\begin{equation}
\label{eqn: basis}
\begin{split}
&e_1:=\begin{bmatrix}
0&1&0\\
-1&0&0\\
0&0&0
\end{bmatrix},\quad e_2:=\sqrt{-1}\begin{bmatrix}
0&1&0\\
1&0&0\\
0&0&0
\end{bmatrix},\quad e_3:=\begin{bmatrix}
0&0&1\\
0&0&0\\
-1&0&0
\end{bmatrix},\quad e_4:=\sqrt{-1}\begin{bmatrix}
0&0&1\\
0&0&0\\
1&0&0
\end{bmatrix}\\
&e_5:=\begin{bmatrix}
0&0&0\\
0&0&1\\
0&-1&0
\end{bmatrix},\quad e_6:=\sqrt{-1}\begin{bmatrix}
0&0&0\\
0&0&1\\
0&1&0
\end{bmatrix},\\
&e_7:=\sqrt{-1}\begin{bmatrix}
2l+k&0&0\\
0&-(2k+l)&0\\
0&0&k-l
\end{bmatrix},\quad e_8:=\sqrt{-1}\begin{bmatrix}
k&0&0\\
0&l&0\\
0&0&-(k+l)
\end{bmatrix},\\
\end{split}
\end{equation}
where $e_8$ generates $\mathfrak{u}(1)_{k,l}$.
The isotropy representation $\mathfrak{su}(3)/\mathfrak{u}(1)_{k,l}$ is decomposed as 
$$
\mathfrak{su}(3)/\mathfrak{u}(1)_{k,l}=\mathfrak{i}\oplus \mathfrak{t}^{k-l}\oplus \mathfrak{t}^{2k+l}\oplus \mathfrak{t}^{k+2l},
$$
where
$$
\mathfrak{t}^{k-l}=\spann\{e_1,e_2\},\quad \mathfrak{t}^{2k+l}=\spann\{e_3,e_4\},\quad 
\mathfrak{t}^{k+2l}=\spann\{e_5,e_6\},\quad 
\mathfrak{i}=\spann\{e_7\}.
$$

Recall our convention of Aloff--Wallach spaces, we set $k$ and $l$ to be coprime and $k\geq l\geq 0$ without loss of generality. Under this setting, if $N_{k,l}$ is generic, i.e., $kl(k-l)\neq 0$, one can conclude that all $\mathfrak{t}$'s are non-trivial and have different weights. The isotropy representation then consists of 4 inequivalent irreducible summands. The group action of $SU(3)$ on \eqref{eqn: basis} generates a frame for an $SU(3)$-invariant Einstein metric $g_{N_{k,l}}$. The matrix representation of $g_{N_{k,l}}$ is hence 
\begin{equation}
\label{eqn: invariant metric}
\begin{bmatrix}
a^2&0&&&&&\\
0&a^2&&&&&&\\
&&b^2&0&&&\\
&&0&b^2&&&\\
&&&&c^2&0&\\
&&&&0&c^2&\\
&&&&&&f^2
\end{bmatrix}
\end{equation}
for some $a,b,c,f\neq 0$.
Consider the cohomogeneity one metric $g=dt^2+g_{N_{k,l}}(t)$, where each component in \eqref{eqn: invariant metric} is a function of $t$. Methods in \cite{eschenburg_initial_2000} can be applied to derive the cohomogeneity one Einstein system. Let $\Delta=k^2+kl+l^2$. As shown in \cite{reidegeld_exceptional_2011}, the cohomogeneity one Ricci-flat system is 
\begin{equation}
\label{eqn: original Einstein equation}
\begin{split}
\frac{\ddot{a}}{a}-\left(\frac{\dot{a}}{a}\right)^2&=-\left(2\frac{\dot{a}}{a}+2\frac{\dot{b}}{b}+2\frac{\dot{c}}{c}+\frac{\dot{f}}{f}\right)\frac{\dot{a}}{a}+\frac{6}{a^2}+\frac{a^2}{b^2c^2}-\frac{b^2}{a^2c^2}-\frac{c^2}{a^2b^2}-\frac{1}{2}\frac{(k+l)^2}{\Delta^2}\frac{f^2}{a^4}\\
\frac{\ddot{b}}{b}-\left(\frac{\dot{b}}{b}\right)^2&=-\left(2\frac{\dot{a}}{a}+2\frac{\dot{b}}{b}+2\frac{\dot{c}}{c}+\frac{\dot{f}}{f}\right)\frac{\dot{b}}{b}+\frac{6}{b^2}+\frac{b^2}{a^2c^2}-\frac{c^2}{a^2b^2}-\frac{a^2}{b^2c^2}-\frac{1}{2}\frac{l^2}{\Delta^2}\frac{f^2}{b^4}\\
\frac{\ddot{c}}{c}-\left(\frac{\dot{c}}{c}\right)^2&=-\left(2\frac{\dot{a}}{a}+2\frac{\dot{b}}{b}+2\frac{\dot{c}}{c}+\frac{\dot{f}}{f}\right)\frac{\dot{c}}{c}+\frac{6}{c^2}+\frac{c^2}{a^2b^2}-\frac{a^2}{b^2c^2}-\frac{b^2}{a^2c^2}-\frac{1}{2}\frac{k^2}{\Delta^2}\frac{f^2}{c^4}\\
\frac{\ddot{f}}{f}-\left(\frac{\dot{f}}{f}\right)^2&=-\left(2\frac{\dot{a}}{a}+2\frac{\dot{b}}{b}+2\frac{\dot{c}}{c}+\frac{\dot{f}}{f}\right)\frac{\dot{f}}{f}+\frac{1}{2}\frac{(k+l)^2}{\Delta^2}\frac{f^2}{a^4}+\frac{1}{2}\frac{l^2}{\Delta^2}\frac{f^2}{b^4}+\frac{1}{2}\frac{k^2}{\Delta^2}\frac{f^2}{c^4}
\end{split}
\end{equation}
with conservation law
\begin{equation}
\label{eqn: original conservation}
\begin{split}
&\left(2\frac{\dot{a}}{a}+2\frac{\dot{b}}{b}+2\frac{\dot{c}}{c}+\frac{\dot{f}}{f}\right)^2-2\left(\frac{\dot{a}}{a}\right)^2-2\left(\frac{\dot{b}}{b}\right)^2-2\left(\frac{\dot{c}}{c}\right)^2-\left(\frac{\dot{f}}{f}\right)^2\\
&=12\left(\frac{1}{a^2}+\frac{1}{b^2}+\frac{1}{c^2}\right)-2\left(\frac{a^2}{b^2c^2}+\frac{b^2}{a^2c^2}+\frac{c^2}{a^2b^2}\right)-\frac{1}{2}\frac{(k+l)^2}{\Delta^2}\frac{f^2}{a^4}-\frac{1}{2}\frac{l^2}{\Delta^2}\frac{f^2}{b^4}-\frac{1}{2}\frac{k^2}{\Delta^2}\frac{f^2}{c^4}.
\end{split}
\end{equation}
The conserved quantity \eqref{eqn: original conservation} is essentially equivalent to $g$ having zero scalar curvature. Note that the RHS of \eqref{eqn: original conservation} is the scalar curvature of $(N_{k,l}, g_{N_{k,l}}(t))$.

If $kl(k-l)=0$, then by our convention, we either have $(k,l)=(1,0)$ or $(k,l)=(1,1)$. In the first case, there are two equivalent isotropy summands $\mathfrak{t}^1$ in $\mathfrak{su}(3)/\mathfrak{u}_{1,0}$. In the latter case, the isotropy representation $\mathfrak{su}(3)/\mathfrak{u}_{1,1}$ consists of three trivial representations and two equivalent isotropy summands $\mathfrak{t}^3$. Therefore, $SU(3)$ invariant metrics on $N_{1,1}$ and $N_{1,0}$ are not necessarily diagonal.

 For $N_{1,0}$, the matrix representation of an $SU(3)$-invariant $g_{N_{1,0}}$ is 
\begin{equation}
\label{eqn: invariant metric N_1,0}
\begin{bmatrix}
a^2&0&&&A_1&A_2\\
0&a^2&&&-A_2&A_1\\
&&b^2&0&&&\\
&&0&b^2&&&\\
A_1&-A_2&&&c^2&0&\\
A_2&A_1&&&0&c^2&\\
&&&&&&f^2
\end{bmatrix}.
\end{equation}
From \cite{reidegeld_thesis} we learn that $N_{SU(3)}(U_{1,0}(1))=U(1)^2$. One can further reduce the number of functions from 6 to 5, since the residual action of $U(1)$ on $g_{N_{1,0}}$ changes $A_1$ and $A_2$ while leaving diagonal entries unchanged. The metric $g_{N_{1,0}}$ is $SU(3)\times U(1)$ invariant if and only if its diagonal. In other words, with the enhanced $SU(3)\times U(1)$ symmetry, the dynamic system is simply $\eqref{eqn: original Einstein equation}$ and $\eqref{eqn: original conservation}$ with $(k,l)=(1,0).$

For $N_{1,1}$, the matrix representation of an $SU(3)$-invariant $g_{N_{1,1}}$ is 
\begin{equation}
\label{eqn: invariant metric N_1,0}
\begin{bmatrix}
a_1^2&A_1&&&&&A_2\\
A_1&a_2^2&&&&&A_3\\
&&b^2&0&B_1&B_2&\\
&&0&b^2&-B_2&B_1&\\
&&B_1&-B_2&c^2&0&\\
&&B_2&B_1&0&c^2&\\
A_2&A_3&&&&&f^2
\end{bmatrix}.
\end{equation}
The normalizer $N_{SU(3)}(U_{1,1}(1))$ is isomorphic to $U(2)$ and $N_{SU(3)}(U_{1,1}(1))/ U_{1,1}(1)\cong SU(2)$. The residual action of $SU(2)$ acts as $SO(3)$ on the three trivial representations \cite{reidegeld_thesis}. Therefore, one can use the $SU(2)$ to partially diagonalize $g_{N_{1,1}}$ so that all $A_i$'s vanish and the number of functions is reduced from 10 to 7. A $SU(3)\times SU(2)$-invariant $g_{N_{1,1}}$ may still not be diagonal. Nevertheless, one can consider the case with diagonal $g_{N_{1,1}}$, as shown in $(4.10)$ in \cite{reidegeld_exceptional_2011}. In this article, we further impose the condition $a_1 = a_2$ and study \eqref{eqn: original Einstein equation} and \eqref{eqn: original conservation} with $(k,l)=(1,1)$. One can check that such a dynamic system is a subsystem of $(4.10)$ in \cite{reidegeld_exceptional_2011}. For the case where $a_1$ and $a_2$ are unequal, please see \cite{bazaikin_new_2007} for more details.

Singular orbits of cohomogeneity one space $M_{k,l}^{(k+l)}$, $M_{k,l}^{(l)}$ and $M_{k,l}^{(k)}$ are generated by $\mathfrak{t}^{2k+l}\oplus \mathfrak{t}^{k+2l}$, $\mathfrak{t}^{k-l}\oplus \mathfrak{t}^{k+2l}$ and $\mathfrak{t}^{k-l}\oplus \mathfrak{t}^{2k+l}$, respectively. Hence according to \cite{eschenburg_initial_2000} for the case $M_{k,l}^{(1)}$ and the power series in \cite{reidegeld_exceptional_2011} for the rest of the orbifolds, the initial conditions for theses three types of $\mathbb{R}^4/\mathbb{Z}_{|i|}$ bundles, are respectively given by 
\begin{equation}
\label{eqn: initial condtion}
\begin{split}
& \lim_{t\to 0}(a,b,c,f,\dot{a},\dot{b},\dot{c},\dot{f})=\left(0,a_0,a_0,0,1,0,0,\frac{2\Delta}{k+l}\right),\\
& \lim_{t\to 0}(a,b,c,f,\dot{a},\dot{b},\dot{c},\dot{f})=\left(b_0,0,b_0,0,0,1,0,\frac{2\Delta}{l}\right),\\
& \lim_{t\to 0}(a,b,c,f,\dot{a},\dot{b},\dot{c},\dot{f})=\left(c_0,c_0,0,0,0,0,1,\frac{2\Delta}{k}\right).
\end{split}
\end{equation}
For exceptional Aloff--Wallach spaces, recent development in \cite{verdiani_smoothness_2020} can also be applied to derive the initial condition even with equivalent isotropy summands.

The $\spin(7)$ condition is derived in \cite{reidegeld_exceptional_2011} as the following
\begin{equation}
\label{eqn: original spin(7) equation}
\begin{split}
\frac{\dot{a}}{a}&=\frac{b}{ac}+\frac{c}{ab}-\frac{a}{bc}-\frac{k+l}{2\Delta}\frac{f}{a^2}\\
\frac{\dot{b}}{b}&=\frac{c}{ab}+\frac{a}{bc}-\frac{b}{ac}+\frac{l}{2\Delta}\frac{f}{b^2}\\
\frac{\dot{c}}{c}&=\frac{a}{bc}+\frac{b}{ac}-\frac{c}{ab}+\frac{k}{2\Delta}\frac{f}{c^2}\\
\frac{\dot{f}}{f}&=\frac{k+l}{2\Delta}\frac{f}{a^2}-\frac{l}{2\Delta}\frac{f}{b^2}-\frac{k}{2\Delta}\frac{f}{c^2}
\end{split}.
\end{equation}
Change the sign of $f$ in \eqref{eqn: original spin(7) equation}. We then obtain the $\spin(7)$ condition with the opposite chirality:
\begin{equation}
\label{eqn: original opposite spin(7) equation}
\begin{split}
\frac{\dot{a}}{a}&=\frac{b}{ac}+\frac{c}{ab}-\frac{a}{bc}+\frac{k+l}{2\Delta}\frac{f}{a^2}\\
\frac{\dot{b}}{b}&=\frac{c}{ab}+\frac{a}{bc}-\frac{b}{ac}-\frac{l}{2\Delta}\frac{f}{b^2}\\
\frac{\dot{c}}{c}&=\frac{a}{bc}+\frac{b}{ac}-\frac{c}{ab}-\frac{k}{2\Delta}\frac{f}{c^2}\\
\frac{\dot{f}}{f}&=-\frac{k+l}{2\Delta}\frac{f}{a^2}+\frac{l}{2\Delta}\frac{f}{b^2}+\frac{k}{2\Delta}\frac{f}{c^2}
\end{split}.
\end{equation}

Although we mainly consider the construction of $\spin(7)$ metrics in this article, we start with the Ricci-flat system. As shown in the following, with a coordinate change, some important estimates can be quickly derived from the new Ricci-flat system while it is not obvious in the new $\spin(7)$ subsystem.

It is clear that 
$$L:=
\begin{bmatrix}
\frac{\dot{a}}{a}&0&&&&&\\
0&\frac{\dot{a}}{a}&&&&&&\\
&&\frac{\dot{b}}{b}&0&&&\\
&&0&\frac{\dot{b}}{b}&&&\\
&&&&\frac{\dot{c}}{c}&0&\\
&&&&0&\frac{\dot{c}}{c}&\\
&&&&&&\frac{\dot{f}}{f}\\
\end{bmatrix}
$$
is the second fundamental form of $N_{k,l}$ in $M^{(i)}_{k,l}$ at time $t$.
The quantity $\trace{L}$ in \eqref{eqn: original Einstein equation} is hence the mean curvature. In many works on the  construction of cohomogeneity one Einstein metrics, the coordinate change $d\eta=\trace{L} dt$ can help to simplify the original Einstein system\cite{dancer_non-kahler_2009}\cite{buzano_family_2015}\cite{chi_cohomogeneity_2019}. This case is no exception. Define functions
\begin{equation}
\label{eqn: Xi and Zi}
\begin{split}
&X_1=\frac{\frac{\dot{a}}{a}}{\trace{L}},\quad X_2=\frac{\frac{\dot{b}}{b}}{\trace{L}},\quad X_3=\frac{\frac{\dot{c}}{c}}{\trace{L}},\quad X_4=\frac{\frac{\dot{f}}{f}}{\trace{L}},\\
&Z_1=\frac{\frac{a}{bc}}{\trace{L}},\quad 
Z_2=\frac{\frac{b}{ac}}{\trace{L}},\quad 
Z_3=\frac{\frac{c}{ab}}{\trace{L}},\quad 
Z_4=f\trace{L}.
\end{split}
\end{equation}
And define functions
\begin{equation}
\begin{split}
&\mathcal{G}=2X_1^2+2X_2^2+2X_3^2+X_4^2\\
&\mathcal{R}_1=6Z_2Z_3+Z_1^2-Z_2^2-Z_3^2-\frac{1}{2}\frac{(k+l)^2}{\Delta^2}Z_2^2Z_3^2Z_4^2\\
&\mathcal{R}_2=6Z_1Z_3+Z_2^2-Z_3^2-Z_1^2-\frac{1}{2}\frac{l^2}{\Delta^2}Z_1^2Z_3^2Z_4^2\\
&\mathcal{R}_3=6Z_1Z_2+Z_3^2-Z_1^2-Z_2^2-\frac{1}{2}\frac{k^2}{\Delta^2}Z_1^2Z_2^2Z_4^2\\
&\mathcal{R}_4=\frac{1}{2}\frac{(k+l)^2}{\Delta^2}Z_2^2Z_3^2Z_4^2+\frac{1}{2}\frac{l^2}{\Delta^2}Z_1^2Z_3^2Z_4^2+\frac{1}{2}\frac{k^2}{\Delta^2}Z_1^2Z_2^2Z_4^2\\
&\mathcal{R}_s=2\mathcal{R}_1+2\mathcal{R}_2+2\mathcal{R}_3+\mathcal{R}_4
\end{split}
\end{equation}
Use $'$ to denote the derivative with respect to $\eta$. \eqref{eqn: original Einstein equation} is transformed to 
\begin{equation}
\label{eqn: new Einstein equation}
\begin{bmatrix}
X_1\\
X_2\\
X_3\\
X_4\\
Z_1\\
Z_2\\
Z_3\\
Z_4
\end{bmatrix}'
=
V(X_i,Z_i)\colon=\begin{bmatrix}
X_1(\mathcal{G}-1)+\mathcal{R}_1\\
X_2(\mathcal{G}-1)+\mathcal{R}_2\\
X_3(\mathcal{G}-1)+\mathcal{R}_3\\
X_4(\mathcal{G}-1)+\mathcal{R}_4\\
Z_1(\mathcal{G}+X_1-X_2-X_3)\\
Z_2(\mathcal{G}+X_2-X_3-X_1)\\
Z_3(\mathcal{G}+X_3-X_1-X_2)\\
Z_4(-\mathcal{G}+X_4)\\
\end{bmatrix}
\end{equation}
The conservation law \eqref{eqn: original conservation} becomes
\begin{equation}
\label{eqn: new conservation}
\mathcal{G}-1+\mathcal{R}_s=0.
\end{equation}
Note that $\left(\frac{1}{\trace(L)}\right)'=\frac{1}{\trace(L)}\mathcal{G}$. Therefore, $\frac{1}{\trace(L)}$ can be treated as function of $\eta$ by
$$\frac{1}{\trace(L)}=\exp\left(\int_{\eta^*}^{\eta}\mathcal{G}d\tilde{\eta}+C\right).$$
To recover the original coordinate, we simply compute
$$
t=\int_{\eta^*}^\eta \frac{1}{\trace(L)} d\tilde{\eta}=\int_{\eta^*}^\eta \exp\left(\int_{\eta^{**}}^{\eta^*}\mathcal{G}d\tilde{\tilde{\eta}}+C\right) d\tilde{\eta}+t_0
$$
and 
$$
a=\frac{1}{\trace(L)}\frac{1}{\sqrt{Z_2Z_3}},\quad b=\frac{1}{\trace(L)}\frac{1}{\sqrt{Z_1Z_3}},\quad  c=\frac{1}{\trace(L)}\frac{1}{\sqrt{Z_1Z_2}},\quad f=\frac{1}{\trace(L)}Z_4
$$

From the definition of $X_i$'s in \eqref{eqn: Xi and Zi}, one expects $2X_1+2X_2+2X_3+X_4=1$ is preserved by the new dynamic system. Indeed, since 
\begin{equation}
\label{eqn: H=1}
\begin{split}
(2X_1+2X_2+2X_3+X_4)'&=(2X_1+2X_2+2X_3+X_4)(\mathcal{G}-1)+\mathcal{R}_s\\
&=(2X_1+2X_2+2X_3+X_4)(\mathcal{G}-1)+1-\mathcal{G} \quad \text{by \eqref{eqn: new conservation}}\\
&=(2X_1+2X_2+2X_3+X_4-1)(\mathcal{G}-1)
\end{split},
\end{equation}
it is clear that $$\mathcal{H}:=\{2X_1+2X_2+2X_3+X_4=1\}$$ is invariant. It is worth mentioning that in many works on constructing cohomogeneity one steady Ricci solitons\cite{dancer_non-kahler_2009}\cite{buzano_family_2015}\cite{wink_cohomogeneity_2017}, coordinate change $d\eta= (-\dot{u}+\trace{L})dt$ is applied, where $u$ is the potential function. For our case, one obtains the same polynomial system as \eqref{eqn: new Einstein equation} with
$$
\mathcal{G}-1+\mathcal{R}_s\leq 0,
$$
where the equality is needed for \eqref{eqn: H=1} to hold. From this perspective, one can treat the invariance of $\mathcal{H}$ as the outcome of setting the potential of cohomogeneity one steady Ricci solitons to be a constant. From \eqref{eqn: new Einstein equation}, it is clear that we can assume $Z_i$'s be non-negative without loss of generality. In fact, the set $\{Z_1,Z_2,Z_3,Z_4\geq 0\}$ is invariant. A straightforward observation also gives the following proposition.
\begin{proposition}
\label{prop: X4 not negative}
The set $\{X_4\geq 0\}$ is invariant.
\end{proposition}
\begin{proof}
It is clear that
\begin{equation}
\left.\langle\nabla X_4, V\rangle\right|_{X_4=0}=\mathcal{R}_4=\frac{1}{2}\frac{(k+l)^2}{\Delta^2}Z_2^2Z_3^2Z_4^2+\frac{1}{2}\frac{l^2}{\Delta^2}Z_1^2Z_3^2Z_4^2+\frac{1}{2}\frac{k^2}{\Delta^2}Z_1^2Z_2^2Z_4^2\geq 0.
\end{equation}
If non-transverse crossings emerge on an integral curve, then in addition to $X_4=0$ at the crossing point, either $Z_i=Z_j=0$ for distinct $i,j\in\{1,2,3\}$ or $Z_4=0$ at that point. Then such an integral curve must lie in either the invariant set $\{Z_i=Z_j=0\}$ or $\{Z_4=0\}$ and the condition $X_4=0$ is held along the integral curve. Hence we exclude the possibility of non-transverse crossings. The proof is complete.
\end{proof}
Therefore, cohomogeneity one Ricci-flat metrics with an Aloff--Wallach space $N_{k,l}$ as the principal orbit is represented by an integral curve to \eqref{eqn: new Einstein equation} on the following subset of $\mathbb{R}^8$:
\begin{equation}
\label{eqn: new conservation set}
\begin{split}
&\mathcal{C}_{RF}\\
&:=\left\{2X_1+2X_2+2X_3+X_4=1,\quad \mathcal{G}-1+\mathcal{R}_s=0,\quad X_4\geq 0,\quad  Z_1,Z_2,Z_3,Z_4\geq 0\right\},
\end{split}
\end{equation}
 a $6$-dimensional algebraic surface with boundary. By Lemma 5.1 in \cite{buzano_family_2015}, we know that $\lim\limits_{\eta\to \infty} t=\infty$. Therefore, if an integral curve is defined on $\mathbb{R}$, then the Ricci-flat metric represented is forward complete.

We now consider the $\spin(7)$ condition \eqref{eqn: original spin(7) equation} and \eqref{eqn: original opposite spin(7) equation} in the new coordinate.
The $\spin(7)$ condition form an invariant subset of $\mathcal{C}_{RF}$ with a lower dimension. Specifically, \eqref{eqn: original spin(7) equation} turns to 
\begin{equation}
\label{eqn: new spin(7) equation}
\begin{split}
F_1&:=X_1+Z_1-Z_2-Z_3+\frac{k+l}{2\Delta}Z_2Z_3Z_4=0\\
F_2&:=X_2+Z_2-Z_3-Z_1-\frac{l}{2\Delta}Z_1Z_3Z_4=0\\
F_3&:=X_3+Z_3-Z_1-Z_2-\frac{k}{2\Delta}Z_1Z_2Z_4=0\\
F_4&:=X_4-\frac{k+l}{2\Delta}Z_2Z_3Z_4+\frac{l}{2\Delta}Z_1Z_3Z_4+\frac{k}{2\Delta}Z_1Z_2Z_4=0
\end{split}
\end{equation}
It is known that $\spin(7)$ metrics are Ricci-flat. Hence \eqref{eqn: original spin(7) equation} is a first order subsystem of \eqref{eqn: original Einstein equation}. This is can also be shown using the new coordinates. Define $\mathcal{C}^+_{\spin(7)}:=\mathcal{C}_{RF}\cap \left(\bigcap\limits_{i=1}^4\left\{F_i= 0\right\}\right)$. We claim the following.
\begin{proposition}
\label{prop: invariatn spin(7)+}
$\mathcal{C}^+_{\spin(7)}$ is invariant.
\end{proposition}
\begin{proof}
It is well-known that $\spin(7)$ metrics are Ricci-flat\cite{bonan_sur_1966}. With the coordinate change applied, the first order condition is transformed to algebraic equations in \eqref{eqn: new spin(7) equation}. Therefore, $\mathcal{C}^+_{\spin(7)}$ is an invariant set in $\mathcal{C}_{RF}$.

The statement can also be proven directly using \eqref{eqn: new Einstein equation},\eqref{eqn: new conservation} and definition of $\mathcal{C}_{RF}$.
Using the relation $2X_1+2X_2+2X_3+X_4=1$ in $\mathcal{C}_{RF}$, we have 
\begin{equation}
\begin{split}
\langle\nabla F_1, V\rangle&=F_1(\mathcal{G}-1)+\frac{k+l}{\Delta}Z_2Z_3Z_4(F_2+F_3+F_4)\\
&\quad +Z_1(3F_1+F_2+F_3+F_4)-Z_2(F_1+3F_2+F_3+F_4)-Z_3(F_1+F_2+3F_3+F_4)\\
\langle\nabla F_2, V\rangle&=F_2(\mathcal{G}-1)-\frac{l}{\Delta}Z_1Z_3Z_4(F_1+F_3+F_4)\\
&\quad -Z_1(3F_1+F_2+F_3+F_4)+Z_2(F_1+3F_2+F_3+F_4)-Z_3(F_1+F_2+3F_3+F_4)\\
\langle\nabla F_3, V\rangle&=F_3(\mathcal{G}-1)-\frac{k}{\Delta}Z_1Z_2Z_4(F_1+F_2+F_4)\\
&\quad -Z_1(3F_1+F_2+F_3+F_4)-Z_2(F_1+3F_2+F_3+F_4)+Z_3(F_1+F_2+3F_3+F_4)\\
\langle\nabla F_4, V\rangle&=F_4(\mathcal{G}-1)\\
&\quad -\frac{k+l}{\Delta}Z_2Z_3Z_4(F_2+F_3+F_4)+\frac{l}{\Delta}Z_1Z_3Z_4(F_1+F_3+F_4)+\frac{k}{\Delta}Z_1Z_2Z_4(F_1+F_2+F_4)
\end{split}.
\end{equation}
Hence the statement is proven.
\end{proof}

Replace $X_i$'s in \eqref{eqn: new conservation}, we obtain
$$
\left(2(Z_1+Z_2+Z_3)-\frac{k+l}{2\Delta}Z_2Z_3Z_4+\frac{l}{2\Delta}Z_1Z_3Z_4+\frac{k}{2\Delta}Z_1Z_2Z_4\right)^2=1.
$$
On the other hand, on $\mathcal{C}^+_{\spin(7)}$, we have 
\begin{equation}
\label{eqn: conservation Z spin(7)}
0=2F_1+2F_2+2F_3+F_4=1-2(Z_1+Z_2+Z_3)+\frac{k+l}{2\Delta}Z_2Z_3Z_4-\frac{l}{2\Delta}Z_1Z_3Z_4-\frac{k}{2\Delta}Z_1Z_2Z_4.
\end{equation}
Therefore, $\mathcal{C}^+_{\spin(7)}$ can also be expressed as 
the intersection
\begin{equation}
\begin{split}
\mathcal{C}^+_{\spin(7)}=&\mathcal{C}_{RF}\cap \left(\bigcap\limits_{i=1}^4\left\{F_i= 0\right\}\right)\\
&\cap \left\{2(Z_1+Z_2+Z_3)-\frac{k+l}{2\Delta}Z_2Z_3Z_4+\frac{l}{2\Delta}Z_1Z_3Z_4+\frac{k}{2\Delta}Z_1Z_2Z_4=1\right\}
\end{split}
\end{equation}
Cohomogeneity one $\spin(7)$ metrics with an Aloff--Wallach space $N_{k,l}$ as the principal orbit is represented by an integral curve to \eqref{eqn: new Einstein equation} restricted on $\mathcal{C}^+_{\spin(7)}\subset \mathcal{C}_{RF}$, a 3-dimensional algebraic surface in $\mathbb{R}^8$ with boundaries. In other words, the $\spin(7)$ metrics can be represented by integral curves to the vector field $\tilde{V}$ that consists of the last four entries of $V(X_i,Z_i)$, where $X_i$'s are polynomials defined as in \eqref{eqn: new spin(7) equation}. Such a dynamic system has a conservation law given by \eqref{eqn: conservation Z spin(7)}.

For the $\spin(7)$ condition with the opposite chirality, \eqref{eqn: original opposite spin(7) equation} becomes
\begin{equation}
\label{eqn: new opposite spin(7) equation}
\begin{split}
&H_1:=X_1+Z_1-Z_2-Z_3-\frac{k+l}{2\Delta}Z_2Z_3Z_4=0\\
&H_2:=X_2+Z_2-Z_3-Z_1+\frac{l}{2\Delta}Z_1Z_3Z_4=0\\
&H_3:=X_3+Z_3-Z_1-Z_2+\frac{k}{2\Delta}Z_1Z_2Z_4=0\\
&H_4:=X_4+\frac{k+l}{2\Delta}Z_2Z_3Z_4-\frac{l}{2\Delta}Z_1Z_3Z_4-\frac{k}{2\Delta}Z_1Z_2Z_4=0
\end{split},
\end{equation}
which is also a subsystem of \eqref{eqn: new Einstein equation}. Define $\mathcal{C}^-_{\spin(7)}$ to be
\begin{equation}
\label{eqn: new opposite spin(7) conservation}
\begin{split}
\mathcal{C}^-_{\spin(7)}=&\mathcal{C}_{RF}\cap\left(\bigcap\limits_{i=1}^4\left\{H_i= 0\right\}\right)\\
&\cap \left\{2(Z_1+Z_2+Z_3)+\frac{k+l}{2\Delta}Z_2Z_3Z_4-\frac{l}{2\Delta}Z_1Z_3Z_4-\frac{k}{2\Delta}Z_1Z_2Z_4=1\right\}
\end{split}.
\end{equation}
With the similar computation as the one in Proposition \ref{prop: invariatn spin(7)+}, we can also show that $\mathcal{C}_{\spin(7)}^-$ is invariant.

\begin{remark}
If we consider the $\spin(7)$ system with vector field $\tilde{V}$ on the 3-dimensional algebraic surface \eqref{eqn: conservation Z spin(7)} in $\mathbb{R}^4$, then the computation to prove Proposition \ref{prop: X4 not negative} becomes much more challenging. As $\mathcal{C}^\pm_{\spin(7)}$ are invariant subsets of $\mathcal{C}_{RF}$, the estimate $X_4\geq 0$ can be carried over to $\mathcal{C}_{\spin(7)}^\pm$.
\end{remark}

\begin{proposition}
\label{prop: X_4 and Zi}
On $\mathcal{C}^\pm_{\spin(7)}$, we have 
$X_4=2(Z_1+Z_2+Z_3)-1$.
\end{proposition}
\begin{proof}
Summing all equations in \eqref{eqn: new spin(7) equation}, we obtain
\begin{equation}
\label{eqn: X4 to Zi}
\frac{1-X_4}{2}+X_4=
X_1+X_2+X_3+X_4=Z_1+Z_2+Z_3
\end{equation}
Hence $X_4=2(Z_1+Z_2+Z_3)-1$ in $\mathcal{C}^+_{\spin(7)}$. Similar argument can prove the statement for $\mathcal{C}^-_{\spin(7)}$.
\end{proof}

There exists an invariant subset that lies in both $\mathcal{C}^+_{\spin(7)}$ and $\mathcal{C}^-_{\spin(7)}$. Define 
\begin{equation}
\begin{split}
\mathcal{C}_{G_2}=&\mathcal{C}^+_{\spin(7)}\cap \mathcal{C}^-_{\spin(7)}
\end{split}.
\end{equation}
Equivalently, we have $\mathcal{C}_{G_2}=\mathcal{C}^\pm_{\spin(7)} \cap \{X_4=0\}\cap \{Z_4=0\}$.
It is clear that $\mathcal{C}_{G_2}$ is a 2-dimensional invariant set. System \eqref{eqn: new Einstein equation} restricted on $\mathcal{C}_{G_2}$ is essentially the system for cohomogeneity one $G_2$ metric with $\mathbb{CP}^2$ as singular orbit and $SU(3)/T^2$ as principal orbit. For a forward complete ALC Ricci-flat metric, components $a,b$ and $c$ in \eqref{eqn: invariant metric} increase linearly and $f$ converges to a constant as $t\to \infty$. The integral curve that represents such a metric converges to the invariant set $\mathcal{C}_{RF}\cap \{X_4=0\}\cap \{Z_4=0\}$ as $\eta\to\infty$. Therefore, one can also think of $\mathcal{C}_{G_2}$ as a subset of the ``space of ALC asymptotics''. Such an invariant subset also occurs in the dynamic system in \cite{chi_einstein_2020}. By \cite{cleyton_cohomogeneity-one_2002}, all integral curves on $\mathcal{C}_{G_2}$ are explicitly known. More details are discussed in Section \ref{sec: asymp}.

\section{Local Existence}
\label{sec: critical points}
In this section, we compute linearizations of some important critical points of \eqref{eqn: new Einstein equation}. In particular, we compute linearization at some critical points that are in $\mathcal{C}^\pm_{\spin(7)}$ and recover the local existence of $\spin(7)$ metrics on the tubular neighborhood around $\mathbb{CP}^2$ as in \cite{reidegeld_exceptional_2011}.

Critical points of \eqref{eqn: new Einstein equation} on $\mathcal{C}_{RF}$ are the following.
\begin{enumerate}[I]
\item
$P_0^{(k+l)}:=\left(\frac{1}{3},0,0,\frac{1}{3},0,\frac{1}{3},\frac{1}{3},\frac{6\Delta}{k+l}\right),\quad P_0^{(l)}:=\left(0,\frac{1}{3},0,\frac{1}{3},\frac{1}{3},0,\frac{1}{3},\frac{6\Delta}{l}\right),\quad P_0^{(k)}:=\left(0,0,\frac{1}{3},\frac{1}{3},\frac{1}{3},\frac{1}{3},0,\frac{6\Delta}{k}\right)$.

These critical points represent the initial conditions \eqref{eqn: initial condtion} in the new coordinate. Integral curves that emanate from these points represent Ricci-flat metrics that are defined on the tubular neighborhood around $\mathbb{CP}^2$ in $M_{k,l}^{(k+l)}$, $M_{k,l}^{(l)}$ and  $M_{k,l}^{(k)}$, respectively.
\item
\begin{enumerate}
\item
$P_1:=\left(\frac{1}{6},\frac{1}{6},\frac{1}{6},0,\frac{1}{6},\frac{1}{6},\frac{1}{6},0\right)$
\item
$\left(\frac{1}{6},\frac{1}{6},\frac{1}{6},0,\frac{\sqrt{10}}{12},\frac{\sqrt{10}}{24},\frac{\sqrt{10}}{24},0\right),\quad \left(\frac{1}{6},\frac{1}{6},\frac{1}{6},0,\frac{\sqrt{10}}{24},\frac{\sqrt{10}}{12},\frac{\sqrt{10}}{24},0\right),\quad \left(\frac{1}{6},\frac{1}{6},\frac{1}{6},0,\frac{\sqrt{10}}{24},\frac{\sqrt{10}}{24},\frac{\sqrt{10}}{12},0\right)$
\end{enumerate}

All these critical points lie in the space of ALC asymptotics $\mathcal{C}_{RF}\cap\{X_4=0,Z_4=0\}$. Moreover, the point $P_1$ lies in $\mathcal{C}_{G_2}$. If an integral curve converges to one of these critical point, then the metric represented has an ALC asymptotics with space at infinity as $\mathbb{S}^1$ bundle over the cone over the homogeneous Einstein metric on $SU(3)/T^2$. If the curve converges to $P_1$, then the asymptotic cone in the ALC asymptotics has $G_2$ holonomy.
\item
$\left(\frac{1}{7},\frac{1}{7},\frac{1}{7},\frac{1}{7},z_1,z_2,z_3,z_4\right)$, where $\mathcal{R}_i(z_1,z_2,z_3,z_4)=\frac{6}{49}$ for each $i$.

The $z_i$'s in these critical points give the solution of homogeneous Einstein metrics on $N_{k,l}$. One need to solve a quartic polynomial to get the explicit value. By a proper coordinate change, one can conclude that for each $N_{k,l}$, there are exactly two real solutions\cite{kowalski_homogeneous_1993}. In particular, for $N_{1,1}$, we have 
$$
P_{AC-1}:=\left(\frac{1}{7},\frac{1}{7},\frac{1}{7},\frac{1}{7},\frac{2}{7},\frac{1}{7},\frac{1}{7},21\right),\quad 
P_{AC-2}:=\left(\frac{1}{7},\frac{1}{7},\frac{1}{7},\frac{1}{7},\frac{2}{21},\frac{5}{21},\frac{5}{21},\frac{63}{5}\right).
$$
If an integral curve converges to one of these critical points, then the metric represented has $AC$ limit as cone over homogeneous Einstein metrics on $N_{k,l}$. Since the homogeneous Einstein metrics on $N_{k,l}$ has nearly parallel $G_2$ structure, the metric cone has its holonomy group contained in $\spin(7)$.

\item
$\left(\frac{1}{2},-\frac{1}{2},-\frac{1}{2},0,\frac{\sqrt{5}}{2},0,0,0\right),\quad \left(-\frac{1}{2},\frac{1}{2},-\frac{1}{2},0,0,\frac{\sqrt{5}}{2},0,0\right),\quad \left(-\frac{1}{2},-\frac{1}{2},\frac{1}{2},0,0,0,\frac{\sqrt{5}}{2},0\right)$.

These critical points lie in $\mathcal{C}_{G_2}$. They are sources in the restricted system on $\mathcal{C}_{G_2}$. Singular $G_2$ metrics in \cite{cleyton_cohomogeneity-one_2002} are represented by integral curves that emanate from these points.
\item
$\left(\frac{1}{2},0,0,0,0,\frac{1}{4},\frac{1}{4},0\right),\quad  \left(0,\frac{1}{2},0,0,\frac{1}{4},0,\frac{1}{4},0\right),\quad \left(0,0,\frac{1}{2},0,\frac{1}{4},\frac{1}{4},0,0\right)$.

These critical points also lie in $\mathcal{C}_{G_2}$. They are saddles in the restricted system on $\mathcal{C}_{G_2}$. The smooth $G_2$ metrics in \cite{bryant_construction_1989} and \cite{gibbons_einstein_1990} are represented by an integral curves that emanates from these points.
\item
$(x_1,x_2,x_3,x_4,0,0,0,0)$, where $2x_1+2x_2+2x_3+x_4=1$ and $2x_1^2+2x_2^2+2x_3^2+x_4^2=1$.
\item
$(0,0,0,1,0,0,0,z_4)$, where $z_4\geq 0$.
\end{enumerate}

\begin{remark}
One can easily verify that critical points listed above are all on $\mathcal{C}_{RF}$ by counting their non-vanishing $Z_i$ entries. For example, if $Z_1$ does not vanish on a critical point, then one immediately learn that $\mathcal{G}=X_2+X_3-X_1$ from the vector field in \eqref{eqn: new Einstein equation} and can exhaust all the possibilities. Recall that $Z_i$'s are non-negative in $\mathcal{C}_{RF}$ hence we do not need to list any critical point with negative $Z_i$ entries.
\end{remark}

The linearization at $P_0^{(k+l)}$ is 
\begin{equation}
\mathcal{L}\left(P_0^{(k+l)}\right)=\begin{bmatrix}
-\frac{2}{9}&0&0&\frac{2}{9}&0&0&0&-\frac{2}{27}\frac{k+l}{\Delta}\\
0&-\frac{2}{3}&0&0&2&\frac{2}{3}&-\frac{2}{3}&0\\
0&0&-\frac{2}{3}&0&2&-\frac{2}{3}&\frac{2}{3}&0\\
\frac{4}{9}&0&0&-\frac{4}{9}&0&\frac{4}{3}&\frac{4}{3}&\frac{2}{27}\frac{k+l}{\Delta}\\
0&0&0&0&\frac{2}{3}&0&0&0\\
\frac{1}{9}&\frac{1}{3}&-\frac{1}{3}&\frac{2}{9}&0&0&0&0\\
\frac{1}{9}&-\frac{1}{3}&\frac{1}{3}&\frac{2}{9}&0&0&0&0\\
-8\frac{\Delta}{k+l}&0&0&2\frac{\Delta}{k+l}&0&0&0&0
\end{bmatrix}
\end{equation}
Eigenvalues and eigenvectors of $\mathcal{L}\left(P_0^{(k+l)}\right)$ are 
$$
\lambda_1=\lambda_2=\lambda_3=\lambda_4=\frac{2}{3},\quad \lambda_5=\lambda_6=-\frac{2}{3},\quad \lambda_7=\lambda_8=-\frac{4}{3}
$$
$$
v_1=\begin{bmatrix}
2\\
0\\
0\\
-4\\
0\\
-1\\
-1\\
-36\frac{\Delta}{k+l}
\end{bmatrix},v_2=\begin{bmatrix}
-3(k+l)\\
4k+5l\\
5k+4l\\
-12(k+l)\\
3(k+l)\\
-5k-4l\\
-4k-5l\\
0
\end{bmatrix},v_3=\begin{bmatrix}
0\\
1\\
-1\\
0\\
0\\
1\\
-1\\
0
\end{bmatrix},v_4=\begin{bmatrix}
0\\
3\\
3\\
0\\
2\\
0\\
0\\
0
\end{bmatrix}, v_5=\begin{bmatrix}
4\\
-3\\
-3\\
4\\
0\\
-2\\
-2\\
36\frac{\Delta}{k+l}
\end{bmatrix}, v_6=\begin{bmatrix}
0\\
1\\
1\\
0\\
0\\
0\\
0\\
0
\end{bmatrix},v_7=\begin{bmatrix}
0\\
2\\
-2\\
0\\
0\\
-1\\
1\\
0
\end{bmatrix},v_8=\begin{bmatrix}
2\\
-1\\
1\\
-4\\
0\\
1\\
0\\
18\frac{\Delta}{k+l}
\end{bmatrix}.
$$
Except $v_4$ and $v_6$, all the other eigenvectors are tangent to $\mathcal{C}_{RF}$. $v_1, v_2$ and $v_7$ are tangent to $\mathcal{C}^+_{\spin(7)}$. For each choice of parameter $(s_1,s_2,s_3)\in\mathbb{R}^3$, we fix $s_1^2+s_2^2+s_3^2=1$ to cut down the redundancy. We have the linearized solution
\begin{equation}
\label{eqn: linearized solution RF}
P_0^{(k+l)}+s_1 e^{\frac{2\eta}{3}}v_1+s_2 e^{\frac{2\eta}{3}}v_2+s_3 e^{\frac{2\eta}{3}}v_3
\end{equation}
for the Ricci-flat system and 
\begin{equation}
\label{eqn: linearized solution k+l spin(7)}
P_0^{(k+l)}+s_1e^{\frac{2\eta}{3}}v_1+s_2 e^{\frac{2\eta}{3}}v_2
\end{equation}
for the $\spin(7)$ subsystem.
By the Hartman--Grobman Theorem, there is a 1 to 1 correspondence between each choice of $(s_1,s_2,s_3)\in\mathbb{S}^2$ and an actual integral curve that emanates from $P_0^{(k+l)}$. By the unstable version of Theorem 4.5 in \cite{coddington_theory_1955}, there is no ambiguity to use $\gamma^{(k+l)}_{(s_1,s_2,s_3)}$ to denote the integral curve that emanate from $P_0^{(k+l)}$ with 
$$
\gamma^{(k+l)}_{(s_1,s_2,s_3)}=P_0^{(k+l)}+s_1 e^{\frac{2\eta}{3}}v_1+s_2 e^{\frac{2\eta}{3}}v_2+s_3 e^{\frac{2\eta}{3}}v_3+O(e^{\left(\frac{2}{3}+\delta\right)\eta}).
$$
We hence abuse the notation by using $\gamma^{(k+l)}_{(s_1,s_2,s_3)}$ to denote the Ricci-flat metric that is defined on the tubular neighborhood around $\mathbb{CP}^2$ in $M_{k,l}^{(k+l)}$. Similar notation is carried over for $\gamma^{(k+l)}_{(s_1,s_2)}:=\gamma^{(k+l)}_{(s_1,s_2,0)}$, a locally defined $\spin(7)$ metric.

The linearization at $P_0^{(k)}$ is 
\begin{equation}
\mathcal{L}\left(P_0^{(k)}\right)=\begin{bmatrix}
-\frac{2}{3}&0&0&0&\frac{2}{3}&-\frac{2}{3}&2&0\\
0&\frac{-2}{3}&0&0&-\frac{2}{3}&\frac{2}{3}&2&0\\
0&0&-\frac{2}{9}&\frac{2}{9}&0&0&0&-\frac{2}{27}\frac{k}{\Delta}\\
0&0&\frac{4}{9}&-\frac{4}{9}&\frac{4}{3}&\frac{4}{3}&0&\frac{2}{27}\frac{k}{\Delta}\\
\frac{1}{3}&-\frac{1}{3}&\frac{1}{9}&\frac{2}{9}&0&0&0&0\\
-\frac{1}{3}&\frac{1}{3}&\frac{1}{9}&\frac{2}{9}&0&0&0&0\\
0&0&0&0&0&0&\frac{2}{3}&0\\
0&0&-8\frac{\Delta}{k}&2\frac{\Delta}{k}&0&0&0&0
\end{bmatrix}
\end{equation}
Eigenvalues and eigenvectors of $\mathcal{L}\left(P_0^{(k)}\right)$ are 
$$
\lambda_1=\lambda_2=\lambda_3=\lambda_4=\frac{2}{3},\quad \lambda_5=\lambda_6=-\frac{2}{3},\quad \lambda_7=\lambda_8=-\frac{4}{3}
$$
$$
v_1=\begin{bmatrix}
0\\
0\\
2\\
-4\\
-1\\
-1\\
0\\
-36\frac{\Delta}{k}
\end{bmatrix},v_2=\begin{bmatrix}
5k+l\\
4k-l\\
-3k\\
-12k\\
-4k+l\\
-5k-l\\
3k\\
0
\end{bmatrix},v_3=\begin{bmatrix}
-1\\
1\\
0\\
0\\
-1\\
1\\
0\\
0
\end{bmatrix},v_4=\begin{bmatrix}
3\\
3\\
0\\
0\\
0\\
2\\
0\\
0
\end{bmatrix}, v_5=\begin{bmatrix}
0\\
0\\
2\\
2\\
-1\\
-1\\
0\\
18\frac{\Delta}{k}
\end{bmatrix}, v_6=\begin{bmatrix}
1\\
1\\
0\\
0\\
0\\
0\\
0\\
0
\end{bmatrix},v_7=\begin{bmatrix}
2\\
-2\\
0\\
0\\
-1\\
1\\
0\\
0
\end{bmatrix},v_8=\begin{bmatrix}
-1\\
1\\
2\\
-4\\
2\\
0\\
0\\
18\frac{\Delta}{k}
\end{bmatrix}.
$$
The first three eigenvectors are tangent to $\mathcal{C}_{RF}$ and the first two eigenvectors are tangent to $\mathcal{C}^-_{\spin(7)}$. We use $\gamma_{(s_1,s_2,s_3)}^{(k)}$ to denote integral curve such that
$$\gamma_{(s_1,s_2,s_3)}^{(k)}=P_0^{(k)}+s_1 e^{\frac{2\eta}{3}}v_1+s_2 e^{\frac{2\eta}{3}}v_2+s_3 e^{\frac{2\eta}{3}}v_3+O(e^{\left(\frac{2}{3}+\delta\right)\eta})$$
near $P_0^{(k)}$. Hence $\gamma_{(s_1,s_2,s_3)}^{(k)}$ and $\gamma_{(s_1,s_2)}^{(k)}:=\gamma_{(s_1,s_2,0)}^{(k)}$ are respectively locally defined Ricci-flat metrics and $\spin(7)$ metrics on the tubular neighborhood around $\mathbb{CP}^2$ in $M_{k,l}^{(k)}$

The analysis of linearizations at $P_0^{(l)}$ is similar. For the linearization at $P_0^{(l)}$, one find three unstable eigenvectors that are tangent to $\mathcal{C}_{RF}$, and two among these three vectors are tangent to $\mathcal{C}^-_{\spin(7)}$. In summary, we prove the following lemma.
\begin{lemma}
\label{lem: local}
Let $i\in\{k+l,k,l\}$.
There is a continuous 2-parameter family of Ricci-flat metrics on the tubular neighborhood around $\mathbb{CP}^2$ on each $M_{k,l}^{(i)}$, represented by integral curves that emanates from $P_0^{(i)}$. If $s_3=0$, $\gamma^{(i)}_{(s_1,s_2)}:=\gamma^{(i)}_{(s_1,s_2,0)}$ gives rise to the locally defined $\spin(7)$ metric on $M_{k,l}^{(i)}$. $\gamma^{(k)}_{(s_1,s_2)}$ and $\gamma^{(l)}_{(s_1,s_2)}$ share the same chirality that is opposite to the one of $\gamma^{(k+l)}_{(s_1,s_2)}$.
\end{lemma}

The linearization at $P_1$ is 
\begin{equation}
\mathcal{L}\left(P_1\right)=\begin{bmatrix}
-\frac{13}{18}&\frac{1}{9}&\frac{1}{9}&0&\frac{1}{3}&\frac{2}{3}&\frac{2}{3}&0\\
\frac{1}{9}&-\frac{13}{18}&\frac{1}{9}&0&\frac{2}{3}&\frac{1}{3}&\frac{2}{3}&0\\
\frac{1}{9}&\frac{1}{9}&-\frac{13}{18}&0&\frac{2}{3}&\frac{2}{3}&\frac{1}{3}&0\\
0&0&0&-\frac{5}{6}&0&0&0&0\\
\frac{5}{18}&-\frac{1}{18}&-\frac{1}{18}&0&0&0&0&0\\
-\frac{1}{18}&\frac{5}{18}&-\frac{1}{18}&0&0&0&0&0\\
-\frac{1}{18}&-\frac{1}{18}&\frac{5}{18}&0&0&0&0&0\\
0&0&0&0&0&0&0&-\frac{1}{6}
\end{bmatrix}
\end{equation}
Eigenvalues and eigenvectors of $\mathcal{L}\left(P_1\right)$ are 
$$
\lambda_1=\lambda_2=\lambda_3=-\frac{1}{6},\quad \lambda_4=\lambda_5=-\frac{5}{6},\quad \lambda_6=\lambda_7=-\frac{2}{3},\quad \lambda_8=\frac{1}{3}
$$
$$
v_1=\begin{bmatrix}
0\\
0\\
0\\
0\\
0\\
0\\
0\\
1
\end{bmatrix},v_2=\begin{bmatrix}
2\\
-1\\
-1\\
0\\
-4\\
2\\
2\\
0
\end{bmatrix},v_3=\begin{bmatrix}
0\\
-1\\
1\\
0\\
0\\
2\\
-2\\
0
\end{bmatrix},v_4=\begin{bmatrix}
0\\
0\\
0\\
1\\
0\\
0\\
0\\
0
\end{bmatrix},v_5=\begin{bmatrix}
-5\\
-5\\
-5\\
0\\
1\\
1\\
1\\
0
\end{bmatrix},v_6=\begin{bmatrix}
4\\
-2\\
-2\\
0\\
-2\\
1\\
1\\
0
\end{bmatrix},v_7=\begin{bmatrix}
0\\
2\\
-2\\
0\\
0\\
-1\\
1\\
0
\end{bmatrix},v_8=\begin{bmatrix}
2\\
2\\
2\\
0\\
1\\
1\\
1\\
0
\end{bmatrix}
$$
It is clear that except $v_5$ and $v_8$, all the other eigenvectors are tangent to $\mathcal{C}_{RF}$. Hence $P_1$ is a sink in the system \eqref{eqn: new Einstein equation} restricted on $\mathcal{C}_{RF}$. Therefore, the critical point is also a sink in the subsystem restricted on $\mathcal{C}^\pm_{\spin(7)}$. Recall that $P_1$ is in $\mathcal{C}_{G_2}$, the 2-dimensional invariant set of ALC asymptotics. If a $\gamma_{(s_1,s_2)}^{(i)}$ gets sufficiently closed to $P_1$, then it converges to $P_1$ and the metric represented has ALC asymptotics. 

\section{Global Existence and Asymptotics}
\label{sec: global}
With the assumption $k\geq l\geq 0$, we prove that $\spin(7)$ metrics $\gamma^{(k+l)}_{(s_1,s_2)}$ and $\gamma^{(k)}_{(s_1,s_2)}$ in Lemma \ref{lem: local} are globally defined for $(s_1,s_2)$ in an open set of $\mathbb{S}^1$. Due to the limitation of our method, it is not clear if we have global existence of $\gamma^{(l)}_{(s_1,s_2)}$. More details are discussed at Section \ref{subsec: Tk Tl}.

\subsection{Global Metrics on $M_{k,l}^{(k+l)}$}
Define 
\begin{equation}
\mathcal{S}=\mathcal{C}^+_{\spin(7)}\cap \left\{0\leq Z_4\leq \frac{6\Delta}{k+l}\right\}
\end{equation}
\begin{proposition}
\label{prop: S is compact}
The set $\mathcal{S}$ consists of two connected component. The component that satisfies $Z_2+Z_3\leq \frac{2}{3}$ is compact.
\end{proposition}
\begin{proof}
By \eqref{eqn: conservation Z spin(7)}, we have 
\begin{equation}
\begin{split}
0&=1-2(Z_1+Z_2+Z_3)+\left(\frac{k+l}{2\Delta}Z_2Z_3-\frac{l}{2\Delta}Z_1Z_3-\frac{k}{2\Delta}Z_1Z_2\right)Z_4\\
&\leq 1-2(Z_2+Z_3)+3Z_2Z_3\quad \text{since $Z_4\leq\frac{6\Delta}{k+l}$ and $Z_i\geq 0$}\\
&\leq 1-2(Z_2+Z_3)+\frac{3}{4}(Z_2+Z_3)^2\\
&=\left(\frac{3}{2}(Z_2+Z_3)-1\right)\left(\frac{1}{2}(Z_2+Z_3)-1\right)
\end{split}.
\end{equation}
Hence either $Z_2+Z_3\leq \frac{2}{3}$ or $Z_2+Z_3\geq 2$ in $\mathcal{S}$.
Consider the component that satisfies $Z_2+Z_3\leq \frac{2}{3}$. As for $Z_1$, we have 
\begin{equation}
\begin{split}
0&=1-2(Z_1+Z_2+Z_3)+\left(\frac{k+l}{2\Delta}Z_2Z_3-\frac{l}{2\Delta}Z_1Z_3-\frac{k}{2\Delta}Z_1Z_2\right)Z_4\\
&\leq 1-2Z_1-2(Z_2+Z_3)+3Z_2Z_3 \quad \text{since $Z_4\leq\frac{6\Delta}{k+l}$ and $Z_i\geq 0$}\\
&\leq 1-2Z_1-2(Z_2+Z_3)+\frac{3}{2}(Z_2+Z_3)^2\\
&\leq 1-2Z_1 \quad \text{since $Z_2+Z_3\leq\frac{2}{3}$}
\end{split}.
\end{equation}
Hence all $Z_1\leq \frac{1}{2}$ and all $Z_i$'s are bounded above. By  \eqref{eqn: new spin(7) equation}, it is clear that all $X_i$'s are bounded. Hence the component is compact.
\end{proof}

By Proposition \ref{prop: S is compact}, we can write $\mathcal{S}=\check{\mathcal{S}}\sqcup \hat{\mathcal{S}}$, where $\check{\mathcal{S}}$ denote the compact component 
$$\mathcal{S}_{\spin(7)}\cap \left\{0\leq Z_2+Z_3\leq \frac{2}{3}\right\}.$$
We focus on $\check{\mathcal{S}}$ in the following since $P_0^{(k+l)}, P_1\in \check{\mathcal{S}}$. We show that $\check{\mathcal{S}}$ is invariant so that these integral curves stay in the compact set $\check{\mathcal{S}}$. Note that 
\begin{equation}
\left.\langle\nabla Z_4, V\rangle\right|_{Z_4=\frac{6\Delta}{k+l}}=Z_4(-\mathcal{G}+X_4)=-\frac{6\Delta}{k+l}(\mathcal{G}-X_4).
\end{equation}
To show that $\check{\mathcal{S}}$ is invariant, it suffices to prove that 
\begin{equation}
\label{eqn: S in G-X4>0}
\check{\mathcal{S}}= \check{\mathcal{S}}\cap \{\mathcal{G}-X_4\geq 0\},
\end{equation}
as demonstrated in Figure \ref{fig: mathcalS}.

\begin{figure}[h!] 
\centering
\begin{subfigure}{.3\textwidth}
  \centering 
  \includegraphics[clip,trim=12cm 3cm 8cm 3cm,width=1\linewidth]{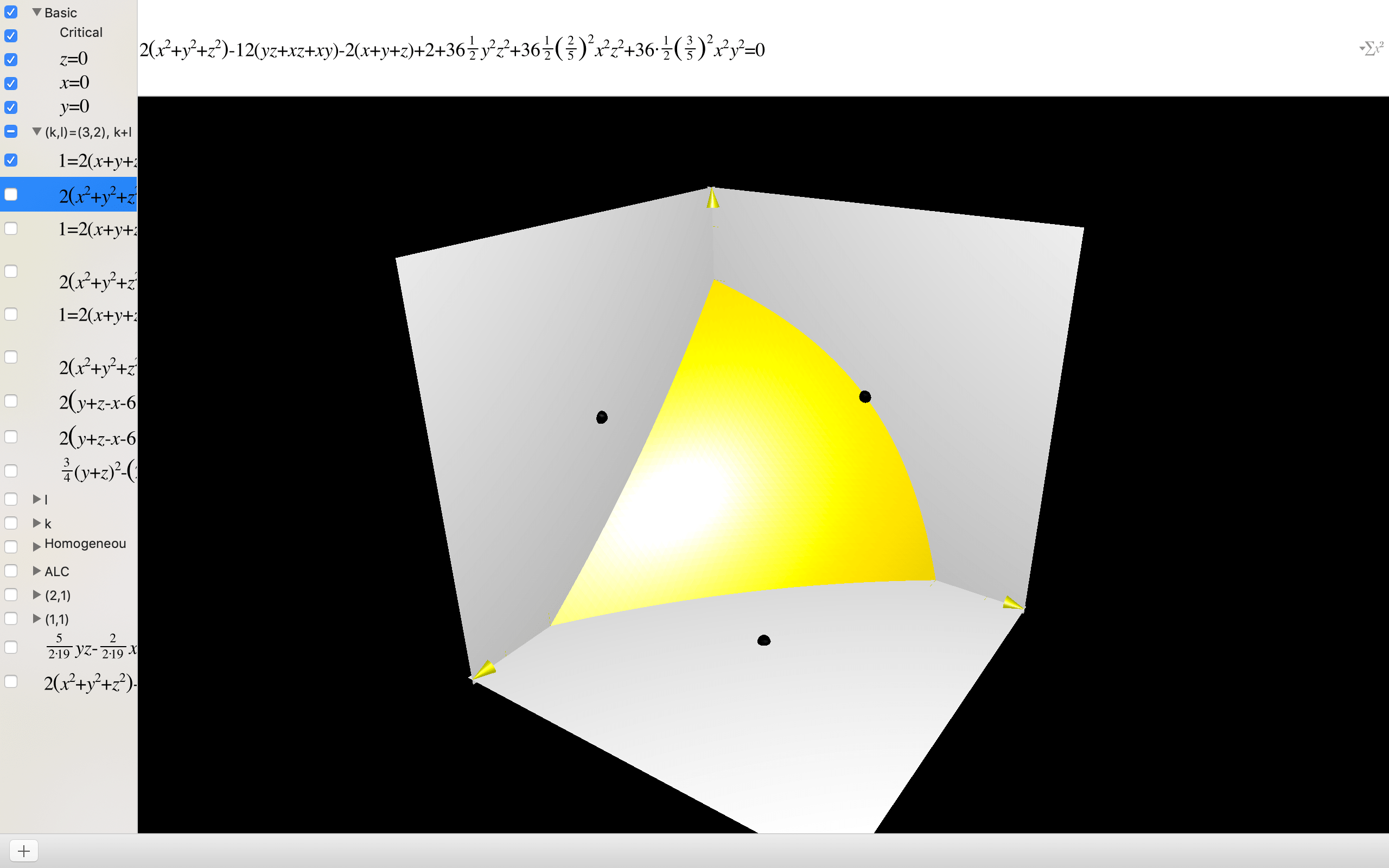}
\end{subfigure}
\begin{subfigure}{.3\textwidth}
  \centering
  \includegraphics[clip,trim=12cm 3cm 8cm 3cm,width=1\linewidth]{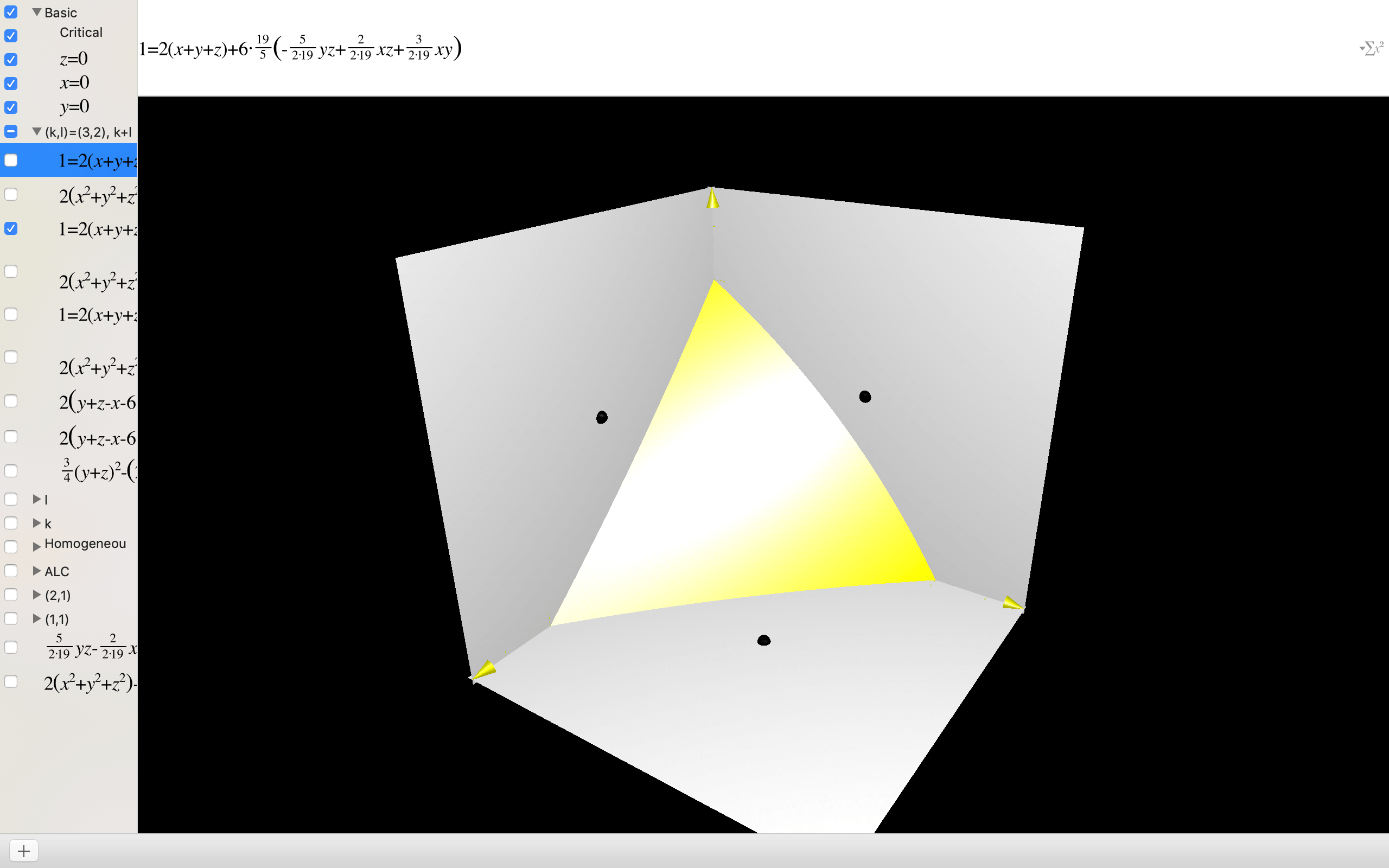}
\end{subfigure}
\begin{subfigure}{.3\textwidth}
  \centering
  \includegraphics[clip,trim=12cm 3cm 8cm 3cm,width=1\linewidth]{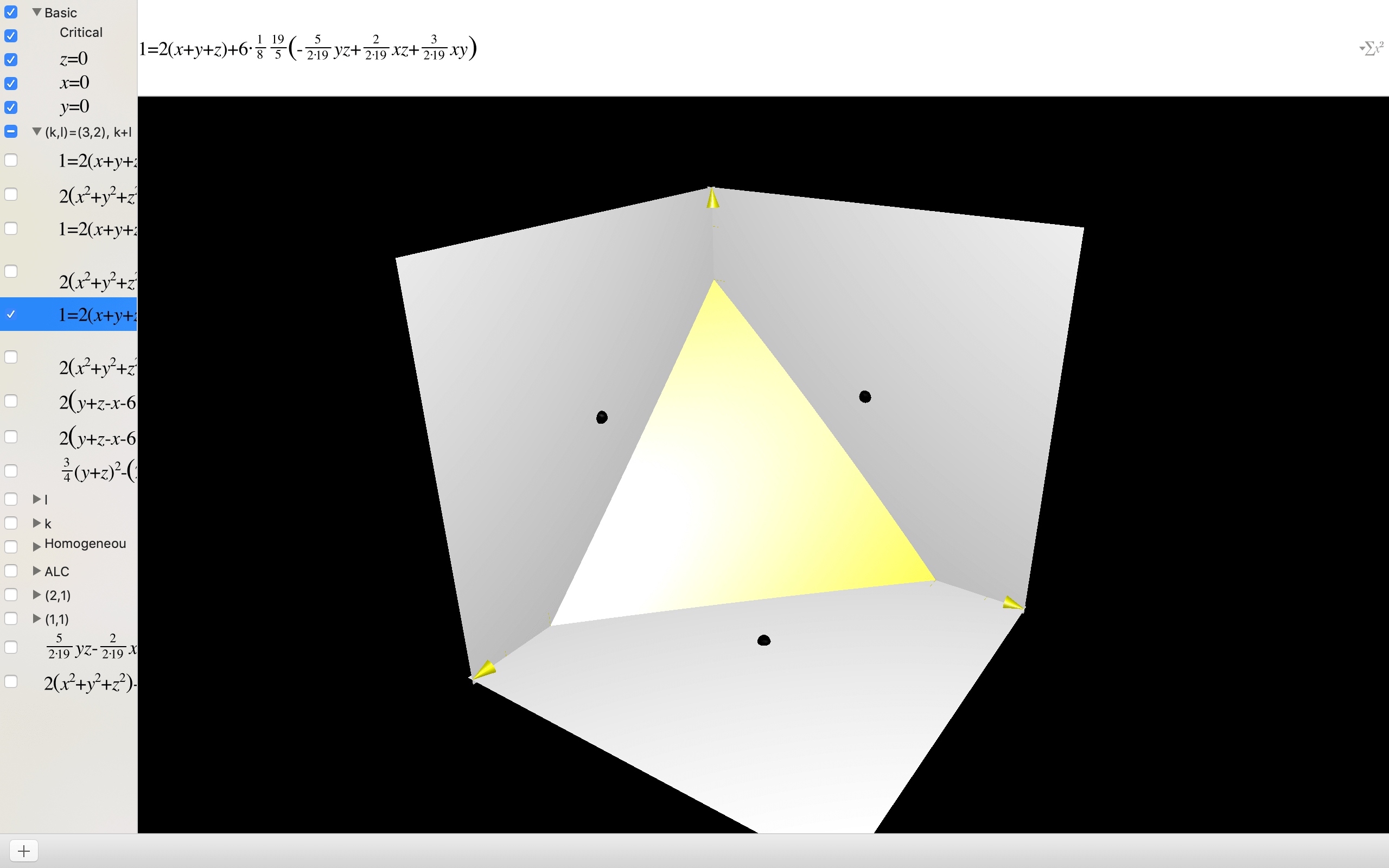}
\end{subfigure}\\
\begin{subfigure}{.3\textwidth}
  \centering
  \includegraphics[clip,trim=12cm 3cm 8cm 3cm,width=1\linewidth]{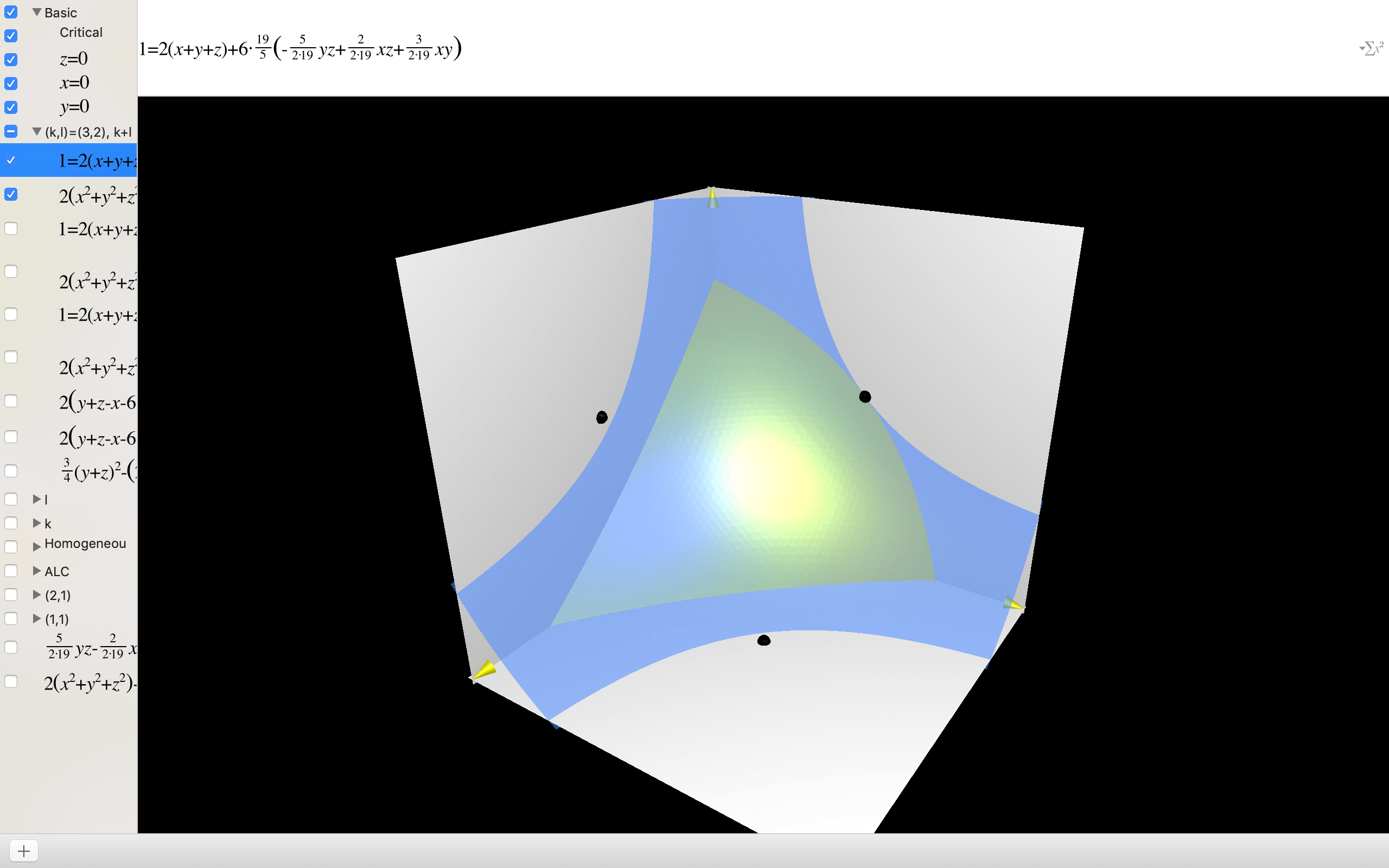}
  \caption{$Z_4=\frac{6\Delta}{k+l}$}
\end{subfigure}
\begin{subfigure}{.3\textwidth}
  \centering
  \includegraphics[clip,trim=12cm 3cm 8cm 3cm,width=1\linewidth]{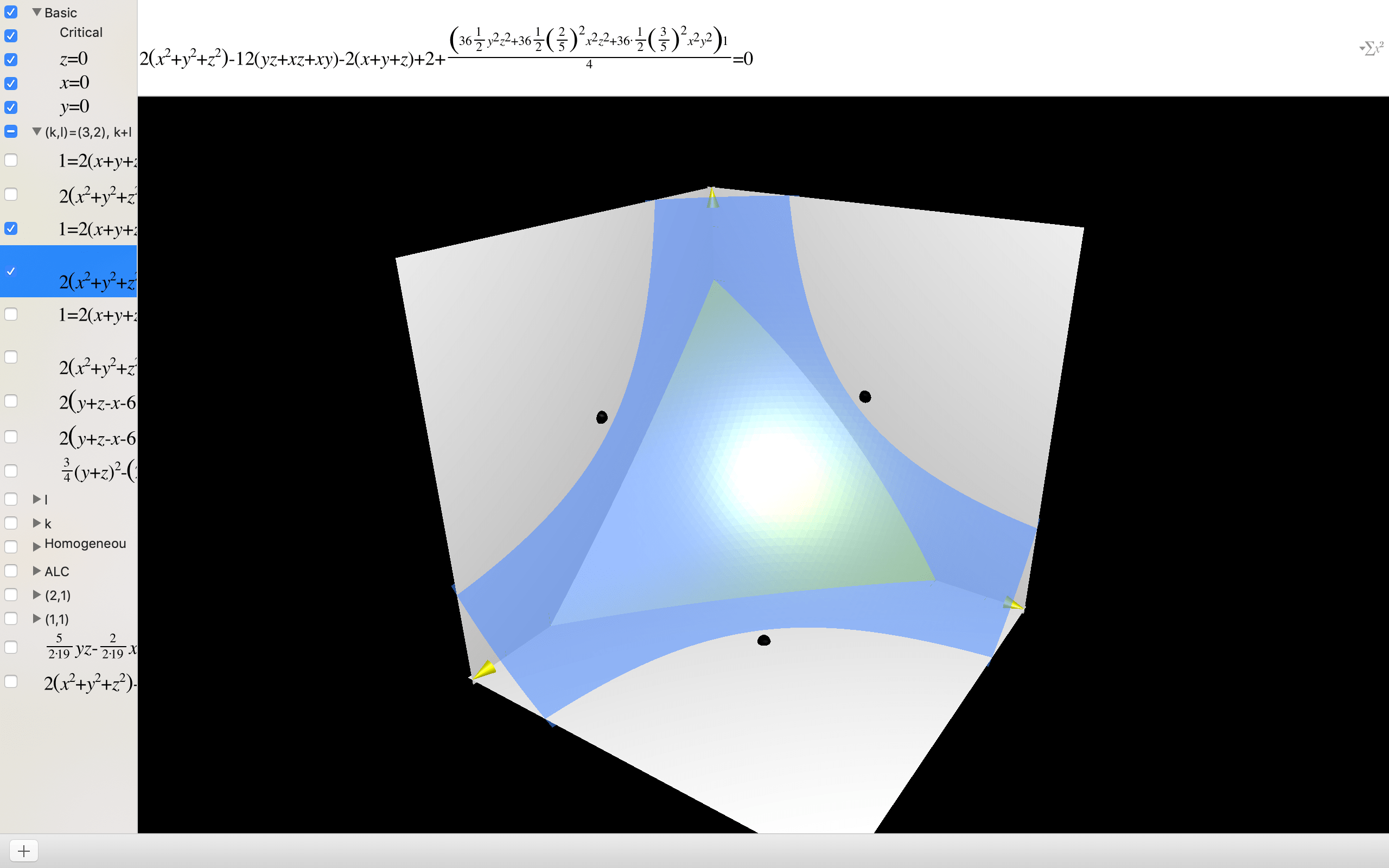}
  \caption{$Z_4=\frac{1}{2}\frac{6\Delta}{k+l}$}
\end{subfigure}
\begin{subfigure}{.3\textwidth}
  \centering
  \includegraphics[clip,trim=12cm 3cm 8cm 3cm,width=1\linewidth]{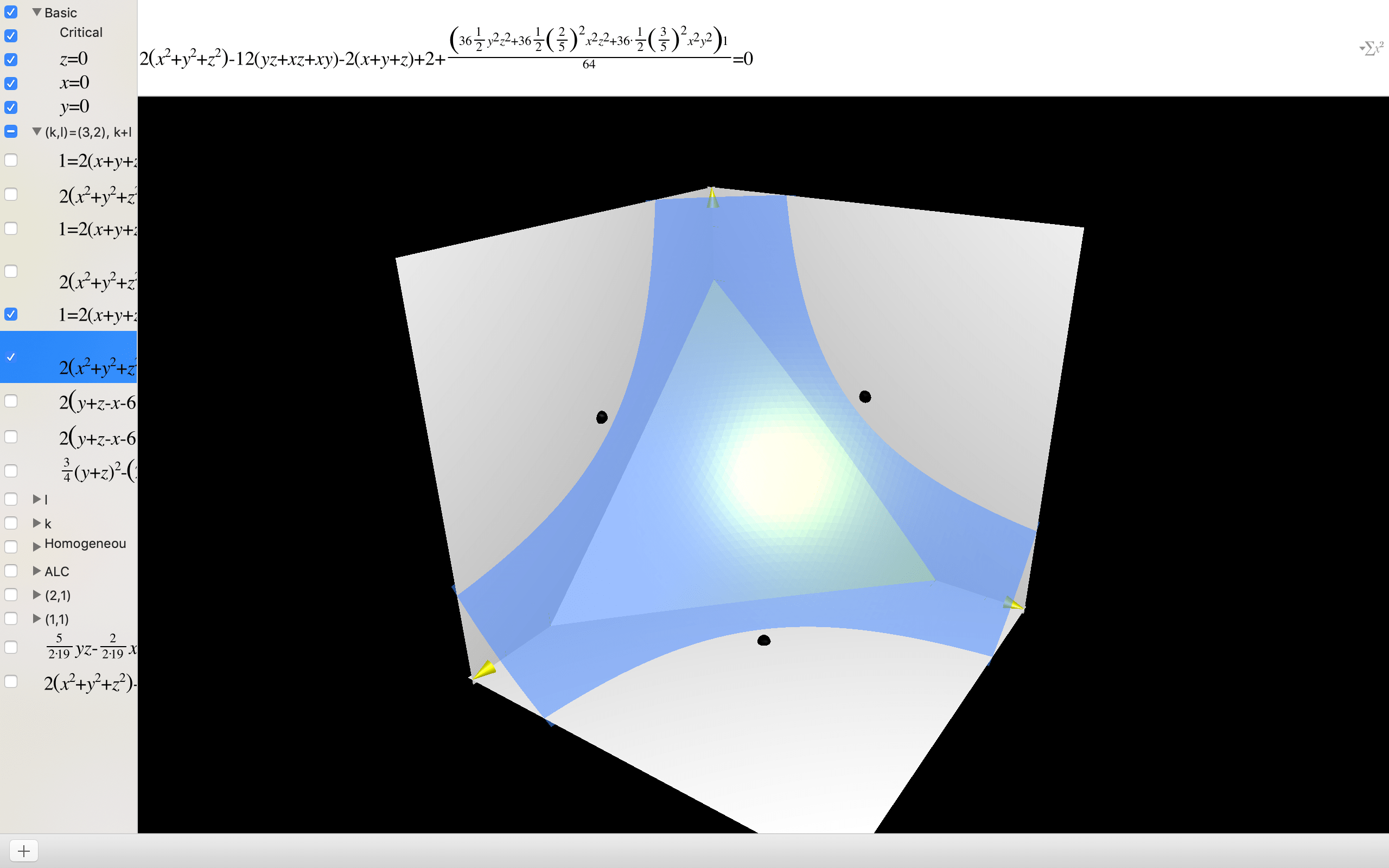}
  \caption{$Z_4=\frac{1}{8}\frac{6\Delta}{k+l}$}
\end{subfigure}
\caption{The figures above are ``picture proof'' $\check{\mathcal{S}}\subset \check{\mathcal{S}}\cap \{\mathcal{G}-X_4\geq 0\}$, with $(k,l)=(3,2)$.
Figures in the first row demonstrate  the compact set $\check{\mathcal{S}}$ for fixed $Z_4\leq \frac{6\Delta}{k+l}$. Figures in the second row demonstrate the set $\check{\mathcal{S}}\cap \{\mathcal{G}-X_4=0\}$ for fixed $Z_4\leq \frac{6\Delta}{k+l}$, with all $X_1$, $X_2$ and $X_3$ replaced by $Z_i$'s using \eqref{eqn: new spin(7) equation} and apply Proposition \ref{prop: X_4 and Zi} for $X_4$. Points shown are $P_0^{(k+l)}$, $P_0^{(k)}$ and $P_0^{(l)}$. Note that $P_0^{(k+l)}$ lies in the boundary of $\check{\mathcal{S}}$ as shown in the first figure. From the observation, surfaces that represent $\check{\mathcal{S}}$ are ``behind'' those that represent $\check{\mathcal{S}}\cap \{\mathcal{G}-X_4=0\}$, meaning $\mathcal{G}-X_4\geq 0$ is satisfied on each point on $\check{\mathcal{S}}$. In that way, we can tell that $Z_4$ is monotonic decreasing if it decreases initially.}
\label{fig: mathcalS}
\end{figure}

The $\spin(7)$ condition \eqref{eqn: new spin(7) equation} plays a key role in proving \eqref{eqn: S in G-X4>0}. Specifically, we simplify the inequalities by replacing all $X_i$'s with $Z_i$'s. Then the question boils down to showing the separation of two algebraic surfaces in $\mathbb{R}^4$. In order to prove the separation, we find algebraic surfaces that are `` between'' the one given by \eqref{eqn: conservation Z spin(7)} and the one from $\check{\mathcal{S}}\cap \{\mathcal{G}-X_4=0\}$. These so-called ``separation surfaces'' are symmetric with respect to the hyperplane $Z_2=Z_3$. In this way, slices of these surfaces with fixed $Z_1$ and $Z_4$ are easier to be proven to be separated.

Define 
\begin{equation}
\begin{split}
&\mathcal{Q}(Z_1,Z_2,Z_3,Z_4)\\
&=\frac{k+l}{8\Delta}Z_4(Z_2+Z_3)^2-\left(2+\frac{k+l}{4\Delta}Z_1Z_4\right)(Z_2+Z_3) +\frac{k+l}{8\Delta}Z_1^2Z_4-2Z_1+1.
\end{split}
\end{equation} 
\begin{proposition}
\label{prop: included in Q}
Each points in $\check{\mathcal{S}}$ satisfies $\mathcal{Q}\geq 0$, i.e., $\check{\mathcal{S}}= \check{\mathcal{S}}\cap \{\mathcal{Q}\geq 0\}.$
\end{proposition}
\begin{proof}
By \eqref{eqn: conservation Z spin(7)}, we have 
\begin{equation}
\begin{split}
0&=1-2(Z_1+Z_2+Z_3)+\frac{k+l}{2\Delta}Z_2Z_3Z_4-\frac{l}{2\Delta}Z_1Z_3Z_4-\frac{k}{2\Delta}Z_1Z_2Z_4\\
&=\frac{k+l}{8\Delta}Z_4(Z_2+Z_3)^2-\left(2+\frac{k+l}{4\Delta}Z_1Z_4\right)(Z_2+Z_3)-\frac{k+l}{8\Delta}Z_4\left((Z_3-Z_2)-\frac{k-l}{k+l}Z_1\right)^2\\
&\quad +\left(\frac{k-l}{k+l}\right)^2\frac{k+l}{8\Delta}Z_1^2Z_4-2Z_1+1\\
&\leq \frac{k+l}{8\Delta}Z_4(Z_2+Z_3)^2-\left(2+\frac{k+l}{4\Delta}Z_1Z_4\right)(Z_2+Z_3) +\left(\frac{k-l}{k+l}\right)^2\frac{k+l}{8\Delta}Z_1^2Z_4-2Z_1+1\\
&\leq \frac{k+l}{8\Delta}Z_4(Z_2+Z_3)^2-\left(2+\frac{k+l}{4\Delta}Z_1Z_4\right)(Z_2+Z_3) +\frac{k+l}{8\Delta}Z_1^2Z_4-2Z_1+1\\
&=\mathcal{Q}
\end{split}.
\end{equation}
Hence $\mathcal{Q}\geq 0$ in $\check{\mathcal{S}}$. The proof is complete.
\end{proof}
Define 
\begin{equation}
\begin{split}
&\mathcal{A}(Z_1,Z_2,Z_3,Z_4)\\
&=\frac{1}{32}\left(\frac{k+l}{\Delta}\right)^2(Z_2+Z_3)^4Z_4^2-2(Z_2+Z_3)^2-(12Z_1+2)(Z_2+Z_3)+2Z_1^2-2Z_1+2.
\end{split}
\end{equation} 
\begin{proposition}
\label{prop: A included in G-X4>0}
We have the following inclusion
$\check{\mathcal{S}}\cap \{\mathcal{A}\geq 0\}\subset \check{\mathcal{S}}\cap \{\mathcal{G}-X_4\geq 0\}.$
\end{proposition}
\begin{proof}
By \eqref{eqn: new conservation}, we have
\begin{equation}
\mathcal{G}-X_4=1-\mathcal{R}_s-X_4.
\end{equation}
By Proposition \ref{prop: X_4 and Zi}, the computation continues as
\begin{equation}
\label{eqn: G-X4}
\begin{split}
&\mathcal{G}-X_4\\
&=2(Z_1^2+Z_2^2+Z_3^2)-12(Z_2Z_3+Z_1Z_3+Z_1Z_2)-2(Z_1+Z_2+Z_3)+2.
\\
&\quad +\frac{1}{2}\frac{(k+l)^2}{\Delta^2}Z_2^2Z_3^2Z_4^2+\frac{1}{2}\frac{l^2}{\Delta^2}Z_1^2Z_3^2Z_4^2+\frac{1}{2}\frac{k^2}{\Delta^2}Z_1^2Z_2^2Z_4^2\\
&\geq 2(Z_1^2+Z_2^2+Z_3^2)-12(Z_2Z_3+Z_1Z_3+Z_1Z_2)-2(Z_1+Z_2+Z_3)+2.
\\
&\quad +\frac{1}{2}\frac{(k+l)^2}{\Delta^2}Z_2^2Z_3^2Z_4^2\\
&=\frac{1}{32}\left(\frac{k+l}{\Delta}\right)^2(Z_2+Z_3)^4Z_4^2-2(Z_2+Z_3)^2-(12Z_1+2)(Z_2+Z_3)+2Z_1^2-2Z_1+2\\
&\quad +(Z_3-Z_2)^2\left(\frac{1}{32}\left(\frac{k+l}{\Delta}\right)^2
(Z_3-Z_2)^2Z_4^2-\frac{1}{16}\left(\frac{k+l}{\Delta}\right)^2(Z_2+Z_3)^2Z_4^2+4\right)\\
&\geq \frac{1}{32}\left(\frac{k+l}{\Delta}\right)^2(Z_2+Z_3)^4Z_4^2-2(Z_2+Z_3)^2-(12Z_1+2)(Z_2+Z_3)+2Z_1^2-2Z_1+2\\
&\quad +(Z_3-Z_2)^2\left(4-\frac{1}{16}\left(\frac{k+l}{\Delta}\right)^2(Z_2+Z_3)^2Z_4^2\right)\\
&\geq \frac{1}{32}\left(\frac{k+l}{\Delta}\right)^2(Z_2+Z_3)^4Z_4^2-2(Z_2+Z_3)^2-(12Z_1+2)(Z_2+Z_3)+2Z_1^2-2Z_1+2\\
&\quad +(Z_3-Z_2)^2\left(4-\frac{9}{4}(Z_2+Z_3)^2\right)\quad  \text{since $Z_4\leq \frac{6\Delta}{k+l}$ in $\check{\mathcal{S}}$}\\
&\geq \frac{1}{32}\left(\frac{k+l}{\Delta}\right)^2(Z_2+Z_3)^4Z_4^2-2(Z_2+Z_3)^2-(12Z_1+2)(Z_2+Z_3)+2Z_1^2-2Z_1+2 \\
&\quad  \text{since $Z_2+Z_3\leq \frac{2}{3}$ in $\check{\mathcal{S}}$}\\
&= \mathcal{A}
\end{split}
\end{equation}
Hence $\mathcal{A}\geq 0$ in $\check{\mathcal{S}}$ implies $\mathcal{G}-X_4\geq 0$ in $\check{\mathcal{S}}$. The proof is complete.
\end{proof}

By Proposition \ref{prop: included in Q} and Proposition \ref{prop: A included in G-X4>0}, to show that $\check{\mathcal{S}}\subset \check{\mathcal{S}}\cap \{\mathcal{G}-X_4\geq 0\}$, it suffices to prove that 
\begin{equation}
\label{eqn: Q in A}
\check{\mathcal{S}}\cap \{\mathcal{Q} \geq 0\}\subset \check{\mathcal{S}}\cap \{\mathcal{A}\geq 0\}.
\end{equation}
The inclusion \eqref{eqn: Q in A} is easier to show since it only involves $Z_i$'s and polynomials $\mathcal{Q}$ and $\mathcal{A}$ can be viewed as polynomials in $Z_2+Z_3$. Specifically, by Proposition \ref{prop: S is compact}, we know that $0\leq Z_1\leq \frac{1}{2}$ in $\check{\mathcal{S}}\cap \{\mathcal{Q} \geq 0\}$. Fix $Z_1=\alpha\in\left[0,\frac{1}{2}\right]$ and $Z_4=\frac{6\Delta}{k+l}\beta$ with $\beta\in (0,1]$, define
\begin{equation}
\begin{split}
{q_1}_{(\alpha,\beta)}(Z_2+Z_3)&:=\mathcal{Q}\left(\alpha,Z_2,Z_3,\frac{6\Delta}{k+l}\beta\right)\\
&=\frac{3\beta}{4}(Z_2+Z_3)^2-\left(2+\frac{3}{2}\alpha\beta\right)(Z_2+Z_3)+\frac{3}{4}\alpha^2\beta-2\alpha+1,\\
{q_2}_{(\alpha,\beta)}(Z_2+Z_3)&:=\mathcal{A}\left(\alpha,Z_2,Z_3,\frac{6\Delta}{k+l}\beta\right)\\
&=\frac{9\beta^2}{8}(Z_2+Z_3)^4-2(Z_2+Z_3)^2-(12\alpha+2)(Z_2+Z_3)+2\alpha^2-2\alpha+2.
\end{split}
\end{equation}
It is clear that for each fixed $(\alpha,\beta)\in \left[0,\frac{1}{2}\right]\times (0,1]$, we have 
\begin{equation}
\begin{split}
{q_1}_{(\alpha,\beta)}(0)&=\frac{3}{4}\alpha^2\beta-2\alpha+1\geq 0,\\
{q_1}_{(\alpha,\beta)}\left(\frac{2}{3}\right)&=-\frac{1}{3}-2\alpha+\frac{\beta}{3}-\alpha\beta+\frac{3}{4}\alpha^2\beta\leq -2\alpha-\frac{5}{8}\alpha\beta\leq 0.
\end{split}
\end{equation}
It is also clear that $q_1$ has two positive roots for each fixed $(\alpha,\beta)$, with one root be no larger than $\frac{2}{3}$ and the other one no smaller.
Let $\omega(\alpha,\beta)$ be the smaller positive root of ${q_1}_{(\alpha,\beta)}$. It is clear that $\omega(\alpha,\beta)\in \left[0,\frac{2}{3}\right]$.
We have 
\begin{equation}
\begin{split}
&\left\{0\leq Z_1\leq \frac{1}{2},\quad 0\leq Z_4\leq \frac{6\Delta}{k+l},\quad 0\leq Z_2+Z_3\leq \frac{2}{3},\quad  \mathcal{Q} \geq 0\right\}\\
&=\bigcup_{(\alpha,\beta)\in \left[0,\frac{1}{2}\right]\times (0,1]} \left\{\left(\alpha,Z_2,Z_3,\frac{6\Delta}{k+l}\beta \right)\mid 0\leq Z_2+Z_3\leq \omega(\alpha,\beta)\right\}
\end{split}
\end{equation}

On the other hand, for each fixed $(\alpha,\beta)\in \left[0,\frac{1}{2}\right]\times (0,1]$, we have 
\begin{equation}
\begin{split}
{q_2}_{(\alpha,\beta)}(0)&=2\alpha^2-2\alpha+2> 0,\\
{q_2}_{(\alpha,\beta)}\left(\frac{2}{3}\right)&=\frac{2}{9}\beta^2-\frac{2}{9}-10\alpha+2\alpha^2\leq 0.
\end{split}
\end{equation}
Let $\zeta(\alpha,\beta)$ be the smallest non-negative root of ${q_2}_{(\alpha,\beta)}$. It is clear that $\zeta(\alpha,\beta)\in \left[0,\frac{2}{3}\right]$. We then have 
$$\bigcup_{(\alpha,\beta)\in \left[0,\frac{1}{2}\right]\times (0,1]} \left\{\left(\alpha,Z_2,Z_3,\frac{6\Delta}{k+l}\beta \right)\mid 0\leq Z_2+Z_3\leq \zeta(\alpha,\beta)\right\}\subset \{\mathcal{A}\geq 0\}.
$$

Hence \eqref{eqn: Q in A} boils down to proving the following proposition.
\begin{proposition}
\label{each slice is contained}
For each fixed $(\alpha,\beta)\in \left[0,\frac{1}{2}\right]\times (0,1]$,
$$\omega(\alpha,\beta)\leq \zeta(\alpha,\beta).$$
Moreover, the equality holds only at $(\alpha,\beta)=(0,1)$.
\end{proposition}
\begin{proof}
It is clear that $\omega(0,1)=\zeta(0,1)=\frac{2}{3}$. By implicit differentiation, we have 
$$
\left.\frac{\partial \omega}{\partial\alpha}\right|_{(\alpha,\beta)=(0,1)}=-3, \quad \left.\frac{\partial \omega}{\partial\beta}\right|_{(\alpha,\beta)=(0,1)}=\frac{1}{3},\quad 
\left.\frac{\partial \zeta}{\partial\alpha}\right|_{(\alpha,\beta)=(0,1)}=-3, \quad \left.\frac{\partial \zeta}{\partial\beta}\right|_{(\alpha,\beta)=(0,1)}=\frac{2}{15}.
$$
$$
\left.\frac{\partial^2 \omega}{\partial\alpha^2}\right|_{(\alpha,\beta)=(0,1)}=24, \quad \left.\frac{\partial^2 \zeta}{\partial\alpha^2}\right|_{(\alpha,\beta)=(0,1)}=\frac{141}{5}>24
$$
Hence $\omega\leq \zeta$ in a small neighborhood around $(\alpha,\beta)=(0,1)$ in $\left[0,\frac{1}{2}\right]\times (0,1]$.

Suppose $\omega=\zeta$ at some point in $\left[0,\frac{1}{2}\right]\times (0,1]$. Then ${q_1}_{(\alpha,\beta)}$ and ${q_2}_{(\alpha,\beta)}$ have common roots. Hence the resultant 
\scriptsize
\begin{equation}
\begin{split}
&r_{(\alpha,\beta)}(q_1,q_2)\\
&=\det\begin{bmatrix}
\frac{9\beta^2}{8}&0&\frac{3\beta}{4}&0&0&0\\
0&\frac{9\beta^2}{8}&-\left(2+\frac{3}{2}\alpha\beta\right)&\frac{3\beta}{4}&0&0\\
-2&0&\frac{3}{4}\alpha^2\beta-2\alpha+1&-\left(2+\frac{3}{2}\alpha\beta\right)&\frac{3\beta}{4}&0\\
-(12\alpha+2)&-2&0&\frac{3}{4}\alpha^2\beta-2\alpha+1&-\left(2+\frac{3}{2}\alpha\beta\right)&\frac{3\beta}{4}\\
2\alpha^2-2\alpha+2&-(12\alpha+2)&0&0&\frac{3}{4}\alpha^2\beta-2\alpha+1&-\left(2+\frac{3}{2}\alpha\beta\right)\\
0&2\alpha^2-2\alpha+2&0&0&0&\frac{3}{4}\alpha^2\beta-2\alpha+1
\end{bmatrix}\\
&= 27\beta^2\bigg[\frac{243}{16384} \alpha^8 \beta^6+\left(-\frac{81}{512}\alpha^7+\frac{81}{1024}\alpha^6\right) \beta^5+\left(\frac{81}{256} \alpha^6-\frac{189}{256} \alpha^5+\frac{27}{128}\alpha^4\right) \beta^4\\
&\quad +\left(-\frac{153}{32}  \alpha^5 -\frac{27}{32} \alpha^4+\frac{27}{16} \alpha^3-\frac{9}{32} \alpha^2 \right) \beta^3+\left( 18 \alpha^4-\frac{39}{2} \alpha^3+\frac{159}{16} \alpha^2-\frac{9}{4} \alpha+\frac{3}{16}\right) \beta^2\\
&\quad +\left(21 \alpha^3-15 \alpha^2+\frac{17}{4} \alpha-\frac{7}{16}\right) \beta+6 \alpha^2-2 \alpha+\frac{1}{4}\bigg]
\end{split}
\end{equation}
\normalsize
vanishes for some $(\alpha,\beta)$. Straightforward observation in the appendix shows that $r(\alpha,\beta)$ is non-negative for any $(\alpha,\beta)\in \left[0,\frac{1}{2}\right]\times (0,1]$ and it vanishes only at $(0,1)$. Hence $\omega\leq \zeta$ and the equality holds only at $(0,1)$. The proof is complete.
\end{proof}

\begin{figure}[h!] 
\centering
  \includegraphics[clip,trim=14cm 2cm 8cm 5cm,width=0.6\linewidth]{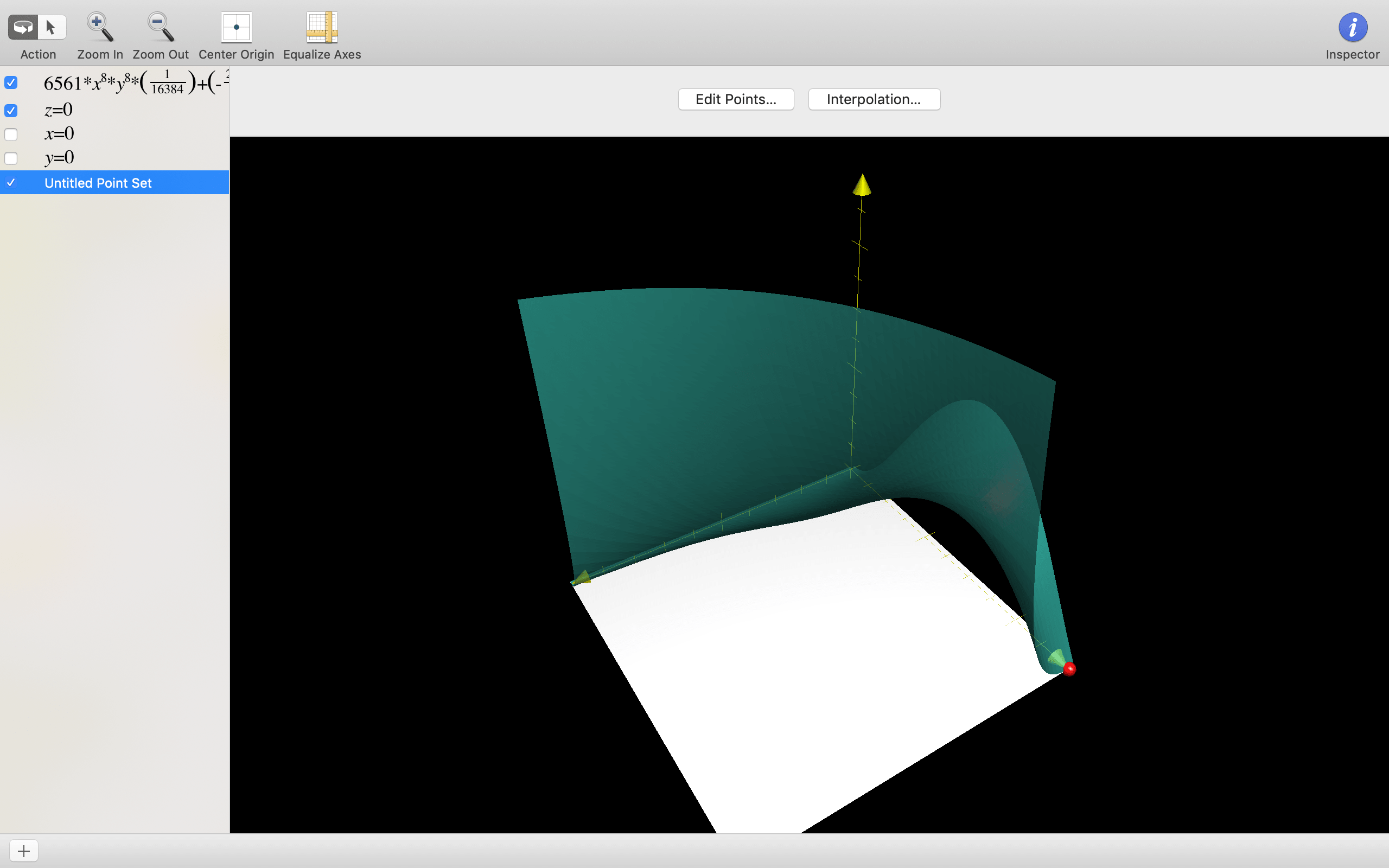}
\caption{The graph of $r(\alpha,\beta)$ defined on $\left[0,\frac{1}{2}\right]\times (0,1]$ is non-negative by straightforward observation. The dot at the right corner is $(0,1)$. The rigorous proof of $r(\alpha,\beta)$ being non-negative is presented in the appendix.}
\label{fig: resultant}
\end{figure}

In summary, we have the following lemma.
\begin{lemma}
\label{lem: invariant set}
The set $\check{\mathcal{S}}$ is invariant.
\end{lemma}
\begin{proof}
We have the following chain of inclusions:
\begin{equation}
\begin{split}
\check{\mathcal{S}}&=\check{\mathcal{S}}\cap \{\mathcal{Q}\geq 0\}\quad \text{by Proposition \ref{prop: included in Q}}\\
&\subset \left\{0\leq Z_1\leq \frac{1}{2},\quad 0\leq Z_2+Z_3\leq \frac{2}{3},\quad 0< Z_4\leq \frac{6\Delta}{k+l},\quad \mathcal{Q}\geq 0 \right\}\quad \text{by Proposition \ref{prop: S is compact}}\\
&= \bigcup_{(\alpha,\beta)\in \left[0,\frac{1}{2}\right]\times (0,1]} \left\{\left(\alpha,Z_2,Z_3,\frac{6\Delta}{k+l}\beta \right)\mid 0\leq Z_2+Z_3\leq \omega(\alpha,\beta)\right\}\\
&\subset \bigcup_{(\alpha,\beta)\in \left[0,\frac{1}{2}\right]\times (0,1]} \left\{\left(\alpha,Z_2,Z_3,\frac{6\Delta}{k+l}\beta \right)\mid 0\leq Z_2+Z_3\leq \zeta(\alpha,\beta)\right\}\quad \text{by Proposition \ref{each slice is contained}}\\
&\subset \left\{\mathcal{A}\geq 0\right\}
\end{split}.
\end{equation}
Therefore we have 
\begin{equation}
\begin{split}
\check{\mathcal{S}}= \check{\mathcal{S}}\cap \{\mathcal{Q}\geq 0\} \subset \check{\mathcal{S}}\cap \{\mathcal{A}\geq 0\} \subset \check{\mathcal{S}}\cap \{\mathcal{G}-X_4\geq 0\}
\end{split}
\end{equation}
by Proposition \ref{prop: A included in G-X4>0}. In other words, inequality 
$$
\mathcal{G}-X_4\geq \mathcal{A}\geq \mathcal{Q}\geq 0
$$
holds in $\check{\mathcal{S}}$. Hence we have
\begin{equation}
\begin{split}
&\left.\langle\nabla Z_4, V\rangle\right|_{Z_4=\frac{6\Delta}{k+l}}=-\frac{6\Delta}{k+l}(\mathcal{G}-X_4)\leq 0.
\end{split}
\end{equation}
If non-transverse crossings emerge on an integral curve in $\check{\mathcal{S}}$, then $Z_4=\frac{6\Delta}{k+l}$ and $\mathcal{G}-X_4= \mathcal{A}=\mathcal{Q}=0$ at each crossing point. By Proposition \ref{prop: included in Q} and Proposition \ref{each slice is contained}, we know that $Z_1=0$ and $Z_2=Z_3$ at the crossing point, which force it to be the critical point $P_0^{(k+l)}$. Non-transverse crossing is hence excluded and the compact set $\check{\mathcal{S}}$ is invariant. The proof is complete.
\end{proof}

\subsection{Global Metrics on $M_{k,l}^{(k)}$}
\label{subsec: Tk Tl}
One may intend to use the method above to prove the global existence of both $\gamma^{(k)}_{(s_1,s_2)}$ and $\gamma^{(l)}_{(s_1,s_2)}$. However, by the following proposition, the method is only effective for the case on $M_{k,l}^{(k)}$.
\begin{proposition}
\label{prop: Tk is compact}
The set 
\begin{equation}
\mathcal{T}_k:=\mathcal{C}^-_{\spin(7)}\cap \left\{ 0\leq Z_4\leq \frac{6\Delta}{k}\right\}
\end{equation}
consists of two connected component, with one of the component being compact.
\end{proposition}
\begin{proof}
We first prove the statement for $\mathcal{T}_k$.
By \eqref{eqn: new opposite spin(7) equation}, we have 
\begin{equation}
\label{eqn: compactness for Tk}
\begin{split}
0&=1-2(Z_1+Z_2+Z_3)-\left(\frac{k+l}{2\Delta}Z_2Z_3-\frac{l}{2\Delta}Z_1Z_3-\frac{k}{2\Delta}Z_1Z_2\right)Z_4\\
&= 1-2(Z_1+Z_2)-Z_3\left(2-\frac{l}{2\Delta}Z_1Z_4\right)-\frac{k+l}{2\Delta}Z_2Z_3Z_4+\frac{k}{2\Delta}Z_1Z_2Z_4\\
&\leq  1-2(Z_1+Z_2)-Z_3\left(2-3\frac{l}{k}Z_1\right)+3Z_1Z_2\quad \text{since  $Z_4\leq \frac{6\Delta}{k}$ in $\mathcal{T}_k$}\\
&\leq  1-2(Z_1+Z_2)-Z_3\left(2-3Z_1\right)+3Z_1Z_2\quad \text{by the setting $k>l$}\\
&\leq  1-2(Z_1+Z_2)-Z_3\left(2-3Z_1\right)+3Z_1Z_2+\frac{3}{4}(-Z_1+Z_2+Z_3)^2\\
&=\frac{3}{4}(Z_1+Z_2+Z_3-2)\left(Z_1+Z_2+Z_3-\frac{2}{3}\right)
\end{split}.
\end{equation}
Hence either $Z_1+Z_2+Z_3\leq \frac{2}{3}$ or $Z_1+Z_2+Z_3\geq 2$ in $\mathcal{T}_k$. Hence we can write $\mathcal{T}_k=\check{\mathcal{T}_k}\sqcup\hat{\mathcal{T}_k}$, where $\check{\mathcal{T}_k}$ is the compact component
$$
\mathcal{C}^-_{\spin(7)}\cap \left\{0\leq Z_1+Z_2+Z_3\leq \frac{2}{3},\quad 0\leq Z_4\leq \frac{6\Delta}{k}\right\}.
$$
Since $P_0^{(k)}\in\check{\mathcal{T}_k}$, the compact component is non-empty. The proof is complete.
\end{proof}

\begin{remark}
\label{rem: sharper Z1}
Since $Z_1+Z_2+Z_3\leq \frac{2}{3}$ in $\check{\mathcal{T}_k}$, the estimate for $Z_1$ can be sharper. From the third last line of  computation \eqref{eqn: compactness for Tk}, we have 
\begin{equation}
\begin{split}
0&\leq 1-2(Z_1+Z_2)-Z_3\left(2-3Z_1\right)+3Z_1Z_2\\
&\leq 1-2(Z_1+Z_2)+3Z_1Z_2\\
&=1-2Z_1+(3Z_1-2)Z_2\\
&\leq 1-2Z_1
\end{split}.
\end{equation}
Hence $Z_1\leq \frac{1}{2}$ in $\check{\mathcal{T}_k}$
\end{remark}

\begin{figure}[h!] 
\centering
  \includegraphics[clip,trim=14cm 2cm 8cm 5cm,width=0.6\linewidth]{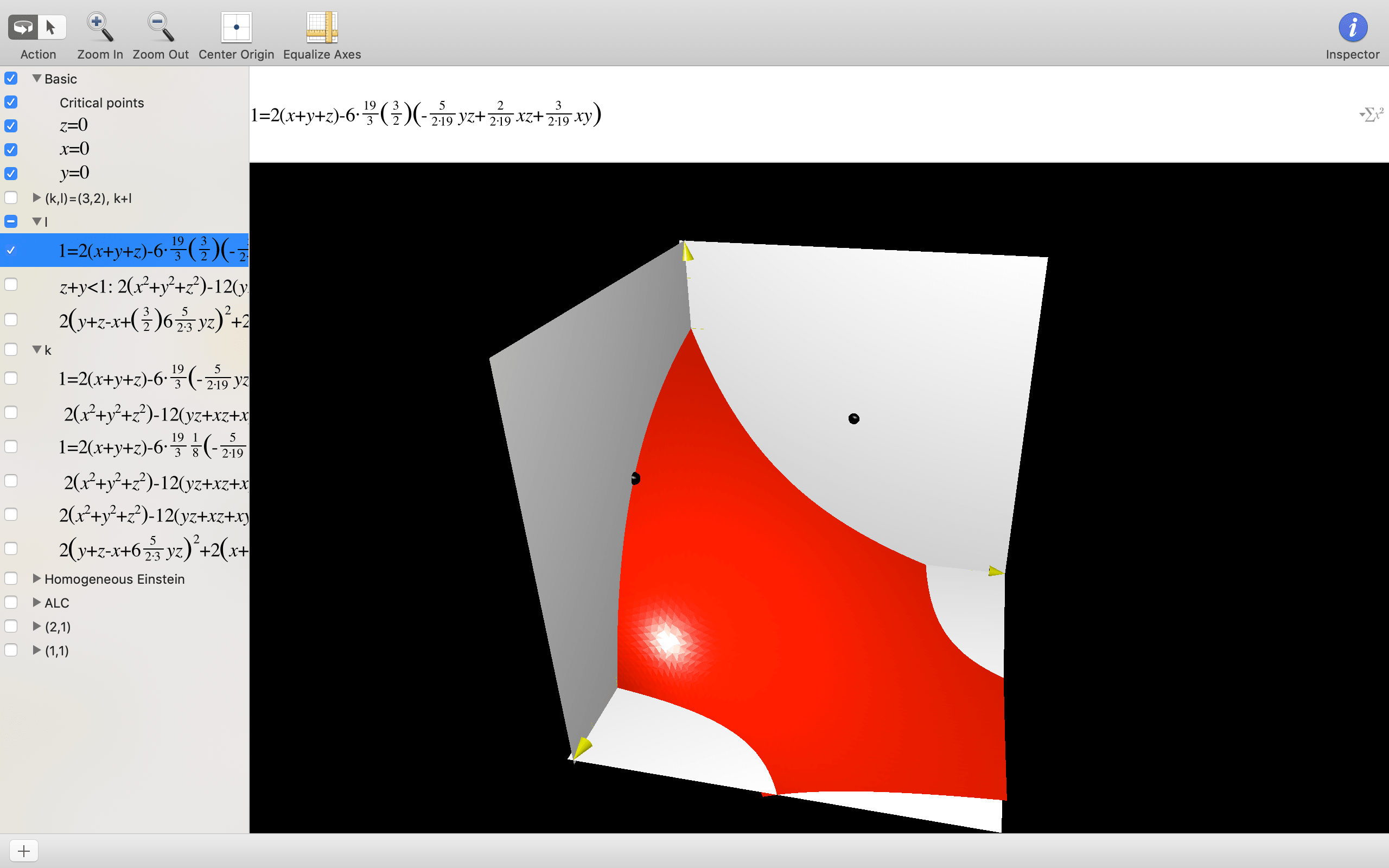}
\caption{The set $\mathcal{T}_l$ with fixed $Z_4=\frac{6\Delta}{l}$ in the case $(k,l)=(3,2)$ is shown above. It is connected and non-compact. More estimates are needed to prove the global existence on $M_{k,l}^{(l)}$.}
\label{fig: Tl}
\end{figure}

\begin{figure}[h!] 
\centering
\begin{subfigure}{.4\textwidth}
  \centering 
  \includegraphics[clip,trim=12cm 3cm 8cm 5cm,width=1\linewidth]{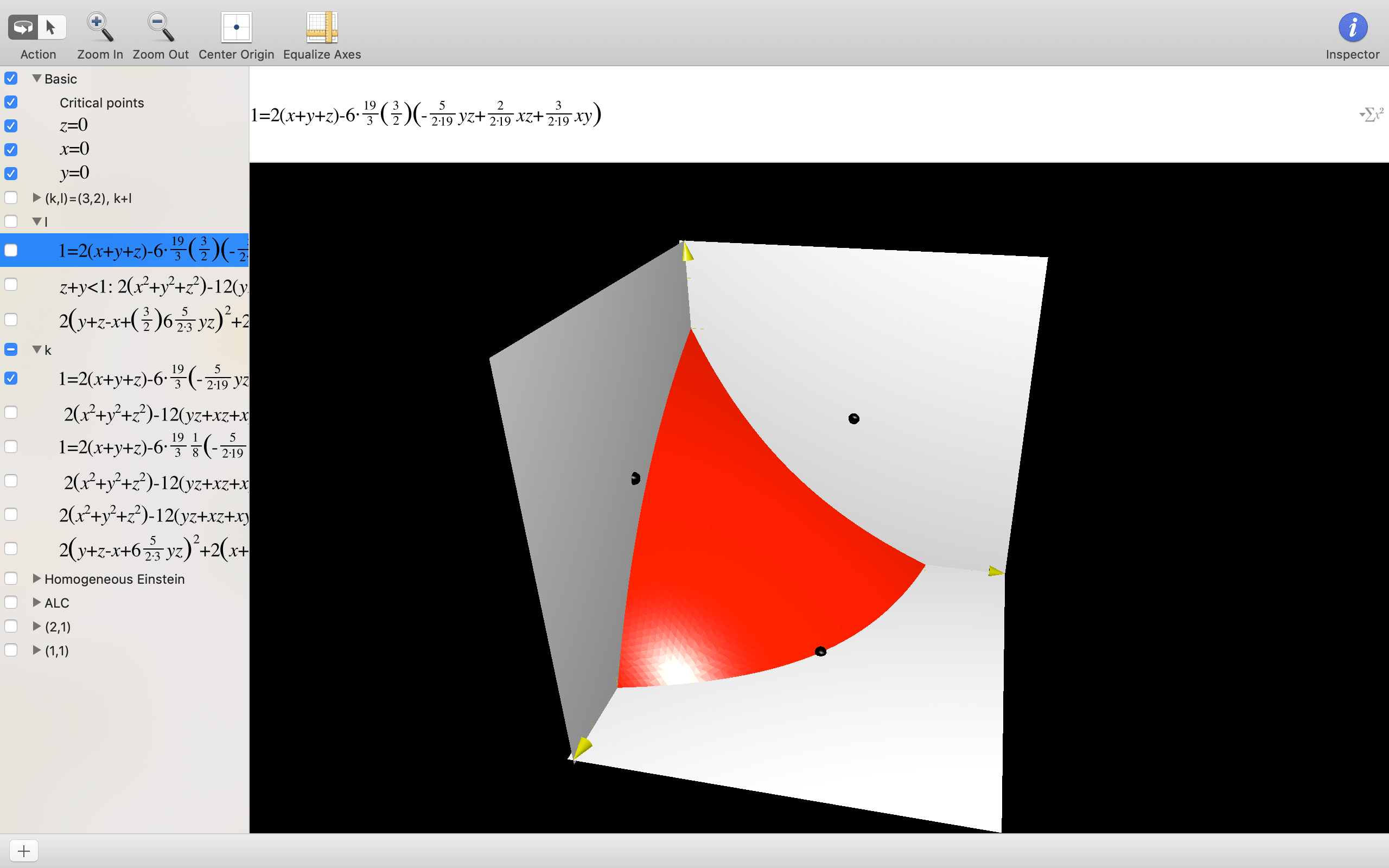}
\end{subfigure}
\begin{subfigure}{.4\textwidth}
  \centering 
  \includegraphics[clip,trim=12cm 3cm 8cm 5cm,width=1\linewidth]{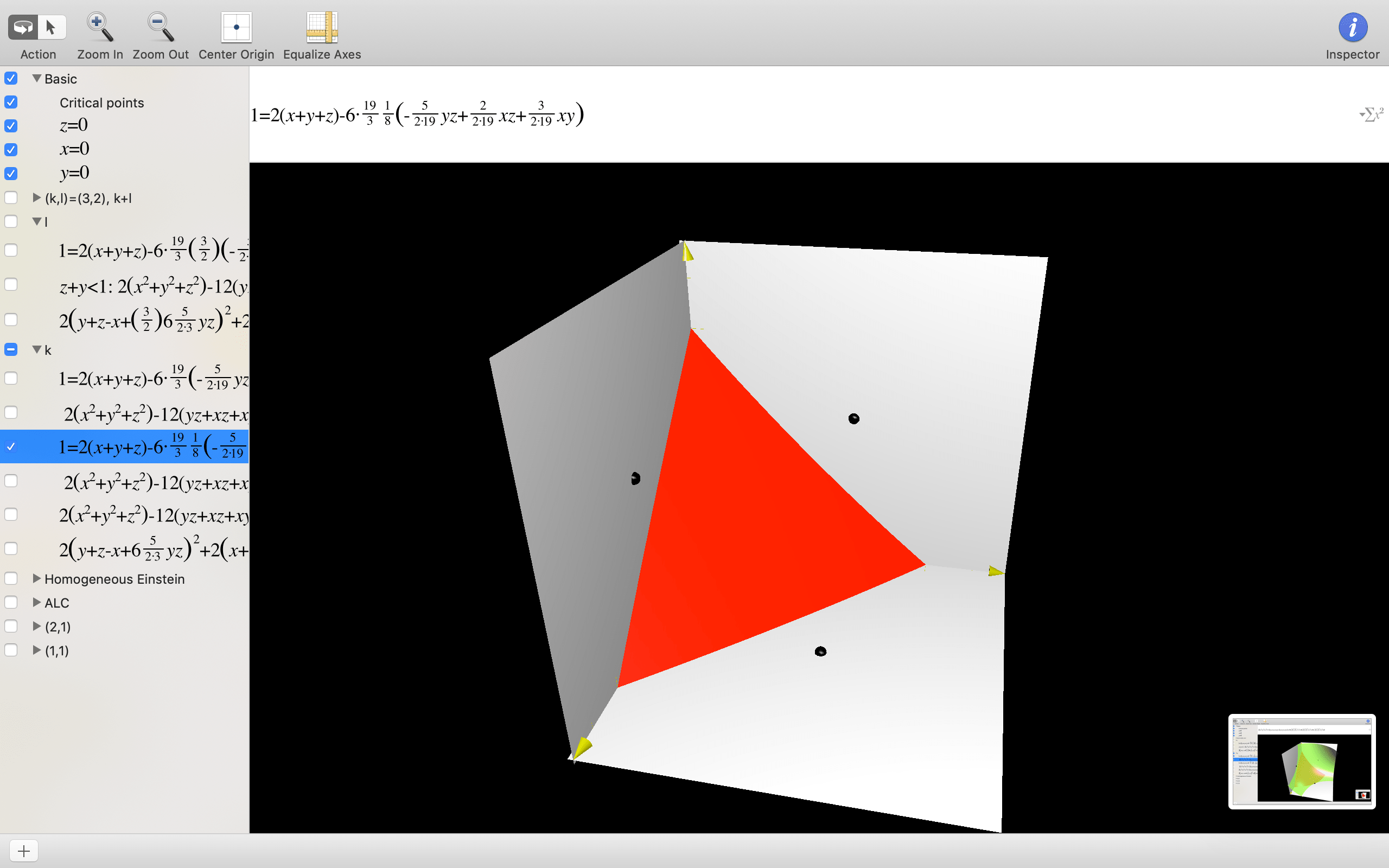}
\end{subfigure}\\
\begin{subfigure}{.4\textwidth}
  \centering
  \includegraphics[clip,trim=12cm 3cm 8cm 5cm,width=1\linewidth]{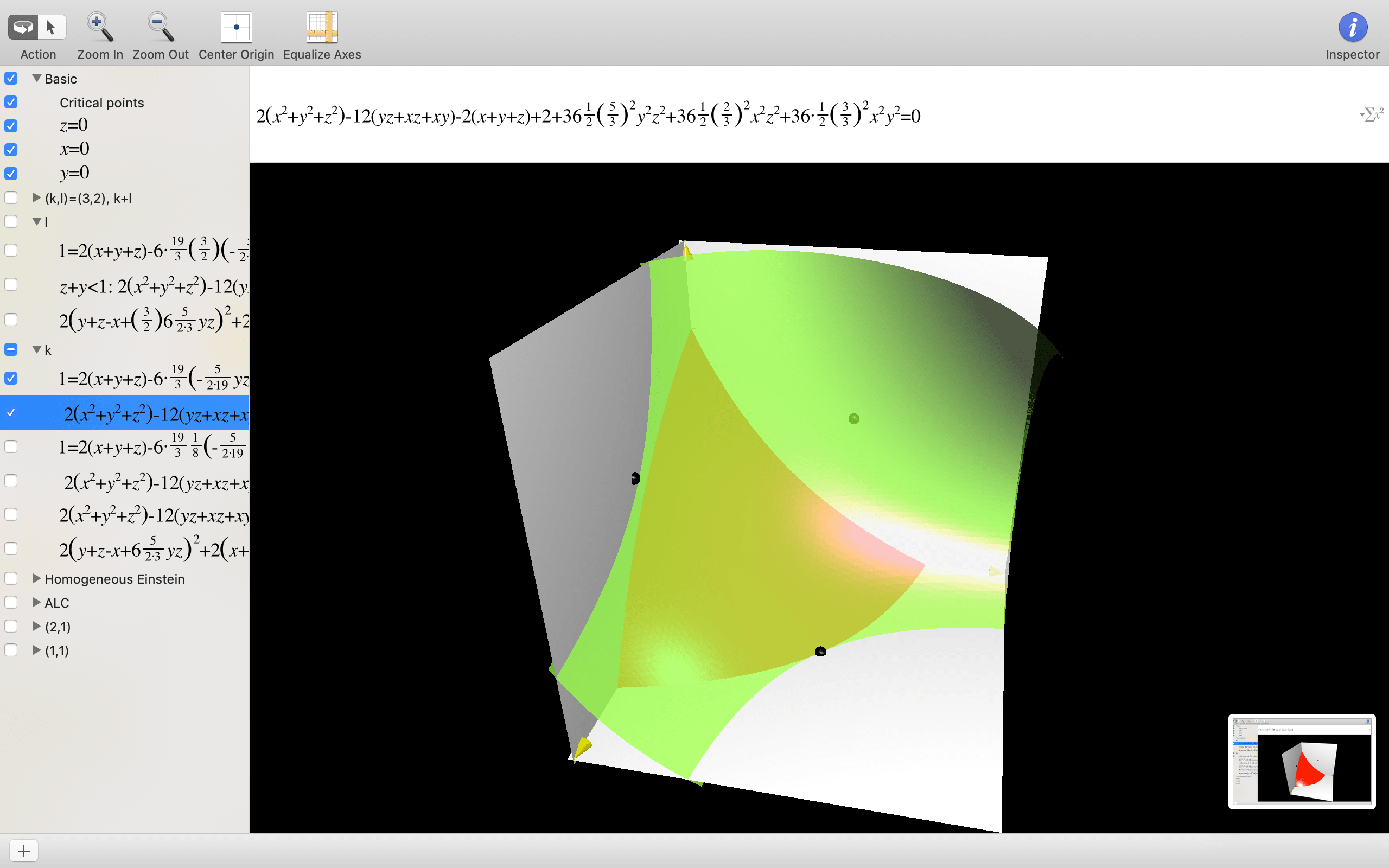}
      \caption{$Z_4=\frac{6\Delta}{k}$}
\end{subfigure}
\begin{subfigure}{.4\textwidth}
  \centering
  \includegraphics[clip,trim=12cm 3cm 8cm 5cm,width=1\linewidth]{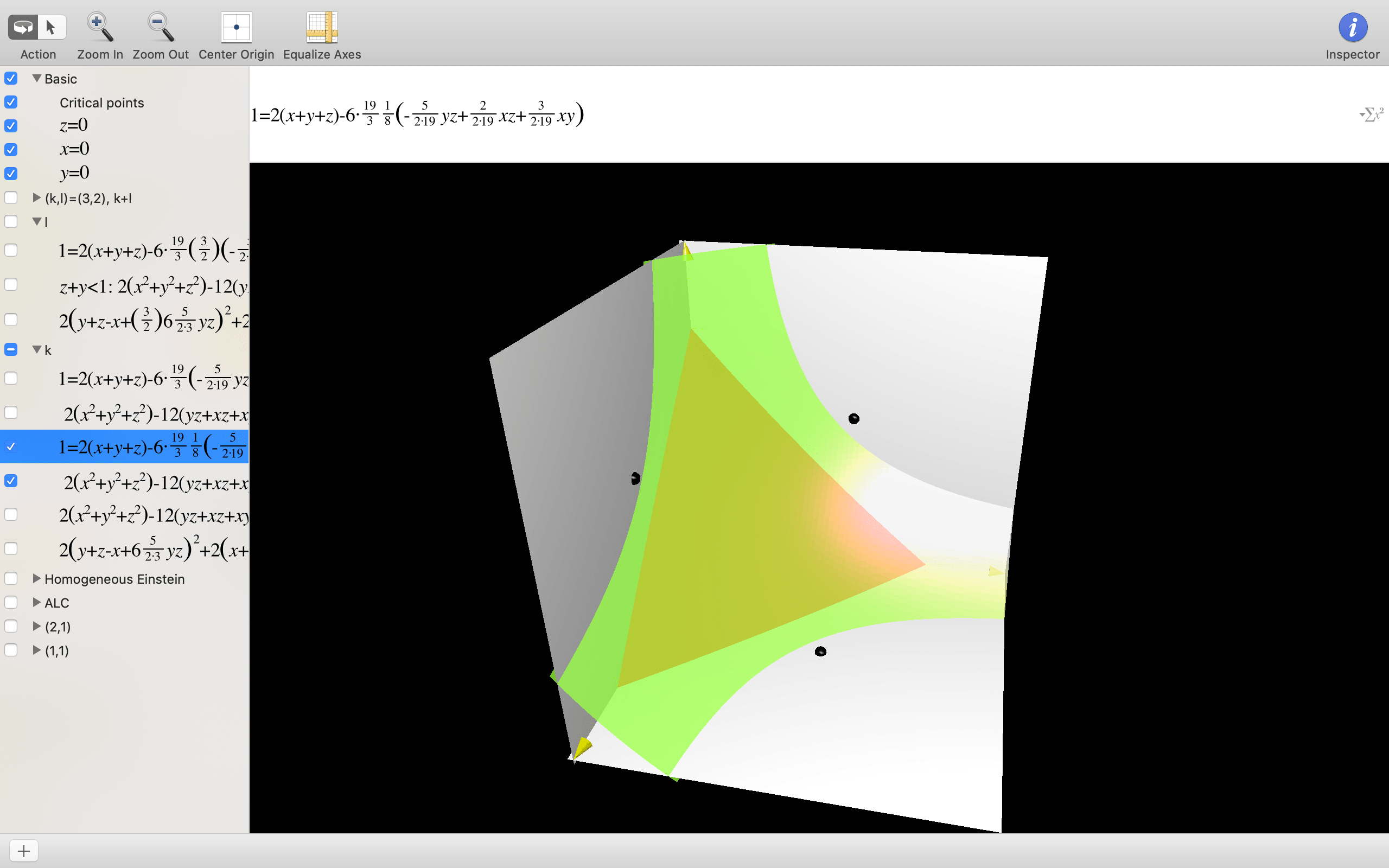}
  \caption{$Z_4=\frac{1}{8}\frac{6\Delta}{k}$}
\end{subfigure}
\caption{The $\spin(7)$ conditions for $\mathcal{T}_k$ and $\mathcal{T}_l$ have the same chirality. In Figure \ref{fig: Tl}, the set is non-compact. With smaller $Z_4$, the set ``shrinks'' to a compact one. Similar to the case for $\check{\mathcal{S}}$, the figures above show $\check{\mathcal{T}_k}\subset \check{\mathcal{T}_k}\cap \{\mathcal{G}-X_4\geq 0\}$, with $(k,l)=(3,2)$.
Figures in the first row demonstrate  the compact set $\check{\mathcal{T}_k}$ for fixed $Z_4\leq \frac{6\Delta}{k}$. Figures in the second row demonstrate the set $\check{\mathcal{T}_k}\cap \{\mathcal{G}-X_4=0\}$ for fixed $Z_4\leq \frac{6\Delta}{k}$.}
\label{fig: Tk}
\end{figure}

\begin{remark}
On the other hand, consider the set
\begin{equation}
\mathcal{T}_l:=\mathcal{C}^-_{\spin(7)}\cap \left\{0\leq Z_4\leq \frac{6\Delta}{l}\right\}.
\end{equation}
The estimate \eqref{eqn: compactness for Tk} does not hold by the setting $k>l$. As shown in Figure \ref{fig: Tl}, $\mathcal{T}_l$ is connected and non-compact. Therefore, even if we are able to prove the invariance of $\mathcal{T}_l$, the upper bound $\frac{6\Delta}{l}$ for $Z_4$ may seems to be too large to start with.
\end{remark}

Now that the construction boils down to showing the inclusion
$$
\check{\mathcal{T}_k}\subset \check{\mathcal{T}_k}\cap \{\mathcal{G}-X_4\geq 0\}.
$$ 
It turns out that for the case $M_{k,l}^{(k)}$, the analysis is much more delicate than the case $M_{k,l}^{(k+l)}$. We first state the following proposition.
\begin{proposition}
\label{prop: smaller Tk}
The set $\mathcal{C}^-_{\spin(7)}\cap \{Z_1-Z_2\geq 0\}\cap \{Z_2-Z_3\geq 0\}$ is invariant.
\end{proposition}
\begin{proof}
Consider
\begin{equation}
\begin{split}
&\left.\langle\nabla (Z_1-Z_2), V\rangle\right|_{Z_1-Z_2=0}\\
&=2Z_1(X_1-X_2)\\
&=4(Z_2-Z_1)+\frac{k+l}{\Delta}Z_2Z_3Z_4+\frac{l}{\Delta}Z_1Z_3Z_4 \quad \text{replace $X_1$ and $X_2$ using \eqref{eqn: new opposite spin(7) equation}}\\
&=\frac{k+l}{\Delta}Z_2Z_3Z_4+\frac{l}{\Delta}Z_1Z_3Z_4\\
&\geq 0
\end{split}
\end{equation}
and 
\begin{equation}
\begin{split}
&\left.\langle\nabla (Z_2-Z_3), V\rangle\right|_{Z_2-Z_3=0}\\
&=2Z_2(X_2-X_3)\\
&=4(Z_3-Z_2)+\frac{k}{\Delta}Z_1Z_2Z_4-\frac{l}{\Delta}Z_1Z_3Z_4 \quad \text{replace $X_2$ and $X_3$ using \eqref{eqn: new opposite spin(7) equation}}\\
&=\frac{k-l}{\Delta}Z_1Z_2Z_4\quad \text{since $Z_2=Z_3$}\\
&\geq 0
\end{split}.
\end{equation}
The proof is complete.
\end{proof}

We redefine $\check{\mathcal{T}_k}$ as $$\check{\mathcal{T}_k}:=
\mathcal{C}^-_{\spin(7)}\cap \left\{0\leq Z_1+Z_2+Z_3\leq \frac{2}{3},\quad 0\leq Z_4\leq \frac{6\Delta}{k}\right\}\cap \{Z_1-Z_2\geq 0\}\cap \{Z_2-Z_3\geq 0\}.$$ 
The problem boils down to prove the separation between $\check{\mathcal{T}_k}$ and $\check{\mathcal{T}_k}\cap \{\mathcal{G}-X_4=0\}$. For the case of $M_{k,l}^{(k)}$, the algebraic surface for $\mathcal{G}-X_4$ is too ``close'' to $\check{\mathcal{T}_k}$. It is difficult to find algebraic surfaces such as $\mathcal{A}=0$ and $\mathcal{Q}=0$ in the case of $M_{k,l}^{(k+l)}$ to separate $\check{\mathcal{T}_k}$ and $\check{\mathcal{T}_k}\cap \{\mathcal{G}-X_4=0\}$. Hence in this case, we take algebraic surface given by \eqref{eqn: new opposite spin(7) conservation} and the one from  \eqref{eqn: new conservation} and Proposition \ref{prop: X_4 and Zi}:
$$\mathcal{G}-X_4=1-\mathcal{R}_s-X_4=\mathcal{R}_s+2-2(Z_1+Z_2+Z_3).$$
Specifically, define 
$$
\mathcal{P}(Z_1,Z_2,Z_3,Z_4):= 1-2(Z_1+Z_2+Z_3)-\frac{k+l}{2\Delta}Z_2Z_3Z_4+\frac{l}{2\Delta}Z_1Z_3Z_4+\frac{k}{2\Delta}Z_1Z_2Z_4.
$$
And define 
\begin{equation}
\begin{split}
&\mathcal{B}(Z_1,Z_2,Z_3,Z_4)\\
&=2(Z_1^2+Z_2^2+Z_3^2)-12(Z_2Z_3+Z_1Z_2+Z_1Z_3)+2-2(Z_1+Z_2+Z_3)\\
&\quad +\frac{1}{2}\left(\frac{k+l}{\Delta}\right)^2Z_2^2Z_3^2Z_4^2+\frac{1}{2}\left(\frac{l}{\Delta}\right)^2Z_1^2Z_3^2Z_4^2+\frac{1}{2}\left(\frac{k}{\Delta}\right)^2Z_1^2Z_2^2Z_4^2
\end{split}
\end{equation}

With each fixed  $Z_2=\alpha Z_1$, $Z_3=\beta Z_1$ and $Z_4=\frac{6\Delta}{k}\delta$ with $(\alpha,\beta,\delta)\in [0,1]\times [0,1]\times (0,1]$, we define slices
\begin{equation}
\begin{split}
{p_1}_{(\alpha,\beta,\delta)}(Z_1)&:=\mathcal{P}\left(Z_1,\alpha Z_1,\beta Z_1,\frac{6\Delta}{k}\delta\right)\\
&=1-2(Z_1+Z_2+Z_3)-\frac{k+l}{2\Delta}Z_2Z_3Z_4+\frac{l}{2\Delta}Z_1Z_3Z_4+\frac{k}{2\Delta}Z_1Z_2Z_4\\
&=\left(-3\frac{k+l}{k} \alpha\beta\delta +3\frac{l}{k}\beta\delta +3\alpha\delta \right)Z_1^2-2(1+\alpha+\beta)Z_1+1\\
{p_2}_{(\alpha,\beta,\delta)}(Z_1)&:=\mathcal{B}\left(Z_1,\alpha Z_1,\beta Z_1,\frac{6\Delta}{k}\delta\right)\\
&=\left(18\left(\frac{k+l}{k}\right)^2\alpha^2\beta^2\delta^2+18\left(\frac{l}{k}\right)^2\beta^2\delta^2+18\alpha^2\delta^2 \right)Z_1^4\\
&\quad +( 2(1+\alpha^2+\beta^2)-12(\alpha\beta+\alpha+\beta))Z_1^2-2(1+\alpha+\beta)Z_1+2
\end{split}.
\end{equation}
Each pair of slice with fixed $(\alpha,\beta,\delta)$ gives two polynomials in $Z_1$ for us to compare. Since 
$$
{p_1}_{(\alpha,\beta,\delta)}(0)=1,\quad {p_1}_{(\alpha,\beta,\delta)}\left(\frac{1}{2}\right)=-\frac{3}{4}\frac{k+l}{k}\alpha^2\beta^2\delta^2-\left(1-\frac{3}{4}\delta\right)\alpha-\left(1-\frac{3}{4}\frac{l}{k}\delta\right)\beta\leq 0
$$
there is a real positive root for each ${p_1}_{(\alpha,\beta,\delta)}(Z_1)$ in $ \left[0,\frac{1}{2}\right]$. It is clear that for each $(\alpha,\beta,\delta)$, $p_1$ has two real roots, with one be no larger than $\frac{1}{2}$ and the other one no smaller.
Let $\xi(\alpha,\beta,\delta)$ be the smaller positive root for each ${p_1}_{(\alpha,\beta,\delta)}(Z_1)$. Since ${p_2}_{(\alpha,\beta,\delta)}(0)=2>0$, we let $\sigma(\alpha,\beta,\delta)$ denote the smallest positive root for each ${p_2}_{(\alpha,\beta,\delta)}(Z_1)$. Suppose $\sigma(\alpha,\beta,\delta)$ does not exist, then $\mathcal{B}> 0$ for that particular slice. 

\begin{proposition}
\label{prop: sigma larger xi}
For each fixed $(\alpha,\beta,\delta)\in \left[0,1\right]\times [0,1]\times (0,1]$ with $\alpha\geq \beta$ such that $\sigma(\alpha,\beta,\delta)$ exists, we have
$$\xi(\alpha,\beta,\delta)\leq \sigma(\alpha,\beta,\delta).$$
Moreover, the equality holds only at $(\alpha,\beta,\delta)=(1,0,1)$.
\end{proposition}
\begin{proof}
It is clear that $\xi(1,0,1)=\sigma(1,0,1)=\frac{1}{3}$. By implicit differentiation, we have 
$$
\left.\frac{\partial \xi}{\partial\alpha}\right|_{(\alpha,\beta,\delta)=(1,0,1)}=-\frac{1}{6}, 
\quad \left.\frac{\partial \xi}{\partial\beta}\right|_{(\alpha,\beta,\delta)=(1,0,1)}=-\frac{1}{2},
\quad 
\left.\frac{\partial \xi}{\partial\delta}\right|_{(\alpha,\beta,\delta)=(1,0,1)}=\frac{1}{6},
$$
$$
\left.\frac{\partial \sigma}{\partial\alpha}\right|_{(\alpha,\beta,\delta)=(1,0,1)}=-\frac{1}{6}, 
\quad \left.\frac{\partial \sigma}{\partial\beta}\right|_{(\alpha,\beta,\delta)=(1,0,1)}=-\frac{1}{2},
\quad 
\left.\frac{\partial \sigma}{\partial\delta}\right|_{(\alpha,\beta,\delta)=(1,0,1)}=\frac{1}{15}.
$$
Hence $(\sigma-\xi)(\alpha,\beta,\delta)\geq 0$ initially if the direction of derivative as negative $\delta$ component. If the direction of derivative is in the $\alpha\beta$-plane, we consider the function $(\sigma-\xi)(\alpha,\beta):=(\sigma-\xi)(\alpha,\beta,1)$. The Hessian of $(\sigma-\xi)(\alpha,\beta)$ at $(\alpha,\beta)=(1,0)$ is then
$$
\begin{bmatrix}
\frac{11}{60}&\frac{1}{2}-\frac{5k-2l}{12k}\\
\frac{1}{2}-\frac{5k-2l}{12k}&\frac{19}{60}+\frac{(k+l)^2+l^2}{15k^2}
\end{bmatrix}.
$$
The determinant is $\frac{19k^2-kl-l^2}{300k^2}\geq \frac{17}{300}$. Hence $\xi\leq \sigma$ in a small neighborhood around $(\alpha,\beta,\delta)=(1,0,1)$ in $\left[0,1\right]\times [0,1]\times (0,1]$.

Suppose $\sigma=\xi$ for some $(\alpha,\beta,\delta)$, then the resultant of of $p_1$ and $p_2$ vanishes at that point. The formula of $r_{(\alpha,\beta,\delta)}(p_1,p_2)$ is presented in the Appendix. With the help of Maple, we know that $r_{(\alpha,\beta,\delta)}(p_1,p_2)$ is non-negative in the region $[0,1]\times[0,1]\times (0,1]$ and it only vanishes at $(1,0,1)$ if $(k,l)\neq (1,1)$. If $(k,l)=(1,1)$, then $r_{(\alpha,\beta,\delta)}(p_1,p_2)$ vanishes only at $(1,0,1)$ and $(1,1,0)$. Therefore $\sigma\geq \xi$ for all $(\alpha,\beta,\delta)\in [0,1]\times[0,1]\times (0,1]$.
\end{proof}

Hence we have the following lemma.
\begin{lemma}
\label{lem: invariant set}
The set $\check{\mathcal{T}_k}$ is invariant.
\end{lemma}
\begin{proof}
We have the following chain of inclusions:
\begin{equation}
\begin{split}
\check{\mathcal{T}_k}&\subset \left\{0\leq Z_3\leq Z_2\leq Z_1\leq \frac{1}{2},\quad 0< Z_4\leq \frac{6\Delta}{k},\quad \mathcal{P}= 0 \right\}\quad \text{by Proposition \ref{prop: Tk is compact} and Proposition \ref{prop: smaller Tk}}\\
&= \bigcup_{(\alpha,\beta,\delta)\in [0,1]\times [0,1]\times (0,1]} \left\{\left(Z_1,\alpha Z_1,\beta Z_1,\frac{6\Delta}{k}\delta \right)\mid Z_1= \xi(\alpha,\beta,\delta)\right\}\\
&\subset \bigcup_{(\alpha,\beta,\delta)\in [0,1]\times [0,1]\times (0,1]} \left\{\left(Z_1,\alpha Z_1,\beta Z_1,\frac{6\Delta}{k}\delta \right)\mid 0\leq Z_1 \leq \sigma (\alpha,\beta,\delta)\right\}\quad \text{by Proposition \ref{prop: sigma larger xi}}\\
&\subset \left\{\mathcal{G}-X_4\geq 0\right\}
\end{split}.
\end{equation}

Similar to the argument in proving Lemma \ref{lem: invariant tilde S}, the non-transverse crossing cannot emerge in $\check{\mathcal{T}_k}$ as the crossing point can only be the critical point $P_0^{(k)}$. Hence $\check{\mathcal{T}_k}$ is a compact invariant set.
\end{proof}

From Section \ref{sec: critical points}, it is clear that $\gamma^{(k+l)}_{(s_1,s_2)}$ and $\gamma^{(k)}_{(s_1,s_2)}$ are in the set $\check{\mathcal{S}}$ and $\check{\mathcal{T}_k}$ initially if $s_1,s_2>0$.
Hence we have the following lemma.
\begin{lemma}
\label{lem: generic long existing}
Metrics represented by $\gamma^{(k+l)}_{(s_1,s_2)}$ on $M_{k,l}^{(k+l)}$ with $s_1,s_2>0$ are forward complete. Metrics represented by $\gamma^{(k)}_{(s_1,s_2)}$ on $M_{k,l}^{(k)}$ with $s_1,s_2>0$ are forward complete.
\end{lemma}

The global existence mentioned in Theorem \ref{thm: main1} is proven.
Our method of proving forward completeness also works on the cohomogeneity one space with exceptional $N_{1,0}$ and $N_{1,1}$ as principal orbit. For $N_{1,0}$, we obtain two continuous 1-parameter families of $\spin(7)$ metrics with both chiralities on the same manifold $M_{1,0}^{(1)}$. They are part of the more complete family in \cite{lehmann_geometric_2020}. For $N_{1,1}$, $M_{k,l}^{(k)}$ and $M_{k,l}^{(l)}$ are equivalent. In that case, we obtain two continuous 1-parameter families of $\spin(7)$ metrics, one on $M_{1,1}^{(2)}$ and the other on $M_{1,1}^{(1)}$. However, such a family does not contain any AC metric and it is a part of the more complete family in Theorem \ref{thm: main2}.

\subsection{Global Metrics on $M_{1,1}^{(2)}$ and $M_{1,1}^{(1)}$}
In this section we prove Theorem \ref{thm: main2}. 
It turns out that for the case where $N_{1,1}$ is the principal orbit, one can have a simpler construction with a larger family of $\spin(7)$ metrics. On $M_{1,1}^{(2)}$, We recover the explicit solution in \cite{cvetic_hyper-kahler_2001} and part of the solution in \cite{bazaikin_new_2007}. On $M_{1,1}^{(1)}$, we construct a new continuous 1-parameter family of $\spin(7)$ metric that has geometric transition from the AC Calabi HyperK\"ahler metric to ALC metrics. Such a family was conjectured in \cite{kanno_spin7_2002-2}.

Recall that in \cite{cvetic_hyper-kahler_2001}, one has an explicit solution to \eqref{eqn: original spin(7) equation} with $(k,l)=(1,1)$. Specifically, one can impose 
$$
b(t)=c(t)=\frac{f(t)}{2},\quad \dot{a}=1-\frac{a^2}{b^2}
$$
for all $t$ and obtain a subsystem of \eqref{eqn: original spin(7) equation}. In the new coordinate, such a subsystem is translated to the invariant subset
$$
\mathcal{C}^+_{\spin(7)}\cap\{X_2-X_3=0,\quad Z_2-Z_3=0\}\cap\{\sqrt{Z_2Z_3}Z_4-3=0\}.
$$
Note that this is a 1-dimensional invariant set and critical points $P_{AC-2}=\left(\frac{1}{7},\frac{1}{7},\frac{1}{7},\frac{1}{7},\frac{2}{21},\frac{5}{21},\frac{5}{21},\frac{63}{5}\right)$ and $P_0^{(2)}$ are contained in it. Hence the example in \cite{cvetic_hyper-kahler_2001} is transformed to an algebraic curve in the new coordinate. Then we construct the following compact invariant set
\begin{lemma}
\label{lem: invariant tilde S}
The set $\tilde{\mathcal{S}}:=\mathcal{C}^+_{\spin(7)}\cap \{X_2=X_3,Z_2=Z_3\}\cap \{\sqrt{Z_2Z_3}Z_4-3\leq 0\}$
is compact and invariant.
\end{lemma}
\begin{proof}
Since $Z_2=Z_3$ in $\tilde{\mathcal{S}}$, we have 
\begin{equation}
\begin{split}
1&=2(Z_1+Z_2+Z_3)-\left(\frac{1}{3}Z_2Z_3Z_4-\frac{1}{6}Z_1Z_3Z_4-\frac{1}{6}Z_1Z_2Z_4\right)\\
&= 2(Z_1+2Z_2)-\left(\frac{1}{3}Z_2-\frac{1}{3}Z_1\right)Z_2Z_4\quad \text{since $Z_2=Z_3$}
\end{split}
\end{equation}
from \eqref{eqn: conservation Z spin(7)}.
Since $X_4\geq 0$ by Proposition \ref{prop: X4 not negative}, 
we know that $Z_2-Z_1\geq 0$. Hence computation above continues as
\begin{equation}
\begin{split}
1&\geq 2(Z_1+2Z_2)-\left(Z_2-\frac{1}{2}Z_1-\frac{1}{2}Z_1\right)\\
&=3Z_1+3Z_2
\end{split}.
\end{equation}
Hence $Z_1$,$Z_2$ and $Z_3$ are bounded above. Therefore $Z_4$ is bounded above by \eqref{eqn: conservation Z spin(7)}. Then all $X_i$'s are bounded. Hence $\tilde{\mathcal{S}}$ is compact.

It is clear that $\mathcal{C}^+_{\spin(7)}\cap \{X_2=X_3,Z_2=Z_3\}$ is invariant. In the subset $\mathcal{C}^+_{\spin(7)}\cap \{X_2=X_3,Z_2=Z_3\}\cap \{\sqrt{Z_2Z_3}Z_4-3\leq 0\}$, we have 
\begin{equation}
\label{eqn: derivative of sqrtZ2Z3Z4}
\begin{split}
&\left.\langle\nabla (\sqrt{Z_2Z_3}Z_4-3), V\rangle\right|_{\sqrt{Z_2Z_3}Z_4-3=0}\\
&=\sqrt{Z_2Z_3}Z_4(X_4-X_1)\\
&=3(X_4-X_1)\\
&=3\left(\left(\frac{1}{3}Z_2^2Z_4-\frac{1}{3}Z_1Z_2Z_4\right)-\left(2Z_2-Z_1-\frac{1}{3}Z_2^2Z_4\right)\right)\\
&=\left(Z_2Z_4-3\right)(2Z_2-Z_1)\\
&\leq 0
\end{split}
\end{equation}
Hence the proof is complete.
\end{proof}
By \eqref{eqn: linearized solution k+l spin(7)}, we know that $\gamma^{(2)}_{(s_1,s_2)}$ is tangent to $\partial \tilde{\mathcal{S}}$ if $(s_1,s_2)=\left(-\frac{3}{\sqrt{10}},\frac{1}{\sqrt{10}}\right)$ and $\gamma^{(2)}_{(s_1,s_2)}$ is in the interior of $\tilde{\mathcal{S}}$ initially if $s_1>-\frac{3}{\sqrt{10}}$. Hence $\left\{\gamma^{(2)}_{(s_1,s_2)}\mid (s_1,s_2)\in\mathbb{S}^1, s_1\geq -\frac{3}{\sqrt{10}}\right\}$ is a 1-parameter family of $\spin(7)$ metrics on $M_{1,1}^{(2)}$, and $\gamma^{(2)}_{\left(-\frac{3}{\sqrt{10}},\frac{1}{\sqrt{10}}\right)}$ is an AC metric.

\begin{remark}
\label{rem: A8}
Note that for $N_{1,1}$, one can easily show that $\mathcal{C}_{RF}\cap\{X_2=X_3,Z_2=Z_3\}$ is invariant. The restricted system of \eqref{eqn: new Einstein equation} on $\mathcal{C}_{RF}\cap\{X_2=X_3,Z_2=Z_3\}$ is essentially the same as the one that appears in \cite{chi_einstein_2020}, where a 2-parameter family of non-positive Einstein metrics are constructed, with all Ricci-flat metrics being $\spin(7)$. Metrics represented by $\gamma^{(2)}_{(s_1,s_2)}$ is nothing new but the geometric analogy to $\mathbb{B}_8$ metrics in \cite{cvetic_cohomogeneity_2002} and \cite{cvetic_new_2002}, belonging to the strictly larger solution set obtained in \cite{bazaikin_noncompact_2008}. The AC metric $\gamma^{(2)}_{\left(-\frac{3}{\sqrt{10}},\frac{1}{\sqrt{10}}\right)}$ is the geometric analogy to the $\spin(7)$ metric in \cite{bryant_construction_1989} and \cite{gibbons_einstein_1990}.
\end{remark}

By mimicking the example in \cite{cvetic_hyper-kahler_2001}, we are able to describe the AC Calabi HyperK\"ahler metric on $M_{1,1}^{(1)}$ as an algebraic curve. In fact, integral curve that represents the AC metric sits on the boundary of a compact invariant subset of $\mathcal{C}^-_{\spin(7)}$.
\begin{lemma}
The set
$$\tilde{\mathcal{T}}:=\mathcal{C}^-_{\spin(7)}\cap\{Z_2+Z_3-Z_1\geq 0\}\cap\{Z_2Z_4+Z_3Z_4-6\leq 0\}$$
is compact and invariant. Moreover, the boundary of $\tilde{\mathcal{T}}$ where both equality hold is an integral curve.
\end{lemma}
\begin{proof}
Since $Z_1Z_4\leq Z_2Z_4+Z_3Z_4\leq 6$ in $\tilde{\mathcal{T}}$, each one of $Z_1Z_4$, $Z_2Z_4$ and $Z_3Z_4$ is bounded. By \eqref{eqn: new opposite spin(7) conservation}, we know that all $Z_i$'s are bounded. Then all variables are bounded. The compactness is clear.

We have 
\begin{equation}
\begin{split}
&\left.\left\langle\nabla \left(Z_2+Z_3-Z_1\right), V\right\rangle\right|_{Z_2+Z_3-Z_1=0}\\
&=(Z_2+Z_3-Z_1)\mathcal{G}+Z_2(X_2-X_1-X_3)+Z_3(X_3-X_1-X_2)-Z_1(X_1-X_2-X_3)\\
&=2Z_2(X_2-X_1)+2Z_3(X_3-X_1)\quad \text{on replacing $Z_1$ with $Z_2+Z_3$}\\
&=2Z_2\left(2(Z_1-Z_2)-\frac{1}{3}Z_2Z_3Z_4-\frac{1}{6}Z_1Z_3Z_4\right)+2Z_3\left(2(Z_1-Z_3)-\frac{1}{3}Z_2Z_3Z_4-\frac{1}{6}Z_1Z_2Z_4\right)\\
&\quad \text{replace $X_i$'s using \eqref{eqn: new opposite spin(7) equation}}\\
&=2Z_2\left(2Z_3-\frac{1}{3}Z_2Z_3Z_4-\frac{1}{6}Z_1Z_3Z_4\right)+2Z_3\left(2Z_2-\frac{1}{3}Z_2Z_3Z_4-\frac{1}{6}Z_1Z_2Z_4\right)\\
&\quad \text{since $Z_2+Z_3=Z_1$}
\end{split}.
\end{equation}
Since $Z_1Z_4\leq Z_2Z_4+Z_3Z_4\leq 6$ in $\tilde{\mathcal{T}}$, computation above continues as
\begin{equation}
\begin{split}
&\left.\left\langle\nabla \left(Z_2+Z_3-Z_1\right), V\right\rangle\right|_{Z_2+Z_3-Z_1=0}\\
&\geq 2Z_2\left(2Z_3-\frac{1}{3}Z_2Z_3Z_4-Z_3\right)+2Z_3\left(2Z_2-\frac{1}{3}Z_2Z_3Z_4-Z_2\right)\\
&=Z_2Z_3\left(4-\frac{2}{3}Z_4(Z_2+Z_3)\right)\\
&\geq 0
\end{split}.
\end{equation}
On the other hand, we have 
\begin{equation}
\begin{split}
&\left.\left\langle\nabla (Z_2Z_4+Z_3Z_4-6), V\right\rangle\right|_{Z_2Z_4+Z_3Z_4-6=0}\\
&=Z_2Z_4(X_2+X_4-X_1-X_3)+Z_3Z_4(X_3+X_4-X_1-X_2)\\
&=Z_2Z_4\left(3Z_1+3Z_3-Z_2-1+\frac{1}{6}Z_1Z_2Z_4-\frac{1}{3}Z_2Z_3Z_4-\frac{1}{6}Z_1Z_3Z_4\right)\\
&\quad + Z_3Z_4\left(3Z_1+3Z_2-Z_3-1+\frac{1}{6}Z_1Z_3Z_4-\frac{1}{3}Z_2Z_3Z_4-\frac{1}{6}Z_1Z_2Z_4\right)\\
&\quad \text{replace $X_1$, $X_2$ and $X_3$ using \eqref{eqn: new opposite spin(7) equation} and apply Proposition \ref{prop: X_4 and Zi}}\\
&=\frac{6}{(Z_2+Z_3)^2}Z_2\left((3Z_1+3Z_3-Z_2-1)(Z_2+Z_3)+Z_1Z_2-2Z_2Z_3-Z_1Z_3\right)\\
&\quad + \frac{6}{(Z_2+Z_3)^2}Z_3\left((3Z_1+3Z_2-Z_3-1)(Z_2+Z_3)+Z_1Z_3-2Z_2Z_3-Z_1Z_2\right)\\
&\quad \text{since $Z_2Z_4+Z_3Z_4=6$}\\
&=\frac{3}{(Z_2+Z_3)^2}(Z_2+Z_3)\bigg((Z_2+Z_3)(6Z_1+2Z_2+2Z_3-2)-4Z_2Z_3\bigg)\\
&\quad +\frac{3}{(Z_2+Z_3)^2}(Z_2-Z_3)\bigg((Z_2+Z_3)(4Z_3-4Z_2)+2Z_1(Z_2-Z_3)\bigg)\\
&\quad \text{apply identity $ab+cd=\frac{1}{2}(a+c)(b+d)+\frac{1}{2}(a-c)(b-d)$}\\
&=\frac{3}{(Z_2+Z_3)^2}(Z_2+Z_3)\bigg((Z_2+Z_3)(2Z_1+4Z_2+4Z_3-2)+4Z_2Z_3+(Z_2+Z_3)(4Z_1-2Z_2-2Z_3)-8Z_2Z_3\bigg)\\
&\quad +\frac{3}{(Z_2+Z_3)^2}(Z_2-Z_3)\bigg((Z_2+Z_3)(4Z_3-4Z_2)+2Z_1(Z_2-Z_3)\bigg)\\
\end{split}.
\end{equation}
By multiplying both hand side of last equality in \eqref{eqn: new opposite spin(7) conservation} by $2(Z_2+Z_3)$, we know that if $Z_2Z_4+Z_3Z_4=6$,  
\begin{equation}
\begin{split}
2(Z_2+Z_3)&=4(Z_1+Z_2+Z_3)(Z_2+Z_3)+\frac{2}{3}Z_2Z_3Z_4(Z_2+Z_3)-\frac{2}{6}Z_1(Z_2Z_4+Z_3Z_4)(Z_2+Z_3)\\
&=4(Z_1+Z_2+Z_3)(Z_2+Z_3)+4Z_2Z_3-2Z_1(Z_2+Z_3)\\
&=(2Z_1+4Z_2+4Z_3)(Z_2+Z_3)+4Z_2Z_3
\end{split}.
\end{equation}
Hence we have 
\begin{equation}
\begin{split}
&\left.\left\langle\nabla (Z_2Z_4+Z_3Z_4-6), V\right\rangle\right|_{Z_2Z_4+Z_3Z_4-6=0}\\
&=\frac{3}{(Z_2+Z_3)^2}(Z_2+Z_3)\bigg((Z_2+Z_3)(4Z_1-2Z_2-2Z_3)-8Z_2Z_3\bigg)\\
&\quad +\frac{3}{(Z_2+Z_3)^2}(Z_2-Z_3)\bigg((Z_2+Z_3)(4Z_3-4Z_2)+2Z_1(Z_2-Z_3)\bigg)\\
&=-\frac{6}{(Z_2+Z_3)^2}(Z_2+Z_3-Z_1)((Z_2+Z_3)^2+2Z_2^2+2Z_3^2)\\
&\leq 0
\end{split}
\end{equation}

Hence $\tilde{\mathcal{T}}$ is indeed invariant.
\end{proof}
Note that the boundary $$\mathcal{C}^-_{\spin(7)}\cap\{Z_2+Z_3-Z_1= 0\}\cap\{Z_2Z_4+Z_3Z_4-6=0\}$$ is a 1-dimensional invariant set and it contains $P_0^{(1)}$ and
$P_{AC-1}=\left(\frac{1}{7},\frac{1}{7},\frac{1}{7},\frac{1}{7},\frac{2}{7},\frac{1}{7},\frac{1}{7},21\right)$.
It is clear that $\gamma^{(1)}_{\left(-\frac{1}{\sqrt{2}},\frac{1}{\sqrt{2}}\right)}$ is tangent to $\partial\tilde{\mathcal{T}}$ and $\gamma^{(1)}_{(s_1,s_2)}$ is in the interior of $\tilde{\mathcal{T}}$ initially if $s_1>-\frac{1}{\sqrt{2}}$. Hence $\left\{\gamma^{(1)}_{(s_1,s_2)}\mid (s_1,s_2)\in\mathbb{S}^1, s_1\geq -\frac{1}{\sqrt{2}}\right\}$ is a 1-parameter family of $\spin(7)$ metrics on $M_{1,1}^{(1)}$, where $\gamma^{(1)}_{\left(-\frac{1}{\sqrt{2}},\frac{1}{\sqrt{2}}\right)}$ is an AC metric. In summary, we have 
\begin{lemma}
\label{lem: exceptional long existing}
Metrics represented by $\gamma^{(2)}_{(s_1,s_2)}$ on $M_{1,1}^{(2)}$ with $(s_1,s_2)\in\mathbb{S}^1$ and $s_1\geq -\frac{3}{\sqrt{10}}$ are forward complete. Metrics represented by $\gamma^{(1)}_{(s_1,s_2)}$ on $M_{1,1}^{(1)}$ with $(s_1,s_2)\in\mathbb{S}^1$ and $s_1\geq -\frac{1}{\sqrt{2}}$ are forward complete. Moreover, metrics $\gamma_{\left(-\frac{3}{\sqrt{10}},\frac{1}{\sqrt{10}}\right)}^{(2)}$ and $\gamma^{(1)}_{\left(-\frac{1}{\sqrt{2}},\frac{1}{\sqrt{2}}\right)}$ are AC.
\end{lemma}

\begin{figure}[h!] 
\centering
\begin{subfigure}{.4\textwidth}
  \centering 
  \includegraphics[clip,trim=17cm 1cm 3cm 3cm,width=1\linewidth]{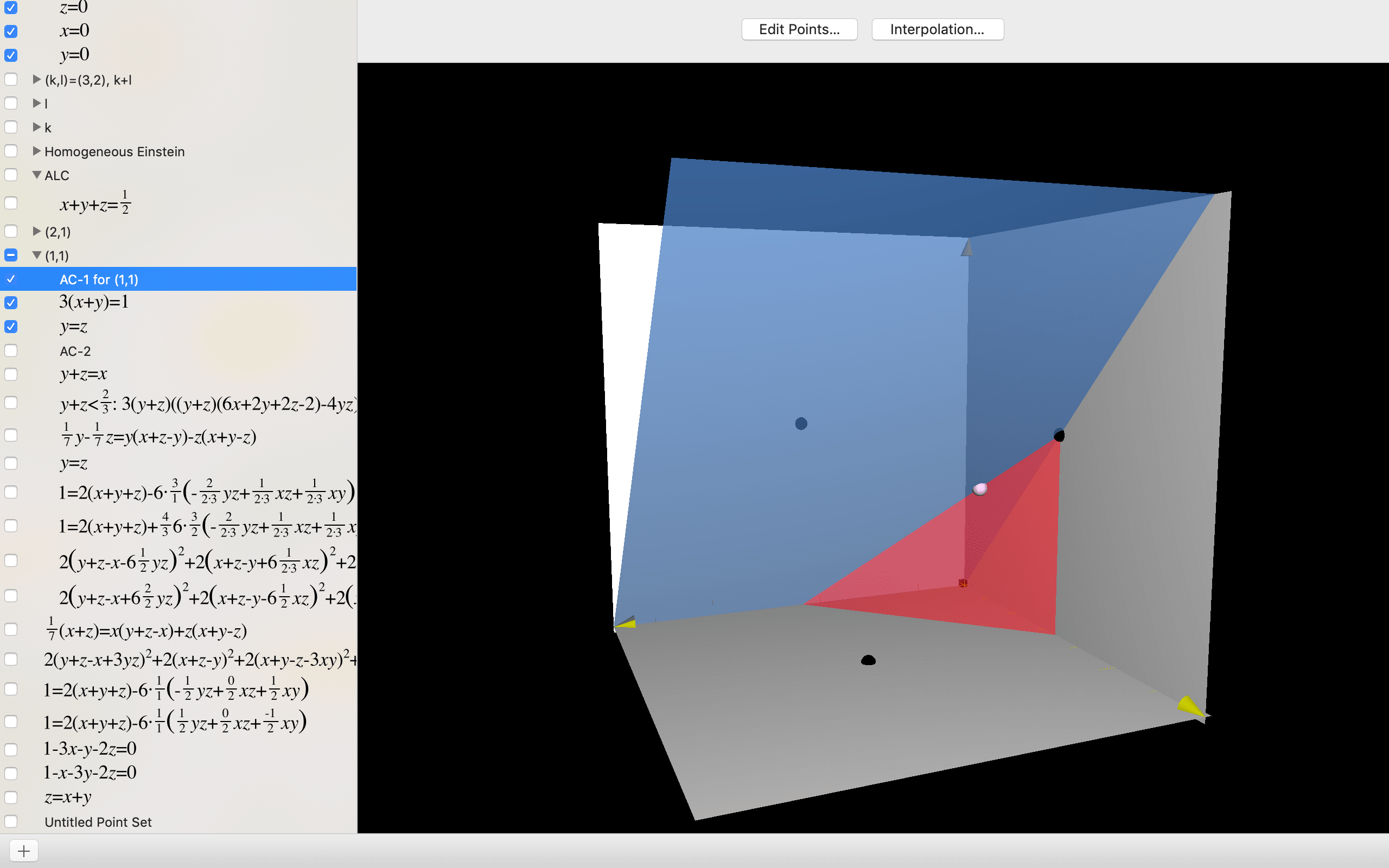}
\caption{$\gamma^{(2)}_{\left(-\frac{3}{\sqrt{10}},\frac{1}{\sqrt{10}}\right)}$}
\end{subfigure}
\begin{subfigure}{.4\textwidth}
  \centering 
  \includegraphics[clip,trim=17cm 1cm 3cm 3cm,width=1\linewidth]{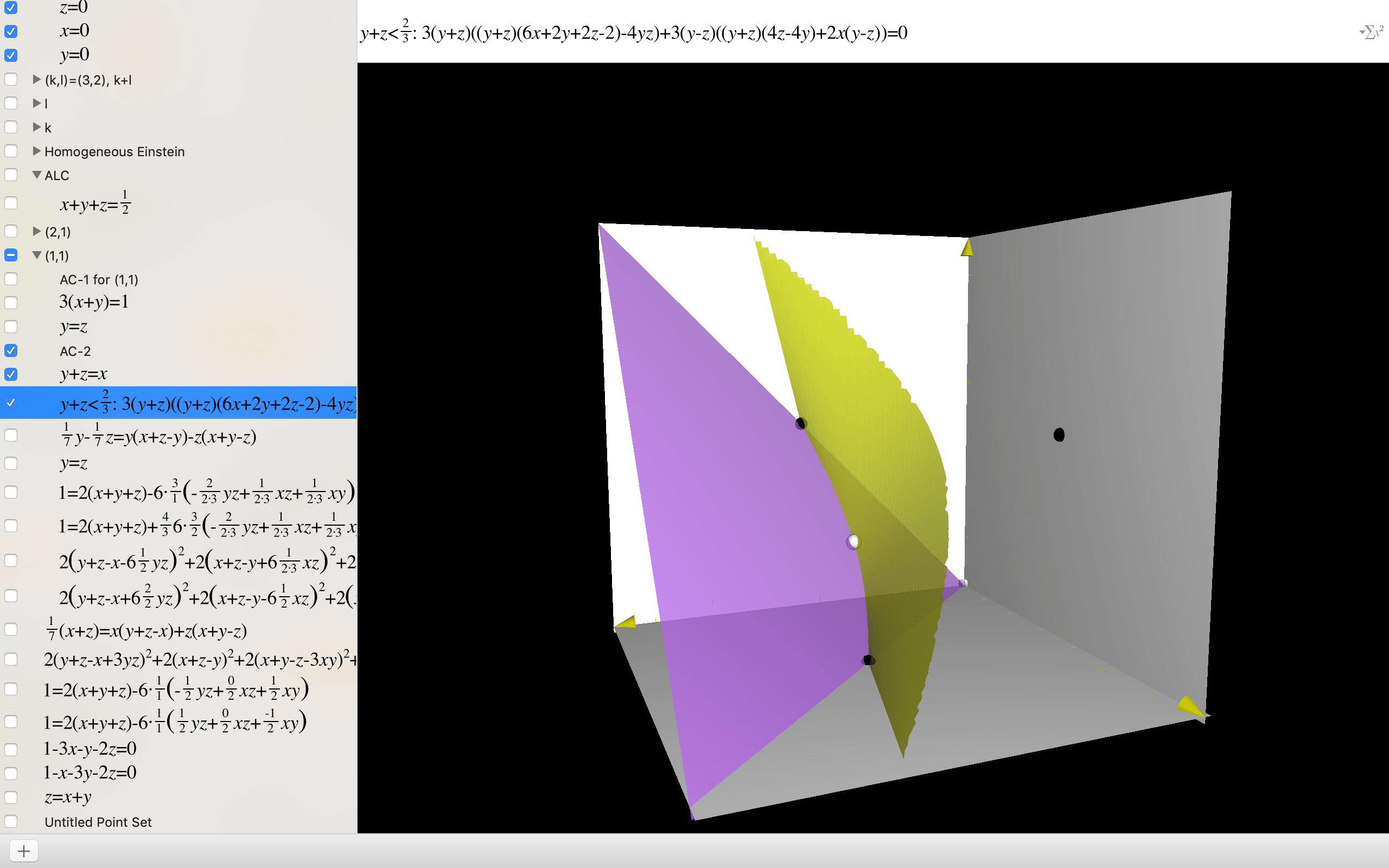}
\caption{$\gamma^{(1)}_{\left(-\frac{1}{\sqrt{2}},\frac{1}{\sqrt{2}}\right)}$}
\end{subfigure}
\label{fig: AC metrics}
\caption{Two AC metrics in $\tilde{\mathcal{S}}$ and $\tilde{\mathcal{T}}$ are in fact represented by algebraic curves. In fact, $\gamma^{(2)}_{\left(-\frac{3}{\sqrt{10}},\frac{1}{\sqrt{10}}\right)}$ lies in a straight line $\{Z_2=Z_3\}\cap\{3Z_1+3Z_2=1\}$}
\end{figure}

By the invariance of $\tilde{\mathcal{S}}$ and $\tilde{\mathcal{T}}$, we also obtain forward complete conically singular metrics with $N_{1,1}$ as principal orbits. 
One can check that for each one of $\mathcal{L}(P_{AC-1})$ and $\mathcal{L}(P_{AC-2})$, there exists a unique unstable eigenvector that is tangent to $\mathcal{C}^-_{\spin(7)}$ and $\mathcal{C}^+_{\spin(7)}$, respectively. Integral curves that emanates from $P_{AC-1}$ and $P_{AC-2}$ along these eigenvectors lie in $\mathcal{C}_{RF}\cap \{X_2=X_3, Z_2=Z_3\}$. Recall Remark \ref{rem: A8}, we know that they are forward complete. In fact, the integral curve that emanates from $P_{AC-1}$ is the geometric analogy to $\mathbb{A}_8$ in \cite{cvetic_cohomogeneity_2002}.

\subsection{Asymptotics}
\label{sec: asymp}
Without further specifying, we use $\gamma$ to denote either $\gamma^{(k+l)}_{(s_1,s_2)}$ or $\gamma^{(k)}_{(s_1,s_2)}$ in Lemma \ref{lem: generic long existing} and Lemma \ref{lem: exceptional long existing} that is defined on $\mathbb{R}$ is this section. We prove the following.
\begin{lemma}
\label{lem: asymp}
For each $\gamma$ in Lemma \ref{lem: generic long existing}, $\lim\limits_{\eta\to\infty}\gamma=P_1$. For each $\gamma$ in Lemma \ref{lem: exceptional long existing}, we have 
$$
\lim_{\eta\to\infty}\gamma_{(s_1,s_2)}^{(2)}=\left\{\begin{array}{ll}
P_1& s_1>-\frac{3}{\sqrt{10}}\\
P_{AC-1}& s_1=-\frac{3}{\sqrt{10}}
\end{array}\right.,\quad \lim_{\eta\to\infty}\gamma_{(s_1,s_2)}^{(1)}=\left\{\begin{array}{ll}
P_1& s_1>-\frac{1}{\sqrt{2}}\\
P_{AC-2}& s_1=-\frac{1}{\sqrt{2}}
\end{array}\right..
$$
\end{lemma}
\begin{proof}
We first consider $\gamma=\gamma^{(k+l)}_{(s_1,s_2)}$ in Lemma \ref{lem: generic long existing}.
By Lemma \ref{lem: invariant set}, we know that $\gamma$ stays in the set where $\mathcal{A}\geq 0$ and $\mathcal{G}-X_4\geq 0$ is held.
Since 
$$\frac{d}{d\eta}Z_4(\gamma(\eta))=\left(-Z_4(\mathcal{G}-X_4)\right)(\eta)\leq 0,$$
we know that the function $Z_4(\gamma)$ is decreasing along $\gamma$, hence it converges to some limit $z\geq 0$ as $\eta \to \infty$. Suppose $z_0>0$, then the $\omega$-limit set of $\gamma$, denoted as $\Omega$, is contained in the level set $\left\{Z_4=z_*\right\}$. By the invariance of the $\omega$-limit set, we know that $\Omega\subset \left\{Z_4=z_*\right\}\cap \{\mathcal{G}-X_4=0\}$. But then by \eqref{eqn: G-X4}, we learn that $\mathcal{A}=0$ for each point in $\Omega$. Moreover, from \eqref{eqn: G-X4} we know that $\mathcal{A}=0$ in $\check{\mathcal{S}}$ implies $Z_1=0$ and $Z_2=Z_3$. Hence in addition to \eqref{eqn: conservation Z spin(7)}, points in $\Omega$ also satisfy 
$$Z_1=0,\quad Z_2=Z_3,\quad Z_2+Z_3\leq \frac{2}{3},\quad \mathcal{A}=0.$$ From all the constraints above, we can conclude that $\Omega=\{P_0^{(k+l)}\}$. By the monotonicity of $Z_4$ along $\gamma$ and the fact that $Z_4<\frac{6\Delta}{k+l}$ initially along $\gamma$, we reach a contradiction. Hence $z_*=0$

Similarly, we have the monotonicity of $Z_4$ along $\gamma=\gamma_{(s_1,s_2)}^{(k)}$. If $Z_4$ does not converges to $0$, then the $\omega$-limit set is a subset of $\mathcal{C}^-_{\spin(7)}\cap \{\mathcal{G}-X_4=0\}$. Then the $\omega$-limit can only be $\{P_0^{(k)}\}$, a contradiction.

Hence $Z_4$ converges to $0$ along each $\gamma$ in Lemma \ref{lem: generic long existing}. By Proposition \ref{prop: X4 not negative}, it is clear that $X_4\geq 0$ along $\gamma$. Then from \eqref{eqn: new spin(7) equation} and \eqref{eqn: new opposite spin(7) equation}, we have 
\begin{equation}
\begin{split}
X_4&=\pm\left(\frac{k+l}{2\Delta}Z_2Z_3Z_4-\frac{l}{2\Delta}Z_1Z_3Z_4-\frac{k}{2\Delta}Z_1Z_2Z_4\right)\\
&\leq\left(\frac{k+l}{2\Delta}Z_2Z_3+\frac{l}{2\Delta}Z_1Z_3+\frac{k}{2\Delta}Z_1Z_2\right)Z_4
\end{split}.
\end{equation}
Since all $Z_i$'s are bounded above, we must have
$$\lim_{\eta\to \infty}Z_4(\gamma_s(\eta))=\lim_{\eta\to \infty}X_4(\gamma(\eta))=0.$$
Hence $\Omega$ is an invariant subset of $\mathcal{C}_{G_2}=\mathcal{C}^{\pm}_{\spin(7)}\cap \{X_4=Z_4=0\}$. The subsystem \eqref{eqn: new spin(7) equation} restricted on $\mathcal{C}_{G_2}$ is essentially the dynamic system of cohomogeneity one $G_2$ metrics with $SU(3)/T^2$ as principal orbit studied in \cite{cleyton_cohomogeneity-one_2002}. Integral curves that represent smooth cohomogeneity one $G_2$ metric in \cite{bryant_construction_1989} and \cite{gibbons_einstein_1990} are included in $\mathcal{C}_{G_2}$. In \cite{chi_invariant_2019}, we show that each integral curve of this subsystem is in fact an algebraic curve in $Z_1$, $Z_2$ and $Z_3$. Moreover, they all converges to $P_1$. By the invariance of $\Omega$, it is clear that $P_1\in \Omega$. Since $P_1$ is a sink, we conclude that $\lim\limits_{\eta\to \infty}\gamma(\eta)=P_1$. 

Consider $\gamma_{(s_1,s_2)}^{(2)}$ in Lemma \ref{lem: exceptional long existing}, the monotonicity of $\sqrt{Z_2Z_3}Z_4$ along $\gamma_{(s_1,s_2)}^{(2)}$ imply that functions $\sqrt{Z_2Z_3}Z_4$ converges. Suppose the limit is nonzero, then by the invariance of the $\omega$-limit set and computation similar to \eqref{eqn: derivative of sqrtZ2Z3Z4}, one can conclude that $\sqrt{Z_2Z_3}Z_4=3$ is satisfied by elements in the $\omega$-limit set. With \eqref{eqn: conservation Z spin(7)} and the constraints $Z_2=Z_3$, we know that $\gamma_{(s_1,s_2)}^{(2)}$ lies in a straight line $\{Z_2=Z_3\}\cap\{3Z_1+3Z_2=1\}$. Hence the limit must be  $P_{AC-1}$. 

If $\sqrt{Z_2Z_3}Z_4<3$ initially, the limit must be zero. Since $Z_3=Z_2\geq Z_1$ in $\tilde{\mathcal{S}}$, we know that the $Z_2$ and $Z_3$ variable must not converge to $0$ due to \eqref{eqn: conservation Z spin(7)}. Hence $Z_4$ converges to zero, then it follows that $X_4$ converges to zero and the limit must be $P_1$. As for $\gamma_{(s_1,s_2)}^{(1)}$ in Lemma \ref{lem: exceptional long existing}, the argument is similar and the statement is proven.
\end{proof}

Theorem \ref{thm: main1} and Theorem \ref{thm: main2} are proven by Lemma \ref{lem: generic long existing}, Lemma \ref{lem: exceptional long existing} and Lemma \ref{lem: asymp}. It is a natural question to ask if $\spin(7)$ metrics on $M_{k,l}^{(i)}$ with generic $N_{k,l}$ can be extended to a larger family where the boundary is given by an AC metric. If such an AC metric exists, can it be described as an algebraic curve in our new coordinate? A further question is whether one can show $\gamma_{(s_1,s_2,s_3)}^{(i)}$ is defined on $\mathbb{R}$, integral curves that represent cohomogeneity one Ricci-flat metric that may not have any special holonomy.

\section{Appendix}
\subsection{Non-negativity of $r_{(\alpha,\beta)}(q_1,q_2)$}
We show that $r_{(\alpha,\beta)}(q_1,q_2)\geq 0$ for any $(\alpha,\beta)\in \left[0,\frac{1}{2}\right]\times (0,1]$. Firstly, we have
\scriptsize
\begin{equation}
\begin{split}
&r(\alpha,\beta)\\
& = 27\beta^2\bigg[\frac{243}{16384} \alpha^8 \beta^6+\left(-\frac{81}{512}\alpha^7+\frac{81}{1024}\alpha^6\right) \beta^5+\left(\frac{81}{256} \alpha^6-\frac{189}{256} \alpha^5+\frac{27}{128}\alpha^4\right) \beta^4\\
&\quad +\left(-\frac{153}{32}  \alpha^5 -\frac{27}{32} \alpha^4+\frac{27}{16} \alpha^3-\frac{9}{32} \alpha^2 \right) \beta^3+\left( 18 \alpha^4-\frac{39}{2} \alpha^3+\frac{159}{16} \alpha^2-\frac{9}{4} \alpha+\frac{3}{16}\right) \beta^2\\
&\quad +\left(21 \alpha^3-15 \alpha^2+\frac{17}{4} \alpha-\frac{7}{16}\right) \beta+6 \alpha^2-2 \alpha+\frac{1}{4}\bigg]\\
& \geq 27\beta^2\bigg[\left(\frac{81}{256} \alpha^6-\frac{189}{256} \alpha^5+\frac{27}{128}\alpha^4\right) \beta^4\\
&\quad +\left(-\frac{153}{32}  \alpha^5 -\frac{27}{32} \alpha^4+\frac{27}{16} \alpha^3-\frac{9}{32} \alpha^2 \right) \beta^3+\left( 18 \alpha^4-\frac{39}{2} \alpha^3+\frac{159}{16} \alpha^2-\frac{9}{4} \alpha+\frac{3}{16}\right) \beta^2\\
&\quad +\left(21 \alpha^3-15 \alpha^2+\frac{17}{4} \alpha-\frac{7}{16}\right) \beta+6 \alpha^2-2 \alpha+\frac{1}{4}\bigg]\quad \text{since $\alpha\in \left[0,\frac{1}{2}\right]$}\\
& \geq 27\beta^2\bigg[\left(-\frac{189}{256} \alpha^5\right) \beta^3\\
&\quad +\left(-\frac{153}{32}  \alpha^5 -\frac{27}{32} \alpha^4+\frac{27}{16} \alpha^3-\frac{9}{32} \alpha^2 \right) \beta^3+\left( 18 \alpha^4-\frac{39}{2} \alpha^3+\frac{159}{16} \alpha^2-\frac{9}{4} \alpha+\frac{3}{16}\right) \beta^2\\
&\quad +\left(21 \alpha^3-15 \alpha^2+\frac{17}{4} \alpha-\frac{7}{16}\right) \beta+6 \alpha^2-2 \alpha+\frac{1}{4}\bigg]\quad \text{since $\beta \in \left[0,1\right]$}\\
& = 27\beta^2\bigg[\left(-\frac{1413}{256}  \alpha^5 -\frac{27}{32} \alpha^4+\frac{27}{16} \alpha^3-\frac{9}{32} \alpha^2 \right) \beta^3+\left( 18 \alpha^4-\frac{39}{2} \alpha^3+\frac{159}{16} \alpha^2-\frac{9}{4} \alpha+\frac{3}{16}\right) \beta^2\\
&\quad +\left(21 \alpha^3-15 \alpha^2+\frac{17}{4} \alpha-\frac{7}{16}\right) \beta+6 \alpha^2-2 \alpha+\frac{1}{4}\bigg]\\
& \geq 27\beta^2\bigg[\left(-\frac{1413}{256}  \alpha^5 -\frac{27}{32} \alpha^4 -\frac{9}{32} \alpha^2 \right) \beta^2+\left( 18 \alpha^4-\frac{39}{2} \alpha^3+\frac{159}{16} \alpha^2-\frac{9}{4} \alpha+\frac{3}{16}\right) \beta^2\\
&\quad +\left(21 \alpha^3-15 \alpha^2+\frac{17}{4} \alpha-\frac{7}{16}\right) \beta+6 \alpha^2-2 \alpha+\frac{1}{4}\bigg] \quad \text{since $\beta \in \left[0,1\right]$}\\
& = 27\beta^2\bigg[\left(-\frac{1413}{256}\alpha^5+\frac{549}{32} \alpha^4-\frac{39}{2} \alpha^3+\frac{309}{32} \alpha^2-\frac{9}{4} \alpha+\frac{3}{16}\right) \beta^2 +\left(21 \alpha^3-15 \alpha^2+\frac{17}{4} \alpha-\frac{7}{16}\right) \beta+6 \alpha^2-2 \alpha+\frac{1}{4}\bigg]\\
& \geq 27\beta^2\bigg[\left(-\frac{1413}{512}\alpha^4+\frac{549}{32} \alpha^4-\frac{39}{2} \alpha^3+\frac{309}{32} \alpha^2-\frac{9}{4} \alpha+\frac{3}{16}\right) \beta^2 +\left(21 \alpha^3-15 \alpha^2+\frac{17}{4} \alpha-\frac{7}{16}\right) \beta+6 \alpha^2-2 \alpha+\frac{1}{4}\bigg]\\
& \quad \text{since $\alpha\in \left[0,\frac{1}{2}\right]$}\\
& = 27\beta^2\bigg[\left(\frac{7371}{512} \alpha^4-\frac{39}{2} \alpha^3+\frac{309}{32} \alpha^2-\frac{9}{4} \alpha+\frac{3}{16}\right) \beta^2 +\left(21 \alpha^3-15 \alpha^2+\frac{17}{4} \alpha-\frac{7}{16}\right) \beta+6 \alpha^2-2 \alpha+\frac{1}{4}\bigg] \\
& \geq  27\beta^2\bigg[\left(-\frac{39}{2} \alpha^3+\frac{309}{32} \alpha^2-\frac{9}{4} \alpha+\frac{3}{16}\right) \beta^2 +\left(21 \alpha^3-15 \alpha^2+\frac{17}{4} \alpha-\frac{7}{16}\right) \beta+6 \alpha^2-2 \alpha+\frac{1}{4}\bigg]\\
& \geq  27\beta^2\bigg[\left(\frac{309}{32} \alpha^2-\frac{9}{4} \alpha+\frac{3}{16}\right) \beta^2 +\left(-\frac{39}{2} \alpha^3+21 \alpha^3-15 \alpha^2+\frac{17}{4} \alpha-\frac{7}{16}\right) \beta+6 \alpha^2-2 \alpha+\frac{1}{4}\bigg] \quad \text{since $\beta \in \left[0,1\right]$}\\
& \geq  27\beta^2\bigg[\left(\frac{309}{32} \alpha^2-\frac{9}{4} \alpha+\frac{3}{16}\right) \beta^2 +\left(-15 \alpha^2+\frac{17}{4} \alpha-\frac{7}{16}\right) \beta+6 \alpha^2-2 \alpha+\frac{1}{4}\bigg]
\end{split}.
\end{equation}
\normalsize
Define function $f_\alpha(\beta)=c_2(\alpha)\beta^2+c_1\beta+c_0(\alpha)$, where 
$$
c_2(\alpha)=\frac{309}{32} \alpha^2-\frac{9}{4} \alpha+\frac{3}{16},\quad 
c_1(\alpha)=-15 \alpha^2+\frac{17}{4} \alpha-\frac{7}{16},\quad 
c_0(\alpha)=6 \alpha^2-2 \alpha+\frac{1}{4}.
$$
It is clear that $c_2,c_0>0$ and $c_1<0$ for any $\alpha\in \left[0,\frac{1}{2}\right]$. It is also clear that $-\frac{c_1}{2c_2}\geq 1$ for $\alpha\in \left[0,\frac{2+\sqrt{73}}{69}\right]$. 
Consider 
$$
\Delta(\alpha)=(c_1^2-4c_2c_0)(\alpha)=-\frac{27}{4}\alpha^4+\frac{15}{4} \alpha^3-\frac{31}{32}\alpha^2+\frac{1}{32}\alpha+\frac{1}{256}.
$$
Since
$$
\Delta''(\alpha)=-81\alpha^2+\frac{45}{2}\alpha-\frac{31}{16}<0, \quad \Delta'\left(\frac{3}{20}\right)=-\frac{779}{8000}<0,\quad \Delta\left(\frac{3}{20}\right)=-\frac{2537}{640000}<0,
$$
it is clear that
$\Delta(\alpha)<0$ any $\alpha\in \left[\frac{3}{20}, \frac{1}{2}\right]$. Since $\frac{2+\sqrt{73}}{69}\approx 0.1528>0.15=\frac{3}{20}$, we have 
$$
\begin{array}{ll}
f_\alpha(\beta)\geq f_\alpha(1)=\frac{21}{32}\alpha^2\geq 0 &\text{for $\alpha\in \left[0,\frac{2+\sqrt{73}}{69}\right]$}\\
f_\alpha(\beta)>0 &\text{for $\alpha\in \left[\frac{2+\sqrt{73}}{69},\frac{1}{2}\right]$}
\end{array}.
$$
Hence $r(\alpha,\beta)\geq 0$ for any  $(\alpha,\beta)\in \left[0,\frac{1}{2}\right]\times (0,1]$ and only vanishes at $(0,1)$. 

\subsection{Formula of $r_{(\alpha,\beta,\delta)}(p_1,p_2)$}
We present the formular for $r_{(\alpha,\beta,\delta)}(p_1,p_2)$ for any $(\alpha,\beta,\delta)\in [0,1]\times[0,1]\times (0,1]$. Let $\rho=\frac{l}{k}$, we have
\tiny
\begin{equation}
\begin{split}
r_{(\alpha,\beta,\delta)}(p_1,p_2)&=36 \delta^2 \bigg(36 \rho^4 \alpha^4 \beta^4 \delta^2-72 \rho^4 \alpha^3 \beta^4 \delta^2+144 \rho^3 \alpha^4 \beta^4 \delta^2+108 \rho^4 \alpha^2 \beta^4 \delta^2+45 \rho^3 \alpha^5 \beta^3 \delta-6 \rho^3 \alpha^4 \beta^4 \delta-72 \rho^3 \alpha^4 \beta^3 \delta^2\\
&\quad +45 \rho^3 \alpha^3 \beta^5 \delta-216 \rho^3 \alpha^3 \beta^4 \delta^2+216 \rho^2 \alpha^4 \beta^4 \delta^2-72 \rho^4 \alpha \beta^4 \delta^2-69 \rho^3 \alpha^4 \beta^3 \delta+60 \rho^3 \alpha^3 \beta^4 \delta+144 \rho^3 \alpha^3 \beta^3 \delta^2\\
&\quad -63 \rho^3 \alpha^2 \beta^5 \delta+216 \rho^3 \alpha^2 \beta^4 \delta^2+135 \rho^2 \alpha^5 \beta^3 \delta-18 \rho^2 \alpha^4 \beta^4 \delta-216 \rho^2 \alpha^4 \beta^3 \delta^2+135 \rho^2 \alpha^3 \beta^5 \delta-216 \rho^2 \alpha^3 \beta^4 \delta^2\\
&\quad +144 \rho \alpha^4 \beta^4 \delta^2+36 \rho^4 \beta^4 \delta^2+174 \rho^3 \alpha^3 \beta^3 \delta-144 \rho^3 \alpha^2 \beta^3 \delta^2+63 \rho^3 \alpha \beta^5 \delta-72 \rho^3 \alpha \beta^4 \delta^2+15 \rho^2 \alpha^6 \beta^2-4 \rho^2 \alpha^5 \beta^3\\
&\quad -63 \rho^2 \alpha^5 \beta^2 \delta+26 \rho^2 \alpha^4 \beta^4-78 \rho^2 \alpha^4 \beta^3 \delta+108 \rho^2 \alpha^4 \beta^2 \delta^2-4 \rho^2 \alpha^3 \beta^5+51 \rho^2 \alpha^3 \beta^4 \delta+288 \rho^2 \alpha^3 \beta^3 \delta^2+15 \rho^2 \alpha^2 \beta^6\\
&\quad -126 \rho^2 \alpha^2 \beta^5 \delta+108 \rho^2 \alpha^2 \beta^4 \delta^2+135 \rho \alpha^5 \beta^3 \delta-18 \rho \alpha^4 \beta^4 \delta -216 \rho \alpha^4 \beta^3 \delta^2+135 \rho \alpha^3 \beta^5 \delta-72 \rho \alpha^3 \beta^4 \delta^2\\
&\quad +36 \alpha^4 \beta^4 \delta^2-174 \rho^3 \alpha^2 \beta^3 \delta-60 \rho^3 \alpha \beta^4 \delta+72 \rho^3 \alpha \beta^3 \delta^2-45 \rho^3 \beta^5 \delta-10 \rho^2 \alpha^5 \beta^2+28 \rho^2 \alpha^4 \beta^3+120 \rho^2 \alpha^4 \beta^2 \delta\\
&\quad -48 \rho^2 \alpha^3 \beta^4+216 \rho^2 \alpha^3 \beta^3 \delta-144 \rho^2 \alpha^3 \beta^2 \delta^2+36 \rho^2 \alpha^2 \beta^5+120 \rho^2 \alpha^2 \beta^4 \delta-144 \rho^2 \alpha^2 \beta^3 \delta^2-6 \rho^2 \alpha \beta^6+63 \rho^2 \alpha \beta^5 \delta\\
&\quad +30 \rho \alpha^6 \beta^2-8 \rho \alpha^5 \beta^3-126 \rho \alpha^5 \beta^2 \delta+52 \rho \alpha^4 \beta^4+51 \rho \alpha^4 \beta^3 \delta+216 \rho \alpha^4 \beta^2 \delta^2-8 \rho \alpha^3 \beta^5-78 \rho \alpha^3 \beta^4 \delta+144 \rho \alpha^3 \beta^3 \delta^2\\
&\quad +30 \rho \alpha^2 \beta^6-63 \rho \alpha^2 \beta^5 \delta+45 \alpha^5 \beta^3 \delta-6 \alpha^4 \beta^4 \delta-72 \alpha^4 \beta^3 \delta^2+45 \alpha^3 \beta^5 \delta+69 \rho^3 \alpha \beta^3 \delta+6 \rho^3 \beta^4 \delta+81 \rho^2 \alpha^4 \beta^2\\
&\quad -24 \rho^2 \alpha^3 \beta^3-306 \rho^2 \alpha^3 \beta^2 \delta+44 \rho^2 \alpha^2 \beta^4-306 \rho^2 \alpha^2 \beta^3 \delta+108 \rho^2 \alpha^2 \beta^2 \delta^2+36 \rho^2 \alpha \beta^5-129 \rho^2 \alpha \beta^4 \delta+15 \rho^2 \beta^6\\
&\quad -6 \rho \alpha^6 \beta+26 \rho \alpha^5 \beta^2+63 \rho \alpha^5 \beta \delta-20 \rho \alpha^4 \beta^3+120 \rho \alpha^4 \beta^2 \delta-72 \rho \alpha^4 \beta \delta^2-20 \rho \alpha^3 \beta^4+216 \rho \alpha^3 \beta^3 \delta-144 \rho \alpha^3 \beta^2 \delta^2\\
&\quad +26 \rho \alpha^2 \beta^5+120 \rho \alpha^2 \beta^4 \delta-6 \rho \alpha \beta^6+15 \alpha^6 \beta^2-4 \alpha^5 \beta^3-63 \alpha^5 \beta^2 \delta+26 \alpha^4 \beta^4+60 \alpha^4 \beta^3 \delta+108 \alpha^4 \beta^2 \delta^2-4 \alpha^3 \beta^5\\
&\quad -69 \alpha^3 \beta^4 \delta+15 \alpha^2 \beta^6-45 \rho^3 \beta^3 \delta-44 \rho^2 \alpha^3 \beta^2-24 \rho^2 \alpha^2 \beta^3+120 \rho^2 \alpha^2 \beta^2 \delta-48 \rho^2 \alpha \beta^4+129 \rho^2 \alpha \beta^3 \delta-4 \rho^2 \beta^5\\
&\quad +46 \rho \alpha^5 \beta+44 \rho \alpha^4 \beta^2-129 \rho \alpha^4 \beta \delta-4 \rho \alpha^3 \beta^3-306 \rho \alpha^3 \beta^2 \delta+72 \rho \alpha^3 \beta \delta^2+44 \rho \alpha^2 \beta^4-306 \rho \alpha^2 \beta^3 \delta+46 \rho \alpha \beta^5\\
&\quad -6 \alpha^6 \beta+36 \alpha^5 \beta^2+63 \alpha^5 \beta \delta-48 \alpha^4 \beta^3-72 \alpha^4 \beta \delta^2+28 \alpha^3 \beta^4+174 \alpha^3 \beta^3 \delta-10 \alpha^2 \beta^5+81 \rho^2 \alpha^2 \beta^2+28 \rho^2 \alpha \beta^3\\
&\quad -63 \rho^2 \alpha \beta^2 \delta+26 \rho^2 \beta^4-76 \rho \alpha^4 \beta-44 \rho \alpha^3 \beta^2+129 \rho \alpha^3 \beta \delta-44 \rho \alpha^2 \beta^3+120 \rho \alpha^2 \beta^2 \delta-76 \rho \alpha \beta^4+15 \alpha^6+36 \alpha^5 \beta\\
&\quad -45 \alpha^5 \delta+44 \alpha^4 \beta^2-60 \alpha^4 \beta \delta+36 \alpha^4 \delta^2-24 \alpha^3 \beta^3-174 \alpha^3 \beta^2 \delta+81 \alpha^2 \beta^4-10 \rho^2 \alpha \beta^2-4 \rho^2 \beta^3+76 \rho \alpha^3 \beta\\
&\quad +118 \rho \alpha^2 \beta^2 -63 \rho \alpha^2 \beta \delta+76 \rho \alpha \beta^3-4 \alpha^5-48 \alpha^4 \beta+6 \alpha^4 \delta-24 \alpha^3 \beta^2+69 \alpha^3 \beta \delta-44 \alpha^2 \beta^3+15 \rho^2 \beta^2\\
&\quad -46 \rho \alpha^2 \beta-46 \rho \alpha \beta^2+26 \alpha^4+28 \alpha^3 \beta-45 \alpha^3 \delta+81 \alpha^2 \beta^2+6 \rho \alpha \beta-4 \alpha^3-10 \alpha^2 \beta+15 \alpha^2\bigg).\\
&=36\delta^2 \tilde{r}_{(\alpha,\beta,\delta)}(p_1,p_2)
\end{split}
\end{equation}
\normalsize
With the help of Maple, one can conclude that $\tilde{r}_{(\alpha,\beta,\delta)}(p_1,p_2)\geq 0$ for any $(\alpha,\beta,\delta)\in [0,1]\times[0,1]\times (0,1]$ and it only vanishes at $(1,0,1)$ if $(k,l)\neq (1,1)$. If $(k,l)=(1,1)$, the function vanishes at $(1,0,1)$ and $(0,1,1)$. The level set of $\tilde{r}_{(\alpha,\beta,\delta)}(p_1,p_2)=0$ with some selected $(k,l)$ are presented below.
\begin{figure}[h!] 
\centering
\begin{subfigure}{.3\textwidth}
  \centering 
  \includegraphics[clip,trim=12cm 1cm 3cm 3cm,width=1\linewidth]{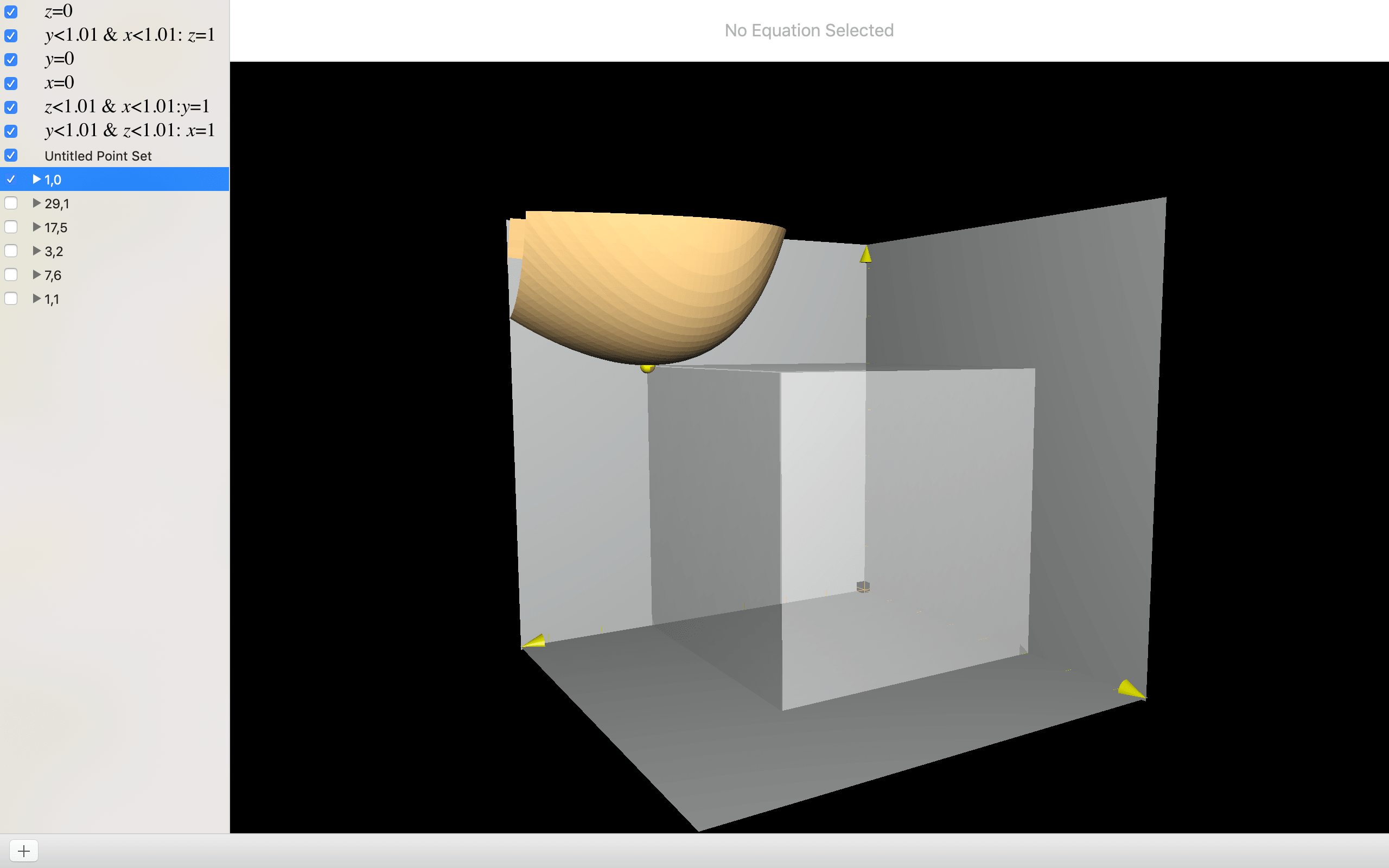}
\caption{$(k,l)=(1,0)$}
\end{subfigure}
\begin{subfigure}{.3\textwidth}
  \centering 
  \includegraphics[clip,trim=12cm 1cm 3cm 3cm,width=1\linewidth]{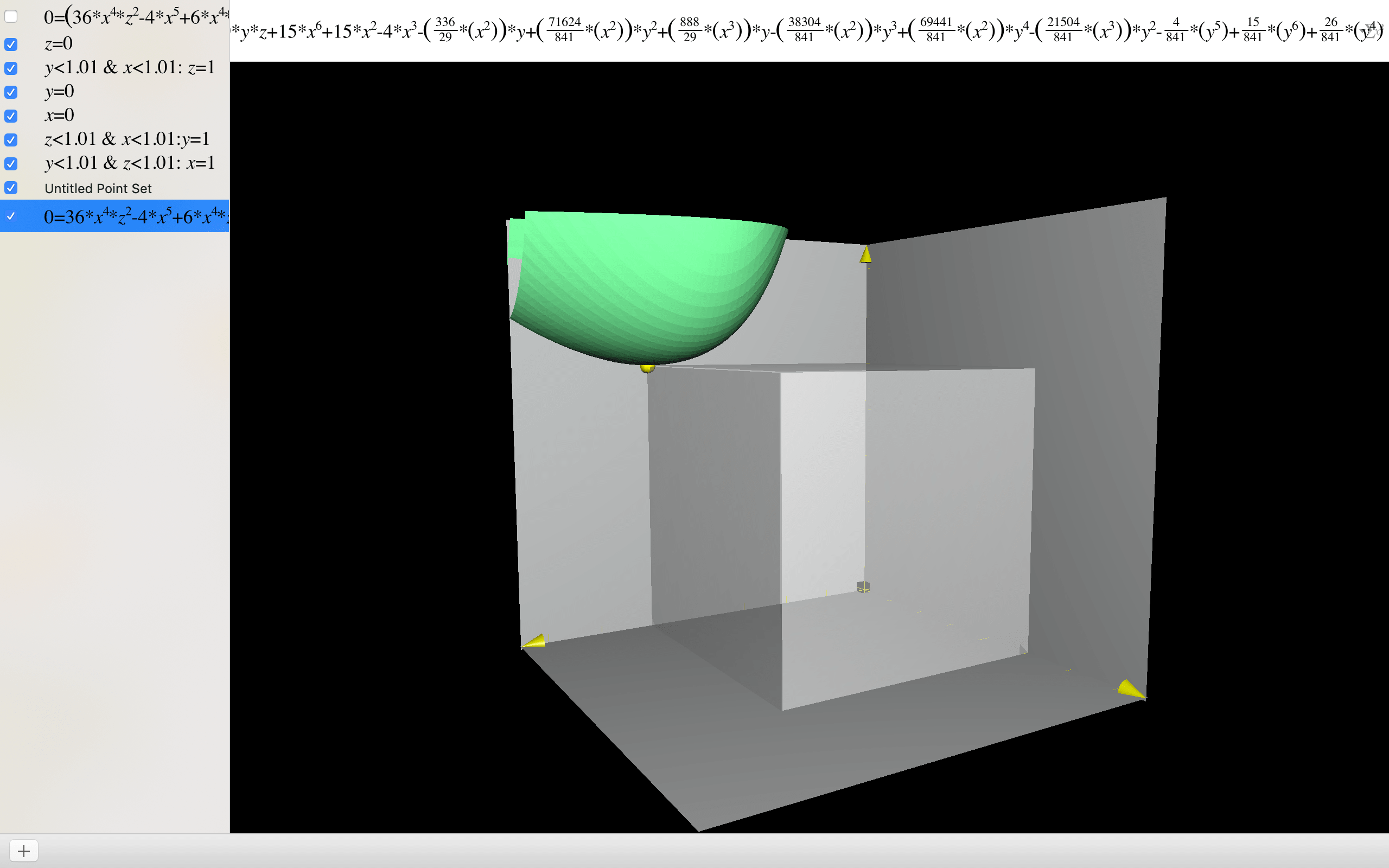}
\caption{$(k,l)=(29,1)$}
\end{subfigure}
\begin{subfigure}{.3\textwidth}
  \centering
  \includegraphics[clip,trim=12cm 1cm 3cm 3cm,width=1\linewidth]{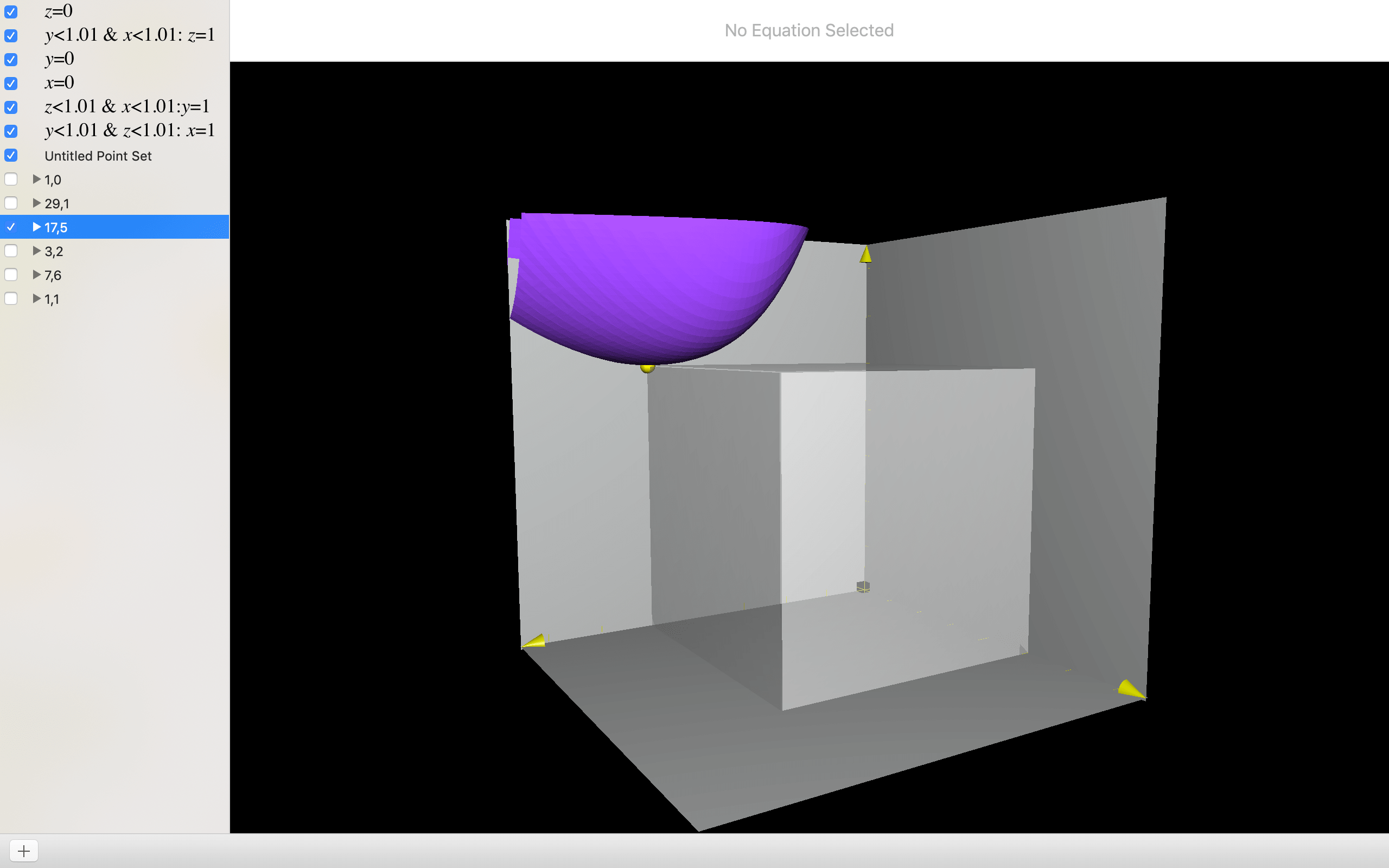}
\caption{$(k,l)=(17,5)$}
\end{subfigure}
\begin{subfigure}{.3\textwidth}
  \centering
  \includegraphics[clip,trim=12cm 1cm 3cm 3cm,width=1\linewidth]{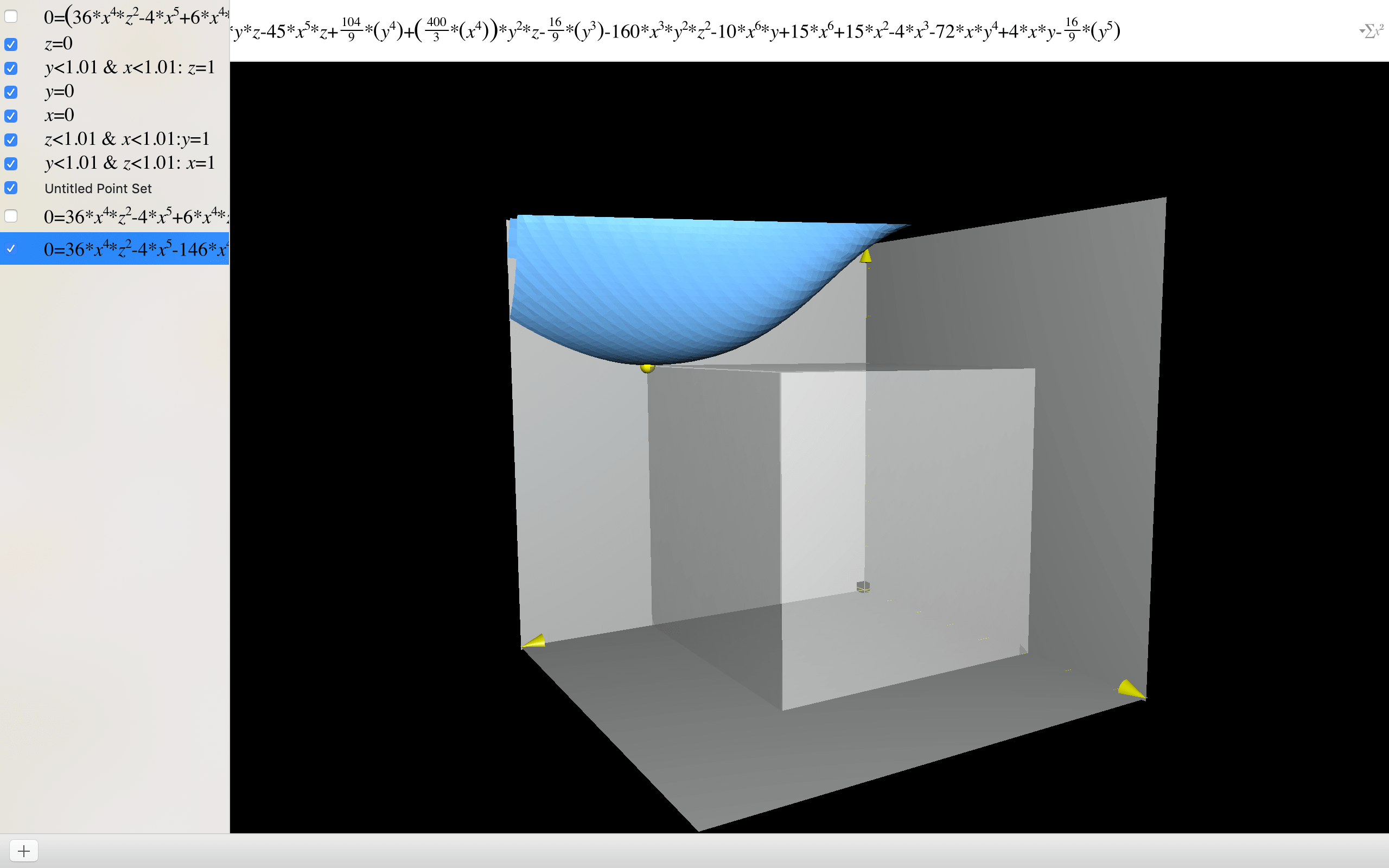}
\caption{$(k,l)=(3,2)$}
\end{subfigure}
\begin{subfigure}{.3\textwidth}
  \centering
  \includegraphics[clip,trim=12cm 1cm 3cm 3cm,width=1\linewidth]{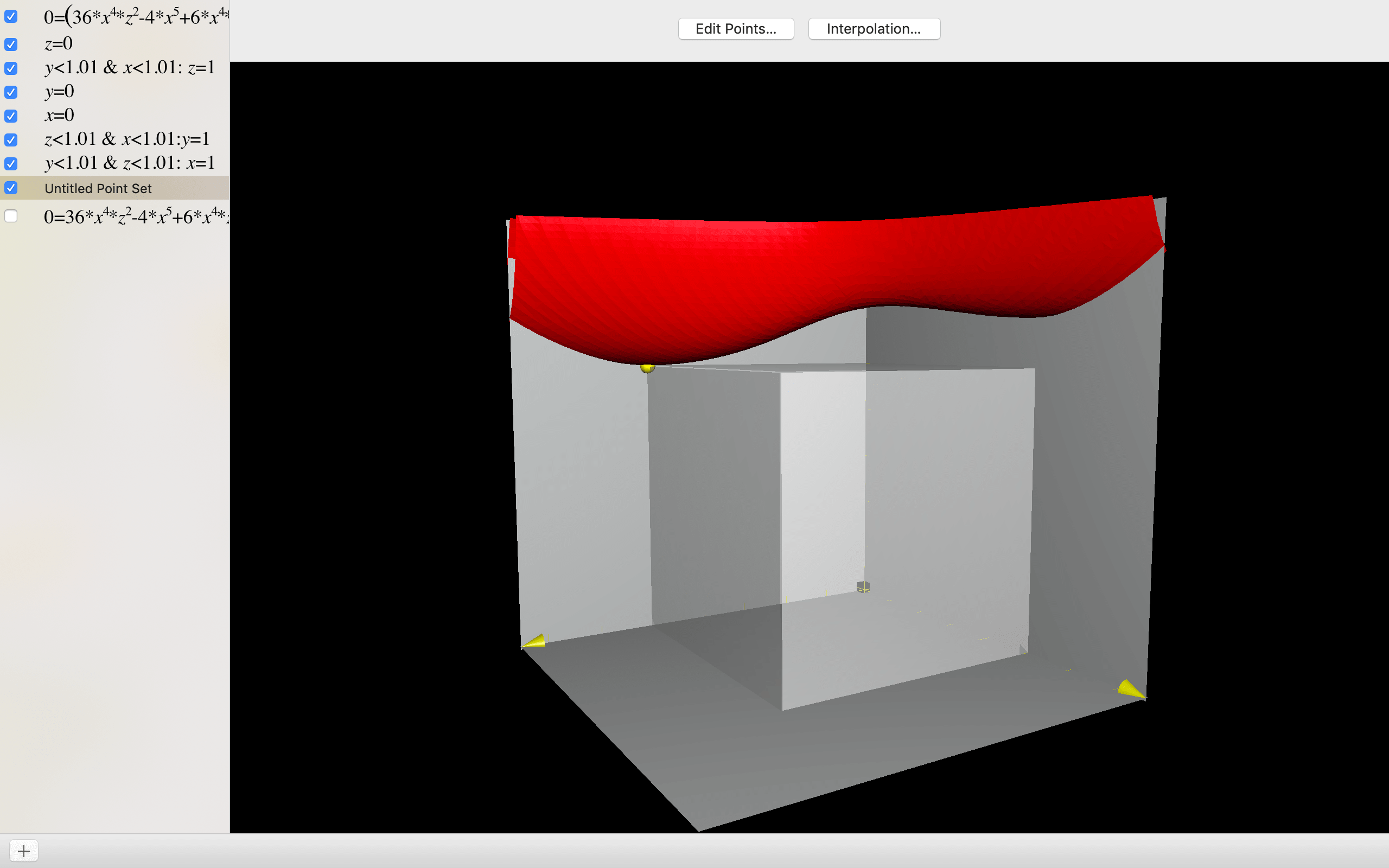}
\caption{$(k,l)=(7,6)$}
\end{subfigure}
\begin{subfigure}{.3\textwidth}
  \centering
  \includegraphics[clip,trim=12cm 1cm 3cm 3cm,width=1\linewidth]{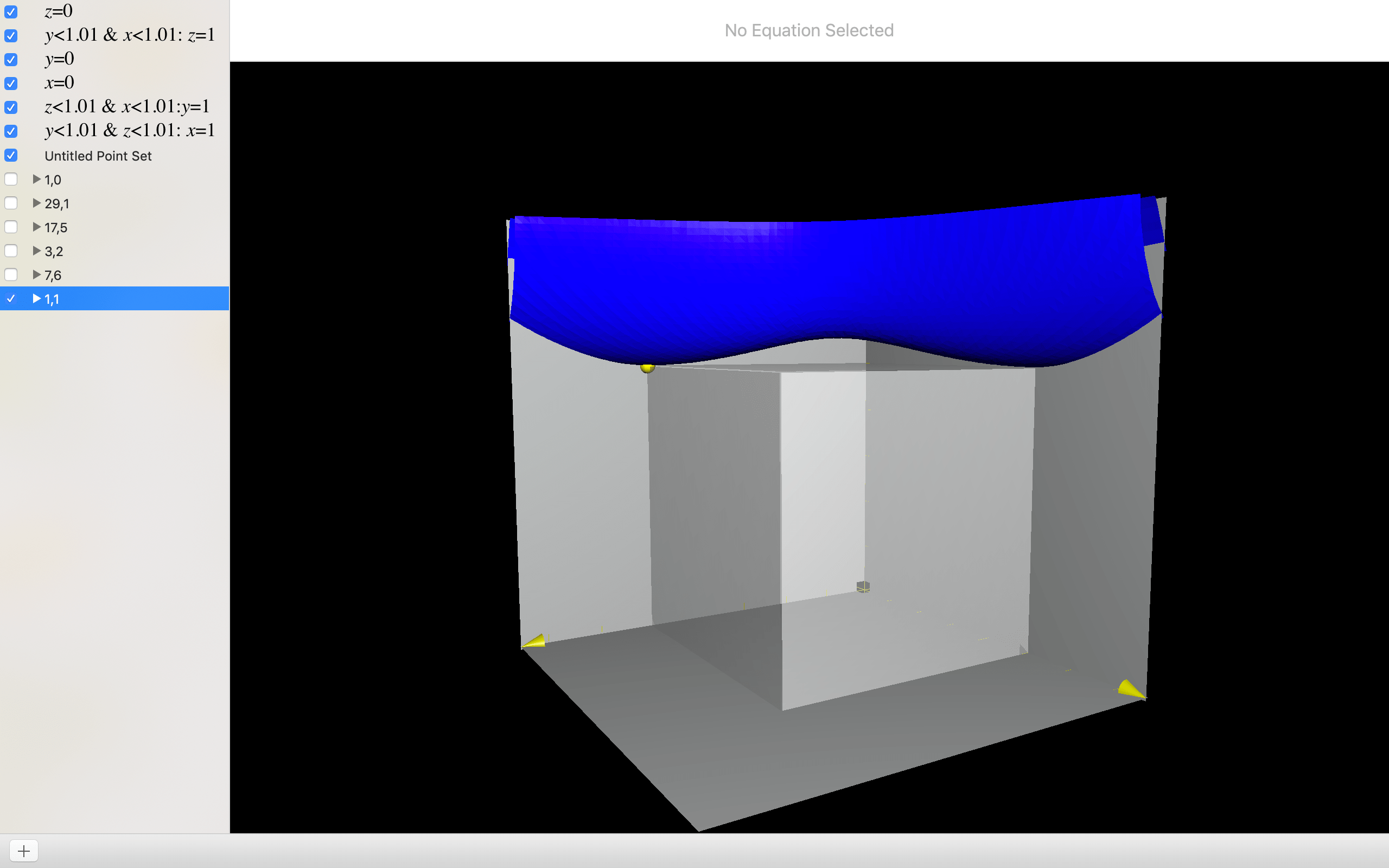}
\caption{$(k,l)=(1,1)$}
\end{subfigure}
\label{fig: resultant-2}
\caption{Figures above are the zero level set of $\tilde{r}_{(\alpha,\beta,\delta)}(p_1,p_2)$, arranged in the increasing order of $\frac{l}{k}$. The cube denote the set $[0,1]\times[0,1]\times (0,1]$ and the dot is $(1,0,1)$. Note that when $(k,l)=(1,1)$, $\tilde{r}_{(\alpha,\beta,\delta)}(p_1,p_2)$ also vanishes at $(0,1,1)$. Such a phenomenon is not surprising since \eqref{eqn: new opposite spin(7) equation} with $(k,l)=(1,1)$ is symmetric with respect to $Z_2$ and $Z_3$.}
\end{figure}
\bibliography{Bibliography}
\bibliographystyle{alpha}
\end{document}